\documentclass[12pt]{article}
 \usepackage{mathrsfs} 
 \usepackage{dsfont}
\usepackage[english,activeacute]{babel}

\usepackage{amssymb,amsmath,amsthm}

\usepackage{enumitem}

\usepackage{hyperref}

\newtheorem{theorem}{Theorem} [section]
\newtheorem{proposition}[theorem]{Proposition}

\newtheorem{lemma}[theorem]{Lemma}

\newtheorem{corollary}[theorem]{Corollary}

\newtheorem{example}[theorem]{Example}

\newcommand{\NN}{{\mathbb N}}
\def\N{{d}}
\newcommand{\Z}{{\mathbb Z}}

\newcommand{\R}{{\mathbb R}}

\def\d{{\rm d}}
\def\e{{\rm e}}
\def\loc{{\rm loc}}
\def\<{\langle}
\def\>{\rangle}
\def\ri{{\rm i}}
\def\supp{{\rm supp}}
\def\:{{\colon}}
\def\M{\mathcal{M}}
\def\BTV{{\rm BTV}}

\def\be#1{\begin{equation}\label{#1}}
\def\ee{\end{equation}}

\newcommand{\vvert}[1]{{\left\vert\kern-0.25ex\left\vert\kern-0.25ex\left\vert #1
    \right\vert\kern-0.25ex\right\vert\kern-0.25ex\right\vert}}

\newcommand{\Norm}{\vvert}

\newcommand{\D}{\displaystyle}

\newcommand{\Chi}{{\cal X}}
\newcommand{\eps}{\varepsilon}

\title{Optimal existence classes  and nonlinear--like dynamics in the
linear   heat equation in $\R^{\N}$}

\author{James C. Robinson${}^1$, \\
An\'{\i}bal Rodr\'{\i}guez-Bernal ${}^{2,}$ \thanks{Partially supported
    by Project  MTM2016-75465, MICINN and   GR58/08
  Grupo 920894, UCM, Spain}
\thanks{Partially supported by PRX17/00522 Programa Salvador de
  Madariaga MECyD, Spain, and by an EPSRC grant EP/R023778/1.}}

\date{\today}

\begin{document}

\maketitle

\setcounter{footnote}{2}

\makeatletter
\begin{center}
${}^{1}$Mathematics Institute\\
University of Warwick  \\  Gibbet Hill Rd, Coventry CV4 7AL, UK\\ {E-mail:
j.c.robinson@warwick.ac.uk}
\\ \mbox{}
\\
${}^{2}$Departamento de Matem\'atica Aplicada\\ Universidad
  Complutense de Madrid\\ 28040 Madrid, Spain\\ and \\
  Instituto de Ciencias Matem\'aticas \\
CSIC-UAM-UC3M-UCM \footnote{${}^{*}$Partially supported by ICMAT Severo Ochoa
  project SEV-2015-0554 (MINECO)}
 \\ {E-mail:
arober@mat.ucm.es}
\end{center}
\makeatother

\begin{abstract}
  We analyse the behaviour of solutions of the linear heat equation in
  $\R^d$ for initial data in the classes $\M_\eps(\R^{\N})$ of Radon measures
  with $\int_{\R^\N}\e^{-\eps|x|^2}\,\d |u_0|<\infty$. We show that these classes
  are in some sense optimal for local and global existence of non-negative solutions: in particular
  $\M_0(\R^{\N})=\cap_{\eps>0}\M_\eps(\R^{\N})$ consists precisely of those initial data for
  which the a solution of the heat equation can be given for all time using the heat kernel representation formula. After
  considering properties of existence, uniqueness, and regularity for
  such initial data, which can grow rapidly at infinity, we go on to
  show that they give rise to properties associated more often with
  nonlinear models. We demonstrate the finite-time blowup of
  solutions, showing that the set of blowup points is the complement of a convex set, and that given any closed convex set there is an initial condition whose solutions remain bounded precisely on this set at the `blowup time'. We also show that wild oscillations are possible from
  non-negative initial data as $t\to\infty$ (in fact we show that this behaviour is generic), and that one can   prescribe the behaviour of $u(0,t)$ to be any real-analytic function $\gamma(t)$ on $[0,\infty)$.
\end{abstract}

\section{ Introduction}
\setcounter{equation}{0}

In this paper we consider the linear heat equation posed on the whole
space $\R^\N$, with very general initial data, which may be either
only locally integrable or even a Radon measure. For an appropriate
class of initial data $u_0$, see e.g.\ \cite{Tychonov}, it is well known that solutions to this equation,
\begin{equation} \label{eq:intro_heat_equation}
u_t-\Delta u=0,\ x\in\R^\N,\ t>0,\qquad u(x,0)=u_0(x),
\end{equation}
can be written using the heat kernel as
\begin{equation} \label{eq:intro_solution_heat_up_t=0}
u(x,t)=S(t)u_{0} (x) := \frac{1}{(4\pi t)^{\N/2}} \int_{\R^\N} \e^{-|x-y|^2/4t}
u_0(y) \, \d y  , \quad x\in \R^\N, \ t>0 .
\end{equation}

It turns out that the behaviour of solutions in
(\ref{eq:intro_solution_heat_up_t=0}) is significantly affected by the way
the mass of the initial data is distributed in space.

If the mass as $|x|\to \infty$ is not too large it is well known that
 the `mass' of the initial data moves
to infinity and the solutions decay to zero in suitable norms.
For example, if $u_{0} \in L^{p}(\R^\N)$ for some
$1\leq p< \infty$ then classical estimates ensure that
\begin{equation} \label{eq:heat_LpLq_estimates}
  \|u(t)\|_{L^{q}(\R^\N)} \leq  (4\pi t)^{-\frac{\N}{2}(\frac{1}{p}-\frac{1}{q})}
    \|u_{0}\|_{L^{p}(\R^\N)},\quad\mbox{for every } t>0 \mbox{ and }q\mbox{ with } p\leq q \leq \infty,
\end{equation}
which in particular implies  that all solutions converge uniformly to zero on the whole
of $\R^\N$. 
In particular,  for $u_{0} \in L^{1}(\R^\N)$ since we also have
\begin{displaymath}
  \int_{\R^{\N}} u(x,t)\, \d x = \int_{\R^\N} u_{0}(y)\,  \d y, \quad
  t>0,
\end{displaymath}
it follows that for such $u_{0}$ the total
mass is preserved but (from  (\ref{eq:heat_LpLq_estimates})) the
supremum tends to zero, i.e.\ the mass moves to infinity.

It is also known that as $t\to \infty$,  solutions asymptotically
resemble the heat kernel
$$K(x,t)=(4\pi t)^{-\N/2}
\e^{-|x|^2/4t},$$
see for example Section 1.1.4 in
\cite{GigaGigaSaal}. The faster the initial data decays as $|x|\to \infty$ the higher the order of the asymptotics of the solution that are described by
the heat kernel, see e.g.\ \cite{DuoanZuazua}.

When the initial data is bounded, $u_{0} \in L^{\infty}(\R^{\N})$, the
decay described above does not necessarily take place. In fact
(\ref{eq:heat_LpLq_estimates}) reduces to
\begin{displaymath}
  \|u(t)\|_{L^{\infty}(\R^\N)} \leq      \|u_{0}\|_{L^{\infty}(\R^\N)}, \quad t>0,
\end{displaymath}
which does not in general imply any decay. 
For example, if $u_{0}\equiv1$ then $u(x,t)=1$ for every $t>1$; for any $R>0$ we can write
\begin{displaymath}
  1 = u(t, \Chi_{B(0,R)}) + u(t, \Chi_{\R^{\N}\setminus B(0,R)} ),
\end{displaymath}
where $\Chi_A$ denotes the characteristic function of the set $A$. Since $\Chi_{B(0,R)} \in  L^{1}(\R^\N)\cap  L^{\infty}(\R^\N)$
the mass of $0\leq u(t, \Chi_{B(0,R)}) $ escapes to infinity but, on
the other hand, the mass of $\Chi_{\R^{\N}\setminus B(0,R)}$, diffused
by $u(t, \Chi_{\R^{\N}\setminus B(0,R)} )$,
moves `inwards' from infinity and both balance precisely at every
time.

Hence, it turns out the dynamics of the solutions
(\ref{eq:intro_solution_heat_up_t=0}) of the heat equation
(\ref{eq:intro_heat_equation}) for bounded initial data is much richer
than for initial data with small mass at infinity. For example, the existence of one-dimensional bounded oscillations was proved in  Section 8 in
\cite{ColletEckman}, while bounded `wild' oscillations in any
dimensions were shown to exist in  \cite{VZ2002} by a scaling method. It is
worth noting that this scaling argument is also applied  in  \cite{VZ2002} to
some nonlinear equations  (porous medium, $p$-Laplacian, and scalar
conservation laws). Indeed, this scaling argument allows one to show that
for $L^{p}(\R^{\N})$ initial data, $1\leq p < \infty$, the solution of
(\ref{eq:intro_heat_equation}) asymptotically approaches the heat
kernel. The scaling argument was later extended to some nonlinear
dissipative reaction diffusion equations in  \cite{CazenaveDW}.

In this paper our goal is to consider some (optimal) classes of
unbounded data that possess large mass at infinity. In such a
situation we show how the mechanism of mass moving inwards from
infinity plays a dominant role on the structure and properties of
solutions of (\ref{eq:intro_heat_equation}). It turns out that in
this setting, solutions of (\ref{eq:intro_heat_equation}) show
surprising dynamical behaviours more akin to what is expected in nonlinear
equations.

For example, in our class of `large' initial data finite-time blowup is possible. We completely characterise (non-negative) initial data
for which the solution ceases to exist in some finite time; we
determine the maximal existence time and characterise the blow-up
points, which are the complement of a convex set. Hence we are able
to construct non-negative initial data for which the solution
exhibits regional, or complete blow--up. One can even find solutions
with a finite pointwise limit at every point in $\R^{\N}$ at the
maximal existence time, but that can not be continued beyond this maximal time (`finite existence time without blowup'). In particular, we prove that given any closed convex set in
$\R^{\N}$, there exists an initial condition such that the solution remains bounded at the maximal existence time precisely on this set. Observe that most of this behaviour is characteristic of
nonlinear non-dissipative problems, see e.g.\ \cite{QuittnerSouplet}.
Our analysis includes and extends the classical example
$u_0(x)=\e^{A|x|^2}$, with $A>0$ for which the solution is given by
\begin{displaymath}
  u(x,t) = \frac{T^{\N/2}}{ (T-t)^{\N/2}}   \e^{\frac{|x|^2}{4(T-t)}},
\end{displaymath}
with $T=\frac{1}{4A}$ which blows up at every point $x\in\R^\N$ as $t\to T$.

For those solutions that exist globally in time we characterise those
that are unbounded and also construct (non-negative)  initial data such that the
solution displays wild unbounded oscillations (cf.\ \cite{VZ2002}). For this, given any
sequence of nonnegative numbers $\{\alpha_{k}\}_{k}$ we construct
initial data such that there exists a sequence of times $t_{k} \to
\infty$ such that for any $k\in \NN$ there exists a subsequence
$\{t_{k_{j}}\}_{j}$ such that
\begin{displaymath}
  u(0,t_{k_{j}}) \to \alpha_{k} \quad\mbox{as}\quad j\to \infty.
\end{displaymath}
We also show that this oscillatory behaviour is generic within a
suitable (optimal) class of solutions.  Notice that unbounded
oscillatory behaviour is an outstanding feature of some nonlinear
non-dissipative equations, see, for example,  Theorem 6.2 in
\cite{PolacikYanagida} where some solutions are shown to satisfy
\begin{displaymath}
  \liminf_{t \to \infty} \|u(t;u_{0})\|_{L^{\infty}(\R^{\N})} =0 \quad\mbox{and}\quad
  \limsup_{t \to \infty} \|u(t;u_{0})\|_{L^{\infty}(\R^{\N})} =\infty.
\end{displaymath}

All the nonlinear-like behaviour described above is caused by the
large mass of the initial data at infinity that is diffused by the
solution of the heat equation and is moved inwards bounded regions in
$\R^{\N}$, so that its effect is felt at later times.

Throughout the paper our analysis is based on the following  spaces: we define the subclass $\M_\eps(\R^\N)$ of Radon measures $\M_\loc(\R^\N)$
by setting
$$
\mathcal{M}_{\eps} (\R^\N):=\left\{\mu \in \mathcal{M}_\loc(\R^\N):\ \int_{\R^\N}
  \e^{-\eps |x|^{2}} \,  \d |\mu( x)|  < \infty\right\};
$$
where $ |\mu|$ denotes the total variation of $\mu$,
with the norm
\begin{displaymath}  
\|\mu\|_{\mathcal{M}_\eps (\R^\N)} := \left(\frac{\eps}{\pi}\right)^{\N/2}
\int_{\R^\N}\e^{-\eps|x|^2} \,\d |\mu(x)|;
\end{displaymath}
i.e.\  $\M_\eps
(\R^\N)$ consists of Radon measures for which
$\e^{-\eps |x|^{2}} \in L^{1}(\d|\mu|)$ and is a Banach space. [This set of measures
was briefly mentioned in \cite{Aro71}, which considered only
non-negative weak solutions of parabolic
problems.]
Since any locally integrable function $f \in  L^{1}_{\loc} (\R^{\N})$ defines
the  Radon measure  $f \, \d x \in  \mathcal{M}_\loc (\R^\N)$ the
class above contains
$$
L^1_\eps(\R^\N):=\left\{f\in L^1_\loc(\R^\N):\ \int_{\R^\N}
  \e^{-\eps|x|^2}|f(x)|\,\d x<\infty\right\}.
$$

These classes turn out to be optimal in several ways for
non-negative solutions
(\ref{eq:intro_solution_heat_up_t=0}) of
(\ref{eq:intro_heat_equation}) which are now given by
\begin{equation} \label{eq:intro_solution_heat_up_t=0_measure}
  u(x,t)  = \frac{1}{(4\pi t)^{\N/2}} \int_{\R^\N} \e^{-|x-y|^2/4t}
\, \d u_{0}(y) .
\end{equation}
First an initial condition in $\M_\eps
(\R^\N)$ gives rise to a (classical) solution of
(\ref{eq:intro_heat_equation}) defined for
$0<t<T(\eps)=\frac{1}{4\eps}$. Conversely for any non-negative solution
(\ref{eq:intro_solution_heat_up_t=0}) of
(\ref{eq:intro_heat_equation})  that is finite
at some  $(x,t)$ then the initial data must belong to $\M_{1/4t}(\R^\N)$.
As a consequence a  non-negative initial
condition in $\M_\loc(\R^\N)$ gives rise to a globally
defined solution if and only if  it belongs to
\begin{displaymath}
\M_0 (\R^\N) :=\bigcap_{\eps>0}\M_\eps (\R^\N).
\end{displaymath}
Within this class of initial data we also show that a non-negative solution is bounded
for some $t_{0}>0$ (and hence for all $t>0$) if and only if the initial
data is a uniform measure in the sense that
\begin{displaymath}
  \sup_{x \in \R^\N} \ \int_{B(x,1)} \d |u_{0}(y)|  < \infty .
\end{displaymath}
Finally we show that a non-negative solution is bounded on sets of the
form $|x|^{2}/t \leq R$, with $R>0$, if and only if
\begin{displaymath}
   \sup_{\eps >0}  \| u_{0} \|_{ \mathcal{M}_{\eps} (\R^{\N})} < \infty
   .
\end{displaymath}

In contrast, if
$$
\eps_0(u_{0})=\inf\{\eps>0:\ 0\leq u_0\in\M_\eps(\R^\N) \}>0
$$
 then the solution will exists only up to   $T= T(u_{0})= \frac{1}{4\eps_0}$ and cannot
 be continued beyond this time at any point.
The points $x$ at which the solution has a finite limit  as $t\to T$
are characterised by a condition on the translated measure, namely  $\tau_{-x} u_{0}\in\M_{\eps_0}
(\R^\N)$, and they must form a convex set. Conversely, as mentioned
above, at any chosen closed  convex subset of $\R^\N$,  there exist
some $u_{0}\geq 0$ such that the
limit as $t\to T$ of the solution is finite precisely
at this set.  In particular there are initial
conditions such that $\lim_{t\to T}u(x,t)<\infty$ for every
$x\in\R^\N$ but the solution cannot be defined past time $T$.

Large initial  data  can also exhibit other unusual properties
not normally associated with the heat equation. For example,
observe  that   for any  $\omega \in \R^{\N}$ the function $\varphi(x)=
\e^{\omega x} \in L^{1}_{0}(\R^{\N}) :=\bigcap_{\eps>0}L^1_\eps(\R^\N)$ satisfies
  $-\Delta \varphi = -|\omega|^{2} \varphi$,
while  $\phi(x) = \e^{\ri\omega x} \in L^{1}_{0}(\R^{\N})$  satisfies
 $-\Delta \phi = |\omega|^{2} \phi$.
It follows that  the spectrum of the Laplacian satisfies in this setting is the whole of $\R$,
\begin{displaymath}
  \sigma_{L^{1}_{0}(\R^{\N})} (-\Delta) = \R,
\end{displaymath}
and that for any $\omega\in\R^{\N}$ the function
  \begin{displaymath}
    u(x,t) = \e^{|\omega|^2 t + \omega x}  \quad x \in \R^{\N}, \ t >0
  \end{displaymath}
is a globally-defined solution of (\ref{eq:intro_heat_equation}) in
$L^{1}_{0}(\R^{\N})$; the exponential growth rate of such solutions can be
arbitrarily large.

The paper is organized as follows. In Section \ref{sec:radon-measures}
we recall some basic properties of Radon measures. In Section
\ref{sec:existence_regularity} we show that for an initial condition in
$\mathcal{M}_{\eps}(\R^{\N})$ the integral expression
(\ref{eq:intro_solution_heat_up_t=0_measure}) defines a classical
solution of the heat equation that attains the initial data in the
sense of measures. Conversely, we show that if
(\ref{eq:intro_solution_heat_up_t=0_measure}) is finite at some
$(x,t)$ for some non--negative measure $u_{0}$, then it must be in
some $\mathcal{M}_{\eps}(\R^{\N})$ space. In Section
\ref{sec:uniqueness} we tackle the problem of uniqueness. In Section
\ref{sec:blowup} we discuss and characterise the non-negative
solutions that cease
to exist in finite time, determining both the blow-up time $T$ and the points at which the solution has a finite limit as $t\to $T. In Section \ref{sec:asympt-behav-heat}  we discuss the
long-time behaviour of global solutions showing, in particular, wild unbounded
oscillations for some initial data; we show that this behaviour is generic (in an appropriate sense). Allowing for sign-changing solutions we also show there how to obtain solutions with any
prescribed behavior in time at $x=0$. Finally, in Section
\ref{sec:extensions-other-probl}, we briefly discuss other problems
that can be dealt with the same techniques. Appendix
\ref{sec:some-auxil-results} contains some required technical
results.

\section{Radon measures on $\R^{\N}$}
\label{sec:radon-measures}
\setcounter{equation}{0}

In this section we will recall some basic results on Radon measures that will
be used throughout the rest of the paper; details can be found in  \cite{BeneCzaja,Folland,FritzRoy,GiaqModicaSoucek}.
A Radon measure in $\R^\N$ is a  regular Borel measure assigning finite measure
to each compact set. The set of all Radon measures in $\R^\N$ is
denoted $\mathcal{M}_\loc (\R^\N)$.

Radon measures arise as the natural representation of linear functionals on the set $C_c(\R^d)$ of real-valued functions of compact support in two distinct settings.

\begin{theorem}
If $L\:C_{c}(\R^\N) \to \R$ is linear and positive, i.e.\ $L(\varphi)\geq0$ for $0\leq \varphi \in C_{c}(\R^\N)$, then there exists a (unique) non-negative Radon measure $\mu \in
\mathcal{M}_\loc (\R^\N)$ such that
$$
  L(\varphi) = \int_{\R^\N} \varphi \, \d\mu\qquad\mbox{for every}\quad\varphi \in C_{c}(\R^\N).
$$
\end{theorem}

A similar result holds if positivity is replaced by continuity, in the following sense: we equip $C_{c}(\R^\N)$ with the final (linear) topology associated with the
inclusions
\begin{displaymath}
  C_{c}(K) \hookrightarrow C_{c}(\R^\N), \qquad K {\subset\subset
    \kern2pt} \R^{\N}
\end{displaymath}
where, for each compact set $K \subset \R^{\N}$ we consider the sup
norm in $ C_{c}(K)$.
More concretely, a sequence $\{\varphi_{j}\}_{j} $ in $C_{c}(\R^\N)$ converges to
$\varphi \in C_{c}(\R^\N)$, iff there exists a compact $K\subset
    \R^\N$ such that ${\rm supp}(\varphi_{j} ) \subset K$ for
    all $j\in \NN$ and  $\varphi_{j} \to \varphi$ uniformly in $K$. A linear map $L\:C_c(\R^N)\to\R$ is then continuous if for every compact set $K\subset
    \R^\N$ there exists a constant $C_{K}$ such that for every
$\varphi \in C_{c}(\R^\N)$ with support in $K$
\begin{displaymath}
  |L(\varphi)| \leq C_{K} \sup_{x\in K} |\varphi(x)| .
\end{displaymath}

\begin{theorem} If  $L\:C_{c}(\R^\N) \to \R$ is linear and continuous (in the sense described above) then there exists a (unique, signed)  Radon measure $\mu \in
\mathcal{M}_\loc (\R^\N)$ such that
\begin{equation}  \label{eq:Radon_as_linear}
  L(\varphi) = \int_{\R^\N} \varphi \, \d \mu\qquad\mbox{for every}\quad\varphi \in C_{c}(\R^\N).
\end{equation}
\end{theorem}

As a consequence of this second theorem the set of Radon
measures can be characterised as the dual space of $C_c(\R^\N)$,
\begin{displaymath}
\mathcal{M}_\loc (\R^\N) =  \Big(C_{c}(\R^\N)\Big)',
\end{displaymath}
and we typically identify $L\in(C_c(\R^\N))'$ with the corresponding
Radon measure $\mu$ from (\ref{eq:Radon_as_linear}). In this way we
can write
$$
\<\mu,
\varphi\> = \D \int_{\R^\N} \varphi \, \d \mu\qquad\mbox{for every}\quad \varphi \in
C_{c}(\R^\N).
$$
Notice that, in particular,
\begin{displaymath}
  L^{1}_{\loc} (\R^{\N}) \subset \mathcal{M}_\loc (\R^\N)
\end{displaymath}
as we identify $f \in  L^{1}_{\loc} (\R^{\N})$ with the measure $f \,
\d x \in  \mathcal{M}_\loc (\R^\N)$.

Any Radon measure $\mu \in \mathcal{M}_\loc (\R^\N)$ can be
(uniquely) split as the difference of two non-negative, mutually
singular,  Radon measures
$\mu = \mu^{+} - \mu^{-}$ (the `Jordan decomposition' of $\mu$). Then we can define the Radon measure $|\mu|$, the `total variation of $\mu$', by setting
\begin{displaymath}
  |\mu|: =  \mu^{+} + \mu^{-}.
\end{displaymath}
Then for every  $\varphi \in C_{c}(\R^\N)$ and  $\mu \in
\mathcal{M}_\loc (\R^\N)$ we have
\be{dagger}
\left|\int_{\R^\N} \varphi \, \d \mu \right| \leq   \int_{\R^\N} |\varphi| \, \d|\mu| .
\ee

Finally  we recall the definition of
measures of bounded total variation.
Consider the space $C_{0}(\R^\N)$ of continuous functions converging
to $0$ as $|x|\to \infty$ with the $\sup$ norm ($C_{c}(\R^\N)$ is dense in this space).

\begin{theorem}
A linear mapping   $L\:C_{0}(\R^\N) \to \R$ is  continuous, 
iff  there exists a (signed)  Radon measure $\mu \in
\mathcal{M}_\loc (\R^\N)$  such that $|\mu|(\R^\N)<\infty$ and

\begin{displaymath}
  L(\varphi) = \int_{\R^\N} \varphi \, \d \mu\qquad\mbox{for every}\quad\varphi \in
C_{0}(\R^\N).
\end{displaymath}
\end{theorem}

The quantity $\|\mu\|_{\BTV} = |\mu|(\R^\N) $ is
the \emph{total variation} of $\mu$ and is the norm of the functional $L$.
In other words
\begin{displaymath}
\mathcal{M}_{\BTV} (\R^\N) =  \Big(C_{0}(\R^\N)\Big)'
\end{displaymath}
is the Banach space  of Radon measures with bounded total variation.
It is then  immediate that $L^{1}(\R^\N) \subset
\mathcal{M}_{\BTV} (\R^\N)$, isometrically, and $\mathcal{M}_{\BTV}
(\R^\N) \subset \mathcal{M}_{\loc} (\R^\N)$. We discuss solutions of
the heat equation with initial data in $\M_\BTV  (\R^\N)$ in Lemma
\ref{lem:heat_total_variation}.

Note that the set of Radon measures is therefore distinct from the
class of tempered distributions on $\R^\N$, which are continuous
linear functionals on the Schwarz class  ${\mathscr S}(\R^\d)$: such functions are
smoother than functions in $C_c(\R^\d)$ but satisfy less stringent
growth conditions, so neither class is contained in the other.
Recall that $\mathscr{S}(\R^\N)$ is made up of $C^{\infty}(\R^{\N})$
functions such that for all multi-indices $\alpha, \beta$
\begin{displaymath}
  |x^{\alpha}| |D^{\beta} \varphi (x)| \to 0 \quad\mbox{as}\quad |x| \to \infty .
\end{displaymath}
The family of seminorms
\begin{displaymath}
  p_{\alpha,\beta}(\varphi) = \sup_{x \in \R^{\N}} (1+ |x^{\alpha}|) |D^{\beta} \varphi (x)|
\end{displaymath}
defines a locally-convex topology on ${\mathscr S}(\R^\N)$, and the tempered distributions are the dual space $\mathscr{S}'(\R^\d)$.

A tempered distribution $L\in  \mathscr{S}'(\R^\N)$ has order $(m,n)\in \NN \times \NN$ if for all
$\varphi \in  \mathscr{S}(\R^\N)$ and some constant $c>0$
\begin{displaymath}
  |\< L, \varphi \>| \leq c p_{\alpha,\beta}(\varphi)
\end{displaymath}
with $|\alpha|= m$ and $|\beta| = n$.

Since $(1+ x^{\alpha}) \varphi
(x) \in \mathscr{S}(\R^\N)$ for every $\varphi \in \mathscr{S}(\R^\d)$
and multi-index  $\alpha$ and  $\mathscr{S}(\R^\N)$ is dense in
$C_{0}(\R^\d)$, it follows that if $L\in  \mathscr{S}'(\R^\N)$ has order $(m,0)$
then $(1+ |x|^{2})^{-m/2} L$ is an element of $\mathcal{M}_{\BTV} (\R^\N)$. That is, $L$
can be identified with a measure $\mu \in \mathcal{M}_{\loc} (\R^{\N})$
such that
\be{Cmcond}
  \int_{\R^{\N}} (1+|x|^{2})^{-m/2} \, \d |\mu(x)| < \infty,
\ee
since
\begin{displaymath}
    |\< L, \varphi \>| =  |\< (1+ |x|^{2})^{-m/2} L, (1+
    |x|^{2})^{m/2}\varphi \>|\leq c p_{\alpha,0}(\varphi) \leq c
    \sup_{x \in \R^{\ }} |\xi (x)|,
\end{displaymath}
with $\xi(x)= (1+ |x|^{2})^{m/2}\varphi (x) \in C_{0}(\R^\d)$ and  $|\alpha|=m$.

Let us denote by $\mathscr{C}_{m}(\R^\d)$ the collection of all measures $\mu$ that satisfy (\ref{Cmcond}). Then any such $\mu$ defines a tempered distribution of order
$(m,0)$ since for any  $\varphi \in
\mathscr{S}(\R^\d)$ we have
\begin{align*}
\left|  \int_{\R^{\N}} \varphi (x) \, d \mu(x) \right| & =  \left|\int_{\R^{\N}} (1+ |x|^{2})^{m/2}\varphi(x)
                                              \, (1+ |x|^{2})^{-m/2}\, \d   \mu(x) \right|\\
&\leq
p_{\alpha, 0}(\varphi) \int_{\R^{\N}} (1+|x|^{2})^{-m/2} \, \d |\mu(x)|
\end{align*}
with $|\alpha|=m$. Hence $\mathscr{C}_{m}(\R^\d)$ is precisely the class of
tempered distributions of order $(m,0)$.

\section{Initial data in $\mathcal{M}_{\eps}(\R^{\N})$: existence and regularity}
\label{sec:existence_regularity}
\setcounter{equation}{0}

Throughout this paper we consider the Cauchy problem
\begin{equation} \label{eq:heat_equation}
 u_{t} - \Delta u =0,\ x\in\R^\N,\ t>0,\qquad u(x,0)= u_{0}(x) ,
\end{equation}
whose solutions we expect to be given in terms of the heat kernel by
\begin{displaymath} 
 u(x,t; u_{0})=S(t)u_{0} (x) = \frac{1}{(4\pi t)^{\N/2}} \int_{\R^\N} \e^{-|x-y|^2/4t}
u_0(y) \, \d y ,
\end{displaymath}
if $u_0\in L^1_\loc(\R^\N)$ or, more generally,  if  $u_{0}
\in \mathcal{M}_\loc (\R^\N)$ is a Radon measure,  by
\begin{equation} \label{eq:solution_heat_up_t=0}
u(x,t; u_{0})  = S(t)u_{0} (x) = \frac{1}{(4\pi t)^{\N/2}} \int_{\R^\N} \e^{-|x-y|^2/4t}
\, \d u_{0}(y) .
\end{equation}
Of course, it is entirely natural to consider sets of measures as initial conditions for the heat equation, since the heat kernel, which is smooth for all $t>0$, is precisely the solution when $u_0$ is the $\delta$ measure.

Notice that from \eqref{eq:solution_heat_up_t=0} and \eqref{dagger} we immediately obtain
\begin{displaymath}
  |S(t) u_{0}| \leq S(t) |u_{0}|, \qquad t>0, \quad u_{0}
\in \mathcal{M}_\loc (\R^\N) .
\end{displaymath}

We start with some estimates for the expression in
(\ref{eq:solution_heat_up_t=0}) 
 which show that the solution can be  essentially estimated
 by its  value  at $x=0$.

\begin{lemma} \label{lem:estimate_from_x=0}

If $u_0\in \mathcal{M}_\loc(\R^\N)$ and $u(x,t)$ is given by
\eqref{eq:solution_heat_up_t=0} then for any $a>1$ we have
\begin{equation}\label{bound_above_x=0}
|u(x,t,u_{0})|\le c_{\N,a}\,u(0,at,|u_{0}|)\,
\e^{\frac{|x|^{2}}{4(a-1)t}}\qquad \mbox{for all}\quad x\in\R^\N,\ t>0,
\end{equation}
where $c_{\N,z}:=z^{\N/2}$ for any $z>0$.

If in addition    $0\leq u_{0}\in \mathcal{M}_{\loc} (\R^\N)$ then for any $0<b<1<a$ we have
\begin{equation} \label{bound_below_above_x=0}
  c_{\N,b}\, u(0,bt)\, \e^{-\frac{|x|^{2}}{4(1-b)t}} \leq u(x,t) \leq
  c_{\N,a}\, u(0,at)\, \e^{\frac{|x|^{2}}{4(a-1)t}}\qquad \mbox{for
    all}\quad x\in\R^\N,\ t>0.
\end{equation}

\end{lemma}

\begin{proof}
For the upper bound we use the fact that for any $0<\delta<1$, 
\begin{equation}\label{sqlower}
|x-y|^{2} \geq  |y|^{2} + |x|^{2} - 2|y||x| \geq
  (1-\delta) |y|^{2} + (1-\frac{1}{\delta}) |x|^{2},
\end{equation}
  from which it follows that
$$
 | u(x,t,u_{0})| \leq \e^{(\frac{1}{\delta}-1) \frac{|x|^{2}}{4t}} \left(
  \frac{1}{(4\pi t)^{\N/2}} \int_{\R^\N} \e^{- (1-\delta)
    \frac{|y|^2}{4t}}
\, \d |u_0(y)| \right);
$$
taking $a= \frac{1}{1-\delta} >1$ yields (\ref{bound_above_x=0}).

For the lower bound when $u_{0}\geq 0$, we argue similarly, now using
the fact that for any $\delta >0$, 
$$
|x-y|^{2} \leq |x|^{2} + |y|^{2} + 2 |x||y| \leq
  (1+\delta)  |x|^{2} +  (1 + \frac{1}{\delta}) |y|^{2};
  $$
we obtain
\begin{displaymath}
  u(x,t) \geq \e^{ -(1+\delta) \frac{|x|^{2}}{4t}} \left(
  \frac{1}{(4\pi t)^{\N/2}} \int_{\R^\N} \e^{- (1+\frac{1}{\delta})
    \frac{|y|^2}{4t}}
\, \d u_0(y) \right)
\end{displaymath}
and then take $b= \frac{\delta}{1+\delta} <1$.
\end{proof}

We now introduce some classes of initial data that are particularly
suited to an analysis of solutions of the heat equation: for $\eps >0$
we define
\begin{equation}  \label{eq:space_L1eps}
L^{1}_{\eps} (\R^\N):=\left\{f\in L^1_\loc(\R^\N):\ \int_{\R^\N}
  \e^{-\eps |x|^{2}} |f(x)|\,  \d x < \infty\right\};
\end{equation}
with the norm
\begin{equation}  \label{eq:norm_L1eps}
\|f\|_{L^1_\eps (\R^\N)} := \left(\frac{\eps}{\pi}\right)^{\N/2}
\int_{\R^\N}\e^{-\eps|x|^2}|f(x)|\,\d x
\end{equation}
for which  a  positive constant function
has  norm equal to itself. 
For the case of measures for $\eps >0$
we define
\begin{equation}  \label{eq:space_Meps}
\mathcal{M}_{\eps} (\R^\N):=\left\{\mu \in \mathcal{M}_\loc(\R^\N):\ \int_{\R^\N}
  \e^{-\eps |x|^{2}} \,  \d |\mu( x)|  < \infty\right\};
\end{equation}
i.e. $\e^{-\eps |x|^{2}} \in L^{1}(d|\mu|)$,
with the norm
\begin{equation}  \label{eq:norm_Meps}
\|\mu\|_{\mathcal{M}_\eps (\R^\N)} := \left(\frac{\eps}{\pi}\right)^{\N/2}
\int_{\R^\N}\e^{-\eps|x|^2} \,\d |\mu(x)| .
\end{equation}
Obviously $L^{1}_{\eps}
(\R^\N) \subset \mathcal{M}_{\eps} (\R^\N)$ isometrically, that is,  if $f\in L^{1}_{\eps}
(\R^\N)$ then $\|f\|_{\mathcal{M}_\eps (\R^\N)} = \|f\|_{L^1_\eps
  (\R^\N)}$.  Also note that $\mathcal{M}_{\eps}(\R^{\N})$ and $L^{1}_{\eps} (\R^\N)$ are
increasing in $\eps>0$ and if $\eps_{1}<\eps_{2}$ then for $\mu \in
\mathcal{M}_{\eps_{1}} (\R^{\N})$
\begin{equation} \label{eq:increase_Meps_norm}
  \|\mu\|_{\mathcal{M}_{\eps_{2}} (\R^\N)}  \leq
  \left(\frac{\eps_{2}}{\eps_{1}}\right)^{\N/2} \|\mu\|_{\mathcal{M}_{\eps_{1}} (\R^\N)} .
\end{equation}
Finally $L^1_\eps (\R^\N)$ and $\mathcal{M}_\eps (\R^\N)$ with the
norms (\ref{eq:norm_L1eps}) and (\ref{eq:space_Meps}) respectively, are
Banach spaces, see Lemma
\ref{lem:Meps_banach}.

The following simple lemma demonstrates the relevance of the spaces
$L^1_\eps (\R^\N)$ and $\mathcal{M}_\eps (\R^\N)$  to the heat
equation. Note that the first part of
the statement does not require that $u_0$ is non-negative. We will improve on the first part of this lemma in
Proposition \ref{prop:semigroup_estimates}, obtaining bounds on $u(t)$
in the norm of $L^1_{\eps(t)}(\R^{\N})$.

\begin{lemma}\label{whyL1e}

Let  $u_0\in \mathcal{M}_\eps(\R^\N)$, set $T(\eps)=1/4\eps$,  and let  $u(x,t)$ be given by
\eqref{eq:solution_heat_up_t=0}. Then for  each $t\in (0, T(\eps))$ we have $u(t) \in
L^{1}_{\delta} (\R^\N)$ for any $\delta>\eps(t):=
\frac{1}{4(T(\eps)-t)} = \frac{\eps}{1-4\eps t}$.

Conversely, if  $0\leq u_{0}\in \mathcal{M}_{\loc} (\R^\N)$ and $u(x,t) < \infty$ for some
$x\in\R^\N$, $t>0$ then
$$
  u_{0} \in \mathcal{M}_{\eps}(\R^\N) \quad\mbox{for every}\quad \eps > 1/4t .
$$

\end{lemma}

\begin{proof}
Taking $u_0\in \mathcal{M}_\eps(\R^\N)$ we use
the upper bound  \eqref{bound_above_x=0} from Lemma \ref{lem:estimate_from_x=0}  to obtain
\begin{displaymath} 
  \int_{\R^\N} \e^{-\delta |x|^{2}}|u(x,t)| \, \d x \leq c_{\N,a}\,u(0,at,|u_{0}|)
  \int_{\R^\N}  \e^{-(\delta - \frac{1}{4(a-1)t}) |x|^{2}}  \, \d x ,
\end{displaymath}
where we  choose any $1<a<T(\eps)/t$. Given such a choice of $a$, to
ensure that the integral is finite we require
$\delta>\frac{1}{4(a-1)t}$.
Noting that the right-hand side of this expression can be made arbitrarily close to
$\frac{1}{4(T(\eps)-t)}$
it follows that $u(t)\in L^1_\delta(\R^\N)$ for any $\delta>\eps(t):=1/4(T(\eps)-t)=\eps/(1-4\eps t)$, as claimed.

Conversely, from the lower bound in (\ref{bound_below_above_x=0}), if
$0\leq u(x,t) < \infty$ for some
$x\in\R^\N$, $t>0$ then for any $0<b<1$
\begin{displaymath}
  u(0,bt) = \frac{1}{(4\pi bt)^{\N/2}} \int_{\R^\N} \e^{-|y|^2/4bt}
\, \d u_0(y) <\infty,
\end{displaymath}
i.e.\ $u_{0} \in \mathcal{M}_{1/4bt}(\R^\N)$. Since we can take any $0<b<1$,
it follows that $u_0\in \mathcal{M}_\eps(\R^\N)$ for any $\eps>1/4t$.
\end{proof}

We reserve the notation $T(\eps)$ and $\eps(t)$ in what follows for the functions defined in the statement of this lemma; for the latter this is something of an abuse of notation, since $\eps(t)$ is really a function that depends on a particular choice of $\eps$ (as well as $t$):
\begin{equation}\label{epsilon-t}
T(\eps)=\frac{1}{4\eps}\qquad\mbox{and}\qquad\eps(t):=\frac{1}{4(T(\eps)-t)}
=\frac{\eps}{1-4\eps t},\quad 0\le t<T(\eps).
\end{equation}

At something of an opposite extreme, the following lemma - which we
will require many times in what follows - allows us to capture some of
the ways in which any solution starting from a continuous function
with compact support retains a trace of its initial data; more or less
it satisfies the same decay as the heat kernel,
$\sim t^{-\N/2}\e^{-|x|^2/4t}$.

\begin{lemma}\label{heat_solution_Cc}
  If $\varphi\in C_c(\R^\N)$ with ${\rm supp}\,\varphi\subset B(0,R)$
  then  for any $0<\delta <1$ and $t>0$
  \begin{itemize}
  \item [\rm(i)] $  |S(t)\varphi(x)|\le\begin{cases}
C_{\varphi}(t) \e^{-\gamma(t) |x|^2}&|x|\ge 2R/\delta\\
\|\varphi\|_{L^\infty(\R^{\N}) } &|x|\le 2R/\delta,
\end{cases}$\\
where
$$
C_{\varphi}(t)  = \frac{\e^{-3(1-\delta)R^2/4\delta t}}{(4\pi
  t)^{\N/2}} \|\varphi\|_{L^1(\R^{\N})}\qquad\mbox{and}\qquad\gamma(t)= \frac{(1-\delta)^2}{4t}.
  $$

  \item [\rm(ii)] $|S(t)\varphi(x)-\varphi (x) |\le\begin{cases}
C_{\varphi}(t) \e^{-\gamma(t) |x|^2}&|x|\ge 2R/\delta\\
\tilde C_{\varphi}(t) &|x|\le 2R/\delta,
\end{cases}$\\
with $C_{\varphi}(t)$ and $\gamma(t)$ as above and $\tilde C_{\varphi}(t)
\to 0$ as $t\to 0$.

  \item [\rm(iii)] In particular, for any $\eps>0$ and $0< T <T(\eps)=\frac{1}{4\eps}$ there exists
  $\gamma=\gamma(T,\eps)>0$ such that
  \begin{displaymath} 
  \e^{\eps|x|^2}|S(t)\varphi (x)|\le
  C_{T,\varphi,\eps}\e^{-\gamma|x|^2}, \quad x \in \R^{\N}
  \quad\mbox{for every }t\in[0,T].
  \end{displaymath}
In addition,
\begin{displaymath} 
\e^{\eps|x|^{2}} \big (S(t)
\varphi(x)-\varphi(x) \big) \to 0\quad\mbox{ uniformly in }\R^\N\quad\mbox{as}\quad t\to 0.
\end{displaymath}
  \end{itemize}

\end{lemma}

\begin{proof}
For any $\varphi \in C_{c}(\R^\N)$
with support in the ball $B(0,R)$, we have
\begin{equation} \label{eq:solution_heat_C0}
  |S(t) \varphi(x)|\le S(t)|\varphi(x)| = \frac{1}{(4\pi t)^{\N/2}} \int_{B(0,R)} \e^{-\frac{|x-y|^2}{4t}}
|\varphi(y)| \,\d y;
\end{equation}
using again
$$
|x-y|^{2} \geq 
(1-\delta) |x|^{2}
-\left(\frac{1}{\delta}-1\right) |y|^{2} \geq (1-\delta) |x|^{2}
-\left(\frac{1}{\delta}-1\right) R^{2}
$$
for any $0<\delta <1$, it follows that 
\begin{equation} \label{eq:point_bound_heat_C0}
  0\leq  S(t) |\varphi (x)| \leq  \frac{ \e^{-(1-\delta)\frac{|x|^2}{4t} + (\frac{1}{\delta}-1) \frac{R^2}{4t}} }{(4\pi t)^{\N/2}}
 \int_{B(0,R)}
|\varphi(y)| \,\d y  =\frac{\e^{-\frac{(1-\delta)}{4t} (|x|^{2} - \frac{R^2}{\delta})}}{(4\pi t)^{\N/2}}
  I(\varphi),
\end{equation}
where $I(\varphi)=\|\varphi\|_{L^1(\R^{\N})}$.

Now note that
\begin{displaymath}
  |x|^{2} -\frac{R^{2}}{\delta} \geq (1-\delta) |x|^{2}+\frac{3R^2}{\delta}
\end{displaymath}
if $|x|\ge 2R/\delta$, and hence for any such $x$ we obtain
$$
 0\leq |S(t) \varphi (x)| \leq\frac{\e^{-3(1-\delta)R^2/4\delta t}}{(4\pi t)^{\N/2}}
  \e^{-\frac{(1-\delta)^2}{4t}|x|^{2}} I(\varphi).
  $$
Since also   $\|S(t)\varphi\|_{L^\infty(\R^{\N})}\le\|\varphi\|_{L^\infty(\R^{\N})}$ for
all $t\ge0$, we get part (i).

Now, observe that for $|x|\ge 2R/\delta$ we get the same upper bound for $ |S(t) \varphi (x) -\varphi(x)|$
as above  and
 since as $\varphi \in {\rm
BUC}(\R^{\N})$ we know from e.g. \cite{Mora,HKMM96,L}
that $S(t)\varphi-\varphi
\to 0$ uniformly in $\R^{\N}$ as $t\to 0$. Hence we get part (ii).

Now fix $\eps>0$ and $0<T<T(\eps)$; we choose $0<\delta<1$ such that
$$
\gamma:=\frac{(1-\delta)^2}{4T}-\eps>0,
$$
i.e.\ so that for all $0\le t\le T$ we have
$(1-\delta)^2/4t\ge(1-\delta)^2/4T=\eps+\gamma$; note that $\gamma$
and $\delta$ can be chosen explicitly in such a way that they depend
only on $T$ and $\eps$.  Then parts (i) and (ii) give part (iii).
\end{proof}

Notice that in  particular if $u_0\in\M_\eps$ and $\varphi$ is as in the previous lemma then
\begin{equation} \label{eq:u0_tested_with_heat}
\int_{\R^\N} S(t) \varphi\,\d u_{0} = \int_{\R^\N} \e^{\eps|x|^{2}}
S(t) \varphi  (x)\,  \e^{-\eps|x|^{2}}
\,\d u_{0}(x)
\end{equation}
 is well defined for all $0\leq t \leq T<T(\eps)$.

The next preparatory result shows that the solution of the heat
equation for an initial condition that decays like a quadratic exponential preserves
this sort of decay, but with a rate that degrades in time.

\begin{lemma} \label{lem:estimate_fast_decay_initialdata}
If $\varphi \in C_{0}(\R^{\N})$ with $|\varphi (x)| \leq A
\e^{-\gamma|x|^{2}}$, $x\in \R^{\N}$, then $u(t)= S(t) \varphi$
satisfies
\begin{displaymath}
  |u(x,t)| \leq \frac{A}{(1+4\pi \gamma t)^{\N/2}}
  \e^{-\frac{\gamma}{1+4 \gamma t} |x|^{2}}, \qquad x\in \R^{\N},
  \quad t>0.
\end{displaymath}
\end{lemma}
\begin{proof}
  Note that completing the square yields
  \begin{displaymath}
       \frac{|x-y|^{2}}{4t} + \gamma |y|^{2}  =
  \frac{1+4 \gamma t}{4t}  \left|y-  \frac{1}{1+4\gamma t} x\right|^{2} +
  \frac{\gamma |x|^{2}}{1+4\gamma t}
  \end{displaymath}
and then
\begin{displaymath}
  |u(x,t)| \leq \frac{A}{(4 \pi t)^{\N/2}} \e^{-\frac{\gamma
      |x|^{2}}{1+4\gamma t} } \int_{\R^{\N}} \e^{- \frac{1+4 \gamma
    t}{4t}  \left|y-  \frac{1}{1+4\gamma t} x\right|^{2}} \, \d y = \frac{A}{(4 \pi t)^{\N/2}} \e^{-\frac{\gamma
      |x|^{2}}{1+4\gamma t} } \int_{\R^{\N}} \e^{- \frac{1+4 \gamma
    t}{4t}  |y|^{2}} \, \d y
\end{displaymath}
and the estimate follows.
\end{proof}

As a consequence, for any $u_{0}\in \mathcal{M}_{\eps}(\R^{\N})$ and
$\varphi$ that decays sufficiently fast, $u_{0}$ and $S(t)\varphi$ can
be integrated against each other for some time, see (\ref{eq:u0_tested_with_heat}).
In fact the following symmetry property holds.

\begin{lemma}\label{lem:preparing_4_Fubini}
Assume that  $\mu\in \mathcal{M}_{\eps}(\R^{\N})$ and  $\phi \in C_{0}(\R^{\N})$ is such that $|\phi (x)| \leq A
\e^{-\gamma|x|^{2}}$, $x\in \R^{\N}$ with $\gamma > \eps$.

Then for every $0<t< T(\eps) -T(\gamma) = \frac{1}{4\eps}-\frac{1}{4\gamma}$
\begin{displaymath}
  \int_{\R^{\N}}  \int_{\R^{\N}} K(x-y,t) |\phi(x)| \, \, \d
  |\mu(y)| \leq   (4\eps t)^{-\N/2}    \|\mu\|_{\mathcal{M}_{\eps}(\R^{\N})}  \int_{\R^{\N}}
    \e^{\eps(t) |x|^{2}} |\phi(x)| \, \d x .
\end{displaymath}
where $K(x,t)=(4\pi t)^{-\N/2}
\e^{-\frac{|x|^2}{4t}}$ is the heat kernel and $\eps(t)= \frac{1}{4(T(\eps)-t)}=
\frac{\eps }{1 -4\eps t}$.

In particular,  for  $0<t< T(\eps) -T(\gamma) = \frac{1}{4\eps}-\frac{1}{4\gamma}$
\begin{equation} \label{Fubbini_4-S(t)}
    \int_{\R^{\N}} \phi\, S(t)\mu  =  \int_{\R^{\N}} S(t) \phi \,
    \d \mu .
\end{equation}

\end{lemma}
\begin{proof}
  Notice that
  \begin{displaymath}
    I=  \int_{\R^{\N}}  \int_{\R^{\N}} K(x-y,t) |\phi(x)| \, \d x\, \d
  |\mu(y)| =  \int_{\R^{\N}}  \int_{\R^{\N}} K(x-y,t)
  \e^{\eps|y|^{2}} |\phi(x)|
 \e^{-\eps|y|^{2}} \, \d x \, \d  |\mu(y)|
  \end{displaymath}
and completing the square
  \begin{displaymath}
       \frac{|x-y|^{2}}{4t} -\eps |y|^{2}  =
  \frac{1-4 \eps t}{4t}  \left|y -  \frac{1}{1-4 \eps t} x\right|^{2} -
  \frac{\eps |x|^{2}}{1 -4\eps t} .
  \end{displaymath}
Hence
\begin{align*}
  I &\leq (4\pi t)^{-\N/2} \int_{\R^{\N}}  \int_{\R^{\N}}   \e^{
    -\frac{1-4 \eps t}{4t}  \left|y -  \frac{1}{1-4 \eps t}
      x\right|^{2}} \e^{\frac{\eps |x|^{2}}{1 -4\eps t}} |\phi(x)|
 \e^{-\eps|y|^{2}} \,  \d x\, \d  |\mu(y)| \\
&\leq   (4\eps t)^{-\N/2}    \|\mu\|_{\mathcal{M}_{\eps}(\R^{\N})}  \int_{\R^{\N}}
    \e^{\frac{\eps |x|^{2}}{1 -4\eps t}} |\phi(x)| \, \d x,
\end{align*}
which is finite as long as $\eps(t) < \gamma$, that is  $0<t< T(\eps)
-T(\gamma) = \frac{1}{4\eps}-\frac{1}{4\gamma}$.

The rest follows from Fubini's theorem.
\end{proof}

We can now show that for $u_0\in \mathcal{M}_\eps (\R^\N)$ [there is no requirement for $u_0$ to be non-negative] the
function defined in
(\ref{eq:solution_heat_up_t=0}) is indeed the solution of the heat
equation on the time interval $(0,1/4\eps)$, and satisfies the initial
data in the sense of measures. There are, of course, many classical
results on the validity of the heat kernel representation, but the
proof that follows has to be particularly tailored to $\mathcal{M}_\eps(\R^\N)$
initial data, since this allows for significant growth at infinity.

\begin{theorem} \label{thm:properties_sltns_given_u0}
Suppose that $u_{0}\in \mathcal{M}_\eps(\R^\N)$, set
  $T(\eps)=1/4\eps$, and let $u(x,t)$ be given by
  \eqref{eq:solution_heat_up_t=0}. Then
\begin{enumerate}  
\item[\rm(i)]   $u(t) \in L^{\infty}_{\loc}(\R^\N)$ for $t\in (0,
T(\eps))$. Also $u\in C^{\infty}
(\R^\N \times (0,T(\eps)))$ and satisfies
\begin{displaymath}
  u_{t} - \Delta u =0 \qquad\mbox{for all}\quad x\in \R^\N, \ 0< t <T(\eps).
\end{displaymath}

\item[\rm(ii)] For every
$\varphi \in C_{c}(\R^\N)$ and $0\leq t< T(\eps)$
\begin{displaymath}
  \int_{\R^\N} \varphi \,   u(t)=   \int_{\R^\N}  S(t)\varphi \, \d u_{0} .
\end{displaymath}
In particular, $u(t) \to u_{0} $ as $t\to 0^{+}$ as a measure, i.e.
\begin{displaymath}
  \int_{\R^\N} \varphi \,  u(t)   \to  \int_{\R^\N} \varphi \, \d
  u_{0}\qquad\mbox{for every}\quad \varphi \in C_{c}(\R^\N).
\end{displaymath}

\item[\rm(iii)] If $0\leq u_{0}\in \mathcal{M}_\eps(\R^\N)$ is non-zero
  then $u(x,t)>0$ for all $x\in\R^\N$, $t\in (0,T(\eps))$, i.e.\ the
  Strong Maximum Principle holds.

\end{enumerate}
\end{theorem}

\begin{proof} (i) If $u_{0}\in \mathcal{M}_{\eps} (\R^\N)$, then  for any $a>1$
\begin{displaymath}
   u(0,at,|u_{0}|) =  \frac{1}{(4\pi at)^{\N/2}} \int_{\R^\N} \e^{-|y|^2/4at}
\, \d  |u_0(y)|  <\infty
\end{displaymath}
provided that $\frac{1}{4at}\geq \eps$, that is $t\leq \frac{1}{4a\eps} <
T(\eps)$. Hence by (\ref{bound_above_x=0}) from Lemma
\ref{lem:estimate_from_x=0}, we have $u(t) \in L^{\infty}_{\loc}(\R^\N)$ for
$t\in (0,  T(\eps))$.

The  rest of part (i) follows from the regularity of the heat kernel, since
for any multi-index  $\alpha = (\alpha_{1}, \ldots, \alpha_{\N}) \in
\NN^{\N}$ and  $n\in \NN$, the derivatives satisfy
\begin{displaymath}
  D^{\alpha,n}_{x,t} K(x,t) =  {p_{\alpha,n}(x,t) \over  t^{{\N/2}+ |\alpha| + 2 n}}\, \e^{-|x|^{2}/4t},
\end{displaymath}
where  $p_{\alpha,n}(x,t)$ is a polynomial of degree not exceeding
$|\alpha| + 2n$. For $t$ bounded away from zero and  $\delta >0$ this
can be bounded by a constant times $ \e^{(-\frac{1}{4t}+ \delta)
  |x|^{2} }$. Therefore for $0<s\leq t\leq \tau < T(\eps)$
\begin{displaymath}
  \int_{\R^\N}  | D^{\alpha,n}_{x,t} K(x-y,t) | \, \d |u_0(y)| \leq C_{s,\tau}
  \int_{\R^\N} \e^{(-\frac{1}{4 \tau}+ \delta)
  |x-y|^{2} }  \, \d  |u_0(y)|
\end{displaymath}
with $0<\delta < \frac{1}{4 \tau}$.
Proceeding as in the upper bound in Lemma \ref{lem:estimate_from_x=0}
the above integral is bounded, for $x$ in compact sets and $0<\alpha <1$, by a multiple
of
\begin{displaymath}
  \int_{\R^\N}  \e^{(-\frac{1}{4 \tau}+ \delta) (1-\alpha)
  |y|^{2} }   \, \d  |u_{0}(y)|    =  \int_{\R^\N}  \e^{(\eps + (-\frac{1}{4\tau}+ \delta) (1-\alpha) )
  |y|^{2} }   \e^{-\eps|y|^{2}}   \, \d |u_{0}(y)|
\end{displaymath}
which is finite as long as we chose $\delta, \alpha$ small such that
$\eps < (1-\alpha) (\frac{1}{4\tau} -\delta)$. For this it suffices that
$\frac{1}{4 T(\eps)(1-\alpha)} = \frac{\eps}{1-\alpha} <
\frac{1}{4\tau} -\delta$ which is possible since $\tau < T(\eps)$.
Hence $u\in C^{\infty} (\R^\N \times (0,T(\eps)))$ and satisfies the
heat equation pointwise. 

For (ii), i.e.\ to show that the initial data is attained in the sense
of measures, notice first that it is enough to consider non-negative
test functions in  $C_{c}(\R^\N)$. Now, from Lemma
\ref{heat_solution_Cc} and (\ref{Fubbini_4-S(t)}) in Lemma
\ref{lem:preparing_4_Fubini},  we get for
$t$ small 
$$
\int_{\R^\N} \varphi   u(t) =   \int_{\R^\N}  S(t)\varphi\, \d u_{0} .
$$

Since Lemma \ref{heat_solution_Cc} also guarantees that $\e^{\eps|x|^{2}} \big (S(t)
\varphi(x)-\varphi(x) \big) \to 0$ uniformly in $\R^\N$ as $t\to 0$,
we can take $t\to0$ in (\ref{eq:u0_tested_with_heat}) and obtain
\begin{displaymath} 
  \int_{\R^\N} \varphi   \, u(t) = \int_{\R^\N} S(t) \varphi \, \d u_{0} =
  \int_{\R^\N} \varphi \, \d  u_{0}  +  \int_{\R^\N} \e^{\eps |x|^{2}}\big(S(t) \varphi
  -\varphi) \, \e^{-\eps |x|^{2}}  \, \d u_{0} \to \int_{\R^\N}
  \varphi \, \d u_{0}
\end{displaymath}
and (ii) is proved.

Part (iii) is a consequence of the lower bound in
(\ref{bound_below_above_x=0}) from Lemma
\ref{lem:estimate_from_x=0}.
\end{proof}

Now we derive some estimates on the solution in the
$L^1_\eps(\R^{\N})$ spaces introduced in (\ref{eq:space_L1eps}), using
the norm from \eqref{eq:norm_L1eps}. We also discuss the continuity of
the solutions in time. Note that part (i) shows that in fact whenever
$u_0\in\M_\eps(\R^{\N})$ we have $u(t)\in L^1_{\eps(t)}(\R^{\N})$; in part (iii) we
obtain a similar result for the derivatives of $u$, but with some loss
in the allowed growth (in $L^1_\delta (\R^{\N})$ only for $\delta>\eps(t)$).

Recalling the notations in (\ref{epsilon-t}), we have the following
result.

\begin{proposition} \label{prop:semigroup_estimates}
Suppose that  $u_{0}\in \mathcal{M}_\eps(\R^\N)$ and
let $u(x,t)$ be given by \eqref{eq:solution_heat_up_t=0}.
\begin{itemize}
\item[\rm(i)] For $0 < t <T(\eps)$
we have $u(t) \in
L^{1}_{\delta}(\R^{\N})$ for any  $\delta \geq \eps(t)$. Moreover
\begin{equation}\label{eq:estimate_solution_L1eps_L1delta}
  \|u(t)\|_{ L^{1}_{\eps(t)}(\R^{\N})} \leq  \|u_{0}\|_{
    \mathcal{M}_{\eps}(\R^{\N})}  .
\end{equation}

\item[\rm(ii)] For $0\leq s<t < T(\eps)$
\begin{equation} \label{eq:semigroup_solution}
  u(t) = S(t-s) u(s).
\end{equation}

\item[\rm(iii)] For any multi-index  $\alpha \in \NN^{\N}$, for $0 < t <T(\eps)$
we have $D^{\alpha}_{x}u(t) \in
L^{1}_{\delta}(\R^{\N})$ for any  $\delta > \eps(t)$. Moreover  for any  $\gamma>1$
we have
\begin{equation} \label{eq:estimate_derivative_solution_L1eps_L1delta}
  \|D^{\alpha}_{x}u(t)\|_{ L^{1}_{\delta(t)}(\R^{\N})} \leq
  \frac{c_{\alpha,\gamma}}{t^{\frac{|\alpha|}{2}}}  \|u_{0}\|_{ \mathcal{M}_{\eps}(\R^{\N})}\qquad\mbox{for all}\quad 0 < t <\frac{T(\eps)}{\gamma},
\end{equation}
where $\delta(t):=
\frac{1}{4(T(\eps)-\gamma t)} = \frac{\eps}{(1-4\eps \gamma t)}$.

\item[\rm(iv)]  For any multi-index  $\alpha \in \NN^{\N}$, $m\in \NN$
and for  each $t_{0} \in  (0,T(\eps))$ there exists
$\delta (t_{0}) > \eps$ such that the mapping $(0,T(\eps)) \ni t
\mapsto D^{\alpha,m}_{x,t}u(t)$ is continuous in $L^{1}_{\delta (t_{0})}(\R^{\N})$ at
$t=t_{0}$.

\end{itemize}

\end{proposition}

\begin{proof}
\noindent (i) Setting  $\delta = \frac{1}{4\tau}$
\begin{displaymath} 
\int_{\R^\N} \e^{-|x|^2/4\tau} |u(x,t)| \, \d x \leq   \frac{1}{(4\pi
  t)^{\N/2}} \int_{\R^\N} \int_{\R^\N} \e^{-|x|^2/4\tau}
\e^{-|x-z|^2/4t} \, \d |u_0(z)|  \, \d x.
\end{displaymath}
Notice that completing the square we obtain
\begin{equation} \label{eq:complete_squares}
  \frac{|x|^{2}}{\tau} + \frac{|x-z|^{2}}{t} =
  \frac{t+\tau}{t\tau} \left|x-  \frac{\tau}{t+\tau} z\right|^{2} +  \frac{|z|^{2}}{t+\tau}
\end{equation}
and so
\begin{displaymath}
  \int_{\R^\N} \e^{-|x|^2/4\tau} |u(x,t)| \, \d x \leq    \frac{1}{(4\pi
  t)^{\N/2}} \int_{\R^\N} \e^{-\frac{|z|^{2}}{4(t+\tau)} } \,
\d  |u_0(z)| \int_{\R^\N}  \e^{-\frac{t+\tau}{4t\tau} |x-  \frac{\tau}{t+\tau} z|^{2}}\, \d x.
\end{displaymath}
Since
\begin{displaymath}
  \int_{\R^\N}  \e^{-\frac{t+\tau}{4t\tau} |x-  \frac{\tau}{t+\tau}
    z|^{2}}\, \d x = \int_{\R^\N}  \e^{-\frac{t+\tau}{4t\tau}
    |x|^{2}}\, \d x = \left(\frac{4 \pi t\tau}{t+\tau}\right)^{\N/2}
\end{displaymath}
it follows that
\begin{displaymath}
   \int_{\R^\N} \e^{-|x|^2/4\tau} |u(x,t)| \, \d x \leq
   \left(\frac{\tau}{t+\tau}\right)^{\N/2}  \int_{\R^\N}
   \e^{-\frac{|z|^{2}}{4(t+\tau)} } \,
\d  |u_0(z)|  .
\end{displaymath}

Now given $t$ with $0<t<T(\eps)$, choose $\tau=T(\eps)-t=(1-4\eps t)/4\eps$; then $1/4\tau=\eps(t)$, $1/4(t+\tau)=\eps$, and this estimate becomes
$$
\eps(t)^{d/2}\int_{\R^\N}\e^{-\eps(t)|x|^2}|u(x,t)|,\d x\le\eps^{d/2}\int_{\R^\N}\e^{-\eps|z|^2}\,\d|u_0(z)|,
$$
which is precisely \eqref{eq:estimate_solution_L1eps_L1delta} up to a constant multiple of both sides.

\noindent (ii) Now
\begin{displaymath}
  S(t-s)u(s) (x) = \frac{1}{(4\pi (t-s))^{\N/2}} \int_{\R^\N} \e^{-|x-y|^2/4(t-s)}
u(y,s) \, \d y
\end{displaymath}
and
\begin{displaymath}
  u(y,s) = \frac{1}{(4\pi s)^{\N/2}} \int_{\R^\N} \e^{-|y-z|^2/4s}
\, \d u_0(z).
\end{displaymath}
Notice that completing the square as in (\ref{eq:complete_squares})
with $x-y$ replacing $x$, $z-y$ replacing $z$ and $t-s$ replacing
$\tau$ and $s$ replacing $t$, we get 
\begin{displaymath}
   S(t-s)u(s) (x) = \frac{1}{(4\pi (t-s))^{\N/2}}  \frac{1}{(4\pi
     s)^{\N/2}}  \int_{\R^\N}  \e^{-|x-z|^2/4t} \, \d u_{0}(z)
   \int_{\R^\N}  \e^{ -\frac{t}{4s(t-s)} |(y-z)-
  \frac{s}{t} (x-z)|^{2}}\, \d y
\end{displaymath}
and
\begin{displaymath}
   \int_{\R^\N}  \e^{ -\frac{t}{4s(t-s)} |(y-z)-
  \frac{s}{t} (x-z)|^{2}}\, \d y  =  \int_{\R^\N}  \e^{
  -\frac{t}{4s(t-s)} |y|^{2}}\, \d y  = \left(\frac{4\pi s(t-s)}{t}\right)^{\N/2}
\end{displaymath}
and the result is proved.

\noindent (iii)
Notice that for any multi-index  $\alpha \in \NN^{\N}$
\begin{displaymath}
D^{\alpha}_{x} u(x,t)=     \int_{\R^\N}  D^{\alpha}_{x}
K(x-y,t) \, \d u_0(y)  = \frac{1}{t^{\N/2 +  |\alpha|/2}}
\int_{\R^\N}  p_{\alpha}(x-y,t) \, \e^{-|x-y|^{2}/4t} \, \d u_0(y)
\end{displaymath}
with $p_{\alpha}(x-y,t)$ is a polynomial of degree $|\alpha|$ in
powers of  $\frac{x-y}{t^{1/2}}$.
Hence for any $0<\beta<1$
\begin{align} \label{eq:point_bound_further_derivatives}
|D^{\alpha}_{x} u(x,t)| &\leq   {c_{\alpha,\beta} \over  t^{{\N/2}+
    |\alpha|/2}}\, \int_{\R^\N}  \e^{- \beta\frac{ |x-y|^{2}}{4t}}
\, \d |u_0(y)|\\
&=\frac{\tilde c_{\alpha,\beta}}{t^{|\alpha|/2}}\,v(x,\gamma t),\nonumber
\end{align}
where $v(x,t)$ is the solution with initial data $|u_0|$ and
$\gamma=1/\beta>1$ is arbitrary. The estimate in
\eqref{eq:estimate_derivative_solution_L1eps_L1delta} follows using
part (i).

\noindent (iv)
Note that we can argue as we did for (\ref{bound_above_x=0}), and use
  (\ref{eq:point_bound_further_derivatives})  to obtain, for $0<\gamma <1$,
\begin{equation} \label{eq:exponential_bound_further_derivatives}
|D^{\alpha}_{x} u(x,t)| \leq   {c_{\alpha,\beta}  \over  t^{{\N/2}+
    |\alpha|/2}}  \e^{(\frac{1}{\gamma}-1) (1-\beta)\frac{|x|^{2}}{4 t}
  } \, \int_{\R^\N}  \e^{-(1-\gamma )(1-\beta) \frac{|y|^{2}}{4 t} }    \, \d |u_{0}(y)| ,
\end{equation}
which is  finite provided we choose $\beta, \gamma$
such that
$(1-\gamma)  (1-\beta) \frac{1}{4 t} > \eps$
i.e.\  provided that $t<T=(1-\gamma)(1-\beta)T(\eps)$.

 From the regularity of $u$ in Theorem
\ref{thm:properties_sltns_given_u0} we know that, as $t\to t_{0}$,
\begin{displaymath}
    D^{\alpha}_{x}  u(t) \to  D^{\alpha}_{x} u(t_{0})  \quad \mbox{in $L^{\infty}_{\loc}(\R^{\N})$}  .
\end{displaymath}
Now, if $\alpha =0$,    (\ref{bound_above_x=0})
implies that for
$\eps(t_{0}) =\frac{\gamma}{4(a-1)t_{0}}$ and $a, \gamma >1$ we
have a uniform quadratic exponential  bound for $u(t)$ for all $t$
close enough to $t_{0}$. For nonzero $\alpha$,
(\ref{eq:exponential_bound_further_derivatives}) implies that
for $0<\beta, \gamma<1$ and $\delta(t_{0}) = (1-\gamma)  (1-\beta)
\frac{1}{4 t_{0}} > \eps$ we
have again a uniform quadratic exponential bound for $D^{\alpha}_{x}
u(t)$ for all $t$ close enough to $t_{0}$.
Now, for $n\in \NN$,
\begin{displaymath}
  \|D^{\alpha}_{x}  u(t) - D^{\alpha}_{x}  u(t_{0})\|_{L^1_{\delta(t_{0})}(\R^\N)}= c  \int_{|x|\leq
  n}\e^{-\delta(t_{0}) |x|^2} |D^{\alpha}_{x} u(t) - D^{\alpha}_{x}
u(t_{0})|(x) \, \d x
\end{displaymath}
\begin{displaymath}
+ c  \int_{|x|\ge
  n}\e^{-\delta(t_{0}) |x|^2} |D^{\alpha}_{x} u(t) - D^{\alpha}_{x} u(t_{0})|(x) \, \d x .
\end{displaymath}
From the uniform quadratic exponential bound,
the second term is
arbitrarily small for sufficiently large $n$, uniformly in $t$ close
to $t_{0}$,  while
the first term is small, with fixed $n$ and  $t$ close enough to
$t_{0}$.

For time derivatives just note that for $m\in \NN$, $\partial_{t}^{m}
u(t) = (-\Delta)^{2m} u(t),$ and then
\begin{displaymath}
D^{\alpha,m}_{x,t}  u(t) = \partial_{t}^{m} D^{\alpha}_{x}
  u(t) = (-\Delta)^{2m}   D^{\alpha}_{x}  u(t)
\end{displaymath}
and we apply the argument above.\end{proof}

We now discuss further the sense in which the initial data is attained
(improving on part (ii) of Theorem
\ref{thm:properties_sltns_given_u0}).
First we show that $u(t)= S(t) u_{0}$ with $u_{0} \in
 \mathcal{M}_{\eps}(\R^{\N})$ attains the initial data against any
 test function that decays fast enough.

 \begin{corollary} \label{cor:attain_initial_measure}
If    $u_{0} \in  \mathcal{M}_{\eps}(\R^{\N})$ and $\varphi \in C_{0}(\R^{\N})$ is such that $|\varphi (x)| \leq A
\e^{-\gamma|x|^{2}}$, $x\in \R^{\N}$, with $\gamma > \eps$, then $u(t)=
S(t) u_{0}$  satisfies
\begin{displaymath}
  \int_{\R^{\N}} u(t) \varphi \to \int_{\R^{\N}} \varphi \, \d u_{0}\qquad\mbox{as}\quad t\to0.
\end{displaymath}
 \end{corollary}
 \begin{proof}
For $0\le    t<T(\eps)$ small and $\eps(t) = \frac{\eps}{1-4\eps t}$
we have $\gamma > \eps(t)$ and then
from~(\ref{Fubbini_4-S(t)}) in  Lemma \ref{lem:preparing_4_Fubini}
$\int_{\R^{\N}} u(t) \varphi = \int_{\R^{\N}} S(t) \varphi \, \d  u_{0}$.
Now, from  Lemma  \ref{lem:estimate_fast_decay_initialdata} it follows that for $t$ sufficiently small
$$|S(t) \varphi|(x) \leq  C   \e^{-\gamma(t)
  |x|^{2}},\quad\mbox{with}\quad \gamma(t)=\frac{\gamma}{1+4 \gamma t} > \eps,$$ 
  and
then   $|S(t) \varphi|(x) \leq  C   \e^{-\eps |x|^{2}} \in L^{1}(\d
|u_{0}|)$.  Also, $S(t) \varphi (x) \to \varphi(x)$ for $x\in \R^{\N}$
and then Lebesgue's theorem gives the result.
 \end{proof}

Assuming the initial data is a pointwise defined function, we get the
following result.

\begin{corollary}\label{cor:time_regularity}
Suppose that  $u_{0}\in L^{1}_\eps(\R^\N)$, set $T(\eps)=1/4\eps$, and
let $u(x,t)$ be given by \eqref{eq:solution_heat_up_t=0}. Then
\begin{itemize}
\item[\rm(i)] $u(t) \to u_{0}$ in $L^{1}_{\delta}(\R^{\N})$ as $t\to 0^+$ for any $\delta>\eps$;

\item[\rm(ii)] if $u_{0} \in L^{p}_{\loc}(\R^{\N})$ with $1\leq p <
\infty$ then
\begin{displaymath}
  u(t) \to u_{0} \quad \mbox{in}\quad L^{p}_{\loc}(\R^{\N})\qquad\mbox{as}\quad t\to0^+;\qquad\mbox{and}
\end{displaymath}
\item[\rm(iii)] if $u_{0} \in C(\R^{\N})$ then
  $u(t) \to u_{0}$ in $L^{\infty}_{\loc}(\R^{\N})$ as $t \to 0^+$.

\end{itemize}
\end{corollary}

\begin{proof}
(i)  Note that for any $\varphi \in C_{c}(\R^{\N})$ we have
\begin{displaymath}
\|S(t) u_{0} - u_{0}  \|_{L^{1}_{\delta}(\R^{\N})} \leq \|S(t) u_{0} -
S(t) \varphi   \|_{L^{1}_{\delta}(\R^{\N})} + \|S(t) \varphi  -
\varphi \|_{L^{1}_{\delta}(\R^{\N})} + \|\varphi  - u_{0}
\|_{L^{1}_{\delta}(\R^{\N})}.
\end{displaymath}

Let $\gamma >0$ and take $\varphi \in C_{c}(\R^{\N})$ such that
\begin{displaymath}
  \|u_{0} -\varphi\|_{L^{1}_{\eps}(\R^{\N})} = \int_{\R^{\d}}
  \e^{-\eps|x|^{2}} |u_{0}(x) - \varphi(x)| \, \d x < \gamma .
\end{displaymath}
To see this note that for $R>0$, if ${\rm supp}\,(\varphi) \subset B(0,R)$
then
\begin{displaymath}
  \int_{\R^{\d}}  \e^{-\eps|x|^{2}} |u_{0}(x) - \varphi(x)| \, \d x  =
  \int_{|x|\leq R}
  \e^{-\eps|x|^{2}} |u_{0}(x) - \varphi(x)| \, \d x  + \int_{|x|>R}
  \e^{-\eps|x|^{2}} |u_{0}(x)| \, \d x .
\end{displaymath}
The second term is small for $R$ large and so is the first one if we
approach $u_{0}$ by $\varphi$ in $L^{1}(B(0,R))$.

Now for any $\delta
>\eps$  and  all sufficiently small $t>0$ we have $\tilde
\delta(t)\le\delta$, where $\tilde \delta(t):=\frac{\eps}{(1-4\eps
  t)}$. Then  from (\ref{eq:increase_Meps_norm}) and
(\ref{eq:estimate_solution_L1eps_L1delta}) we have
\begin{displaymath}
   \|S(t) (u_{0}-\varphi) \|_{ L^{1}_{\delta}(\R^{\N})} \leq
     \left(\frac{\delta}{\eps}\right)^{\N/2}   \|S(t) (u_{0}-\varphi) \|_{ L^{1}_{\tilde
       \delta (t)}(\R^{\N})} \leq
 \left(\frac{\delta}{\eps}\right)^{\N/2}     \|u_{0}-\varphi \|_{
   L^{1}_{\eps}(\R^{\N})} <
 \left(\frac{\delta}{\eps}\right)^{\N/2} \gamma .
\end{displaymath}

Finally, as in Lemma \ref{heat_solution_Cc}  we have $S(t)\varphi
-\varphi \to 0$ uniformly in $\R^{\N}$ as $t\to 0$. Hence $ \|S(t) \varphi  -
\varphi \|_{L^{1}_{\delta}(\R^{\N})}  \to 0$ as $t\to 0$, which proves
(i).

\noindent (ii)  and (iii).   Fix $x_{0} \in \R^{\N}$ and $\delta>0$ and take $0\leq \varphi \in
C_{c}(\R^{\N})$ such that $0\leq \varphi \leq 1$, $\varphi =1$ on $B(x_{0}, \delta)$, and $\supp(\varphi)
\subset B(x_{0}, 2\delta)$.
Decompose $u_{0} = \varphi u_{0} + (1-\varphi) u_{0}$ and write
$$
u(t,u_{0})= u(t, \varphi u_{0}) + u(t, (1-\varphi) u_{0}).
$$

Then, if $u_{0} \in L^{p}_{\loc}(\R^{\N})$ with $1\leq p <
\infty$ we have  $\varphi u_{0} \in L^{p}(\R^{\N})$ then, as $t\to 0$,
\begin{displaymath}
  u(t,  \varphi u_{0}) \to \varphi u_{0}\quad \mbox{in $ L^{p}(\R^{\N})$}.
\end{displaymath}
In particular $u(t, \varphi u_{0}) \to  u_{0}$ in
$L^{p}(B(x_{0},\delta))$.  If $u_{0} \in C(\R^{\N})$ then $\varphi
u_{0} \in {\rm BUC}(\R^{\N})$ then, as $t\to 0$,
\begin{displaymath}
u(t,  \varphi u_{0}) \to \varphi u_{0}\quad \mbox{in $ L^{\infty}(\R^{\N})$}.
\end{displaymath}
In particular $u(t, \varphi u_{0}) \to  u_{0}$ in
$L^{\infty}(B(x_{0},\delta))$.

Now we prove that, as $t\to 0$,  $u(t, (1-\varphi) u_{0}) \to 0$
uniformly in a ball $B(x_{0}, \tilde \delta)$ for some $\tilde\delta <
\delta$, independent of $x_{0}$; this will conclude the proof of (ii)
and (iii).

For this notice that  for $x \in B(x_{0}, \delta/2)$
\begin{displaymath}
 u(t,  (1-\varphi) u_{0}) (x) =  \frac{1}{(4\pi t)^{\N/2}}
  \int_{|y-x_{0}|\geq \delta} \e^{-\frac{|x-y|^2}{4t}} (1-\varphi)(y) u_0(y)   \,\d y .
\end{displaymath}

Then
$|x-y|\geq |x_{0} - y| - |x-x_{0}| \geq  \delta - \delta/2 =
\delta/2$. Hence for $0<t <t_{0}$ and  $0<\alpha<1$,
$|x-y|^{2} \geq  \alpha
|x-y|^{2} + (1-\alpha)\frac{\delta^{2}}{4}$ and  we obtain
\begin{displaymath}
| u(t,  (1-\varphi) u_{0}) (x)| \leq
\frac{\e^{-(1-\alpha)\frac{\delta^{2}}{16 t}}}{(4\pi t)^{\N/2}}
  \int_{|y-x_{0}|\geq \delta} \e^{-\frac{\alpha|x-y|^2}{4t}} | u_0(y)|   \,\d y .
\end{displaymath}

Now we look for a uniform estimate in $x \in B(x_{0}, \tilde \delta)$ for
the right-hand side above. For this note that for $0<\beta <1$,
\begin{displaymath}
  |x-y|^{2} \geq  |y-x_{0}|^{2} + |x-x_{0}|^{2} - 2|y-x_{0}||x-x_{0}| \geq
  (1-\beta) |y-x_{0}|^{2} + (1-\frac{1}{\beta}) |x-x_{0}|^{2},
\end{displaymath}
thus for  $x \in B(x_{0}, \tilde \delta)$ and $0<t<t_{0} =
\frac{\alpha (1-\beta)}{8\eps}$ we have
\begin{displaymath}
| u(t,   (1-\varphi) u_{0}) (x)| \leq
\frac{\e^{-(1-\alpha)\frac{\delta^{2}}{16 t}}}{(4\pi t)^{\N/2}}
\e^{(\frac{1}{\beta}-1) \frac{|x-x_{0}|^{2}}{4t}}
  \int_{|y-x_{0}|\geq \delta} \e^{-\frac{\alpha (1-\beta)|x_{0}-y|^2}{4t}} | u_0(y)|   \,\d y
\end{displaymath}
\begin{equation} \label{eq:uniformly_2_zero}
  \leq \frac{\e^{-(1-\alpha)\frac{\delta^{2}}{16 t} +
      (\frac{1}{\beta}-1) \frac{\tilde \delta^{2}}{4t}}}{(4\pi t)^{\N/2}}
  \int_{\R^{\N}} \e^{-2\eps |x_{0}-y|^2} | u_0(y)|   \,\d y;
\end{equation}
again
  $|x_{0}-y|^{2} \geq  |y|^{2} + |x_{0}|^{2} - 2|y||x_{0}| \geq
  1/2 |y|^{2} - |x_{0}|^{2}$ gives
\begin{displaymath}
  \int_{\R^{\N}} \e^{-2\eps |x_{0}-y|^2} | u_0(y)|   \,\d y  \leq \e^{ 2\eps |x_{0}|}
  \int_{\R^{N}} \e^{- \eps |y|^2} | u_0(y)|   \,\d y
\end{displaymath}
and so (\ref{eq:uniformly_2_zero}) tends to $0$ as $t\to 0$ uniformly in $x \in
B(x_{0}, \tilde \delta)$  if $(\frac{1}{\beta}-1) \tilde
\delta^{2} < (1-\alpha)\frac{\delta^{2}}{4}$.  Notice, finally, that $\tilde \delta$
does not depend on $x_{0}$. 
\end{proof}

Notice that by comparing $\e^{-\eps |x|^{2}}$ and
  $(1+|x|^{2})^{-m/2}$ it follows that the class of tempered
  distributions of class $(m,0)$ as introduced 
at the end of Section \ref{sec:radon-measures},   satisfies, for all
  $\eps >0$,
  \begin{displaymath}
    \mathscr{C}_{m} (\R^{\N}) \subset     \mathscr{C}_{m+1} (\R^{\N})
    \subset     \mathcal{M}_{\eps} (\R^{\N}) .
  \end{displaymath}

\section{Initial data in $\mathcal{M}_{\eps}(\R^{\N})$: uniqueness}
\label{sec:uniqueness}
\setcounter{equation}{0}

Now we prove a uniqueness result for heat solutions with initial data
$u_{0} \in \mathcal{M}_{\eps}(\R^\N)$. Observe that uniqueness for
\textit{non-negative} weak solutions of (\ref{eq:heat_equation}) can
be found in \cite{Aro71}. On the other hand, one can find a proof of
the uniqueness of classical solutions with no sign assumptions but
with bounded continuous initial data in  \cite{Tychonov} in
dimension one and in e.g. \cite{J} (Chapter 7, page
176) in arbitrary dimensions, provided that they satisfy the pointwise bound
\begin{equation} \label{eq:quadratic_exponential_bound_upt=0}
  |u(x,t)|\leq M \e^{a |x|^{2}} , \ x \in \R^{\N}, \ 0<t<T.
\end{equation}

Here we prove a uniqueness result adapted to initial data in
$\mathcal{M}_\eps(\R^{\N})$, with no sign condition imposed. Observe
that if $u_0\in \mathcal{M}_\eps(\R^\N)$ for some $\eps>0$ and
$u(t)=S(t)u_{0}$ is  as in
(\ref{eq:solution_heat_up_t=0}), then assumptions
(\ref{eq:uniqueness_Du_in_L1delta}),
(\ref{eq:uniqueness_initial_data_semigroup}),
(\ref{eq:uniqueness_initial_data_fastdecay}),
(\ref{eq:uniqueness_u_in_L1eps}) and (\ref{eq:uniqueness_initial_data_Cc})
below are  ensured by  parts (i) and  (iii) in Proposition
\ref{prop:semigroup_estimates}, Corollary
\ref{cor:attain_initial_measure} and part (iii) in Theorem
\ref{thm:properties_sltns_given_u0}. Also
(\ref{eq:uniqueness_semigroup_solution}) is ensured by part (ii) in
Proposition \ref{prop:semigroup_estimates}.

\begin{theorem}
\label{thr:uniqueness}

Suppose that $u$, defined in  $\R^\N \times (0,T]$, is such that for
some $\delta >0$ and  for each $0<t<T$, $u(t) \in
L^{1}_{\delta}(\R^{\N})$. 
\begin{enumerate}
\item[{\rm(i)}] Suppose furthermore  that
  \begin{equation} \label{eq:uniqueness_Du_in_L1delta}
u, \nabla u,  \Delta u \in   L^{1}_{\loc}((0,T),
L^{1}_{\delta}(\R^{\N}))
  \end{equation}
and satisfies $u_{t} -\Delta u=0$ almost everywhere in $\R^{\N}
\times (0,T)$.
Then we have
\begin{equation} \label{eq:uniqueness_semigroup_solution}
u (t)= S(t-s)u(s)
\end{equation}
for any $0<s<t <T$.

\end{enumerate}
 Assume hereafter that $u$ satisfies
 \eqref{eq:uniqueness_semigroup_solution}  for any $0<s<t <T$.
\begin{enumerate}
\item[{\rm(ii)}]
Then  for each $0<t<T$ and every $\varphi \in C_{c}(\R^\N)$ the following limit exist
\begin{displaymath}
  \lim_{s \to 0} \int_{\R^\N} u(s) S(t) \varphi  = \int_{\R^\N} u(t)
  \varphi .
\end{displaymath}

\item[{\rm(iii)}] There exists  $u_0\in \mathcal{M}_\eps(\R^\N)$ for
  some $\eps >0$ and  such that $u(t)=S(t)u_{0}$ for $0<t<T$ if and
  only if for
every $\varphi \in C_{c}(\R^\N)$ and $t$ small enough
\begin{equation} \label{eq:uniqueness_initial_data_semigroup}
  \lim_{s \to 0} \int_{\R^\N} u(s) S(t) \varphi  = \int_{\R^\N}
  S(t)\varphi \, \d u_{0}.
\end{equation}

\item[{\rm(iv)}] Condition \eqref{eq:uniqueness_initial_data_semigroup} is satisfied  provided
  either one of the following holds:

\begin{enumerate}

\item [{\rm(iv-a)}] For any  function
$\phi \in C_{0}(\R^{\N})$ such that $|\phi (x)| \leq A
\e^{-\gamma|x|^{2}}$, $x\in \R^{\N}$, with $\gamma > \eps$ we have, as $ t \to 0$
\begin{equation}  \label{eq:uniqueness_initial_data_fastdecay}
  \lim_{t \to 0} \int _{\R^\N}  \phi u(t)   \to   \int _{\R^\N}  \phi
  \, \d u_{0} .
\end{equation}

\item[{\rm(iv-b)}] For some $\tau \leq T$ small and $0<t\leq \tau$ we
  have  $u(t) \in
L^{1}_{\eps}(\R^\N)$ with
  \begin{equation} \label{eq:uniqueness_u_in_L1eps}
  \int_{\R^\N} \e^{- \eps |x|^{2}}|u(x,t)| \, \d x \leq M \quad t\in (0,\tau];
  \end{equation}
i.e. $u\in L^{\infty}((0,\tau], L^{1}_{\eps}(\R^{\N}))$
and for every $\varphi \in C_{c}(\R^\N)$,   as $ t \to 0$
\begin{equation} \label{eq:uniqueness_initial_data_Cc}
  \int_{\R^\N}  \varphi \, u(t) \to    \int_{\R^\N}  \varphi \, \d
  u_{0} .
\end{equation}

\end{enumerate}
\end{enumerate}

\end{theorem}

\begin{proof}
For (i) the key to the proof is to show that for every $\varphi \in C_{c}(\R^\N)$ and $0<s<t<T$ one has
\begin{equation} \label{eq:integral_semigroup_solution}
   \int_{\R^\N} u(t) \varphi = \int_{\R^\N} u(s) S(t-s) \varphi
\end{equation}
for some small enough $T$ depending only on $\delta>0$ in
(\ref{eq:uniqueness_Du_in_L1delta}).
In such a case,  we then  apply  (\ref{Fubbini_4-S(t)}) with
$\mu = u(s)\in L^{1}_{\delta}(\R^{\N})$ and $\phi = S(t-s) \varphi$,
provided  $0\leq t-s \leq T$ is small enough (depending on $\delta>0$)
such that from Lemma
\ref{heat_solution_Cc},  $\phi$ satisfies the assumption in  Lemma
\ref{lem:preparing_4_Fubini} to obtain that the  right hand side of
(\ref{eq:integral_semigroup_solution}) equals $\D  \int_{\R^\N} S(t-s)
u(s) \varphi$. Hence, we get  (\ref{eq:uniqueness_semigroup_solution}) for $0<s<t<T$.
Then for $0<t_{0}<T$ consider $v(t) = u(t+t_{0})$ for $0\leq t\leq T$ which satisfies
the assumptions in (i). Hence (\ref{eq:uniqueness_semigroup_solution})
implies in particular $u(t+t_{0}) = S(t) u(t_{0})$ for $0\leq t_{0} , t\leq T$
which combined with (\ref{eq:semigroup_solution}) gives (\ref{eq:uniqueness_semigroup_solution}) for
 $0<s<t<2T$. In a finite numer of steps we obtain this property on any
finite time interval.

For the  proof of (\ref{eq:integral_semigroup_solution}) we  fix
$0< t<T$ and    differentiate
$$
I(s) :=
\int_{\R^\N} u(s) S(t-s) \varphi\qquad s\in (0,t)
$$
to obtain
\begin{equation} \label{eq:expression_4_derivative_I(s)}
  I'(s) =  \int_{\R^\N} \partial_{s} u (s) S(t-s) \varphi -  \int_{\R^\N} u(s)
  \partial_{s} S(t-s) \varphi
  =  \int_{\R^\N} \Delta u (s) S(t-s) \varphi -  \int_{\R^\N} u(s)
   \Delta S(t-s) \varphi .
\end{equation}

For this observe that from   Lemma
\ref{heat_solution_Cc} and
decreasing $T$ if necessary but depending only on $\delta$,   we have that for all $0<t<T$
\begin{displaymath}
   |S(t) \varphi(x)| \leq c  \e^{-\alpha |x|^{2}} , \quad x\in
   \R^{\N}, \quad 0<t<T,
\end{displaymath}
with $\alpha = \alpha (T) > \delta$, with $\delta$ as in
(\ref{eq:uniqueness_Du_in_L1delta}).
Also,  by (\ref{eq:solution_heat_C0}) and proceeding as in
(\ref{eq:point_bound_further_derivatives}) and as in
(\ref{eq:point_bound_heat_C0}) we obtain
\begin{displaymath}
|  \Delta S(t) \varphi (x) | \leq
 \frac{c}{t^{\N/2 +1}}
  \e^{-(1-\delta)^{2}\frac{|x|^2}{4t} + \frac{(1-\delta)^{2}}{\delta} \frac{R^2}{4t}} \int_{B(0,R)}
|\varphi (y)|  \,\d y \quad x\in
   \R^{\N}, \quad 0<t<T,
\end{displaymath}
and again as in Lemma \ref{heat_solution_Cc},  we obtain
\begin{displaymath} \label{eq:bound_Delta_heat_C0}
  |\Delta  S(t) \varphi(x)| \leq c \e^{-\alpha |x|^{2}}  \quad x\in
   \R^{\N}, \quad 0<t<T,
\end{displaymath}
with $\alpha= \alpha(T) >\delta$ and some $c=c(\varphi, R, T)$  with $\delta$ as
in (\ref{eq:uniqueness_Du_in_L1delta}).

With these, using (\ref{eq:uniqueness_Du_in_L1delta}),  the integrand on the right-hand side of
(\ref{eq:expression_4_derivative_I(s)}) has a bound
\begin{displaymath}
 | \Delta u (\cdot)| | S(t-\cdot) \varphi| + | u(\cdot)| | \Delta S(t-\cdot) \varphi|
 \in L^{1}_{loc}((0,t), L^{1}(\R^{\N})) . 
\end{displaymath}
 Hence, by differentiation inside the
integral, (\ref{eq:expression_4_derivative_I(s)}) is proved.

Now observe that (\ref{eq:uniqueness_Du_in_L1delta}) and the upper
bounds above for $S(t)\varphi$, $\Delta S(t) \varphi$   mean that
we can use a.e. $s\in (0,t)$, Lemma
\ref{lem:green_formula} below  to integrate
by parts in (\ref{eq:expression_4_derivative_I(s)}) to get  $I'(s)=0$
for $s\in (0,t)$
which gives that $I(s)$ is constant in $(0,t)$.

Now we show that, as $s\to t$ we have  $I(s) \to I(t) = \D \int_{\R^{\N}}
u(t) \varphi$. For this, write
\begin{displaymath}
  I(s) = \int_{\R^{\N}} u(s) \varphi + \int_{\R^{\N}} u(s) \big(
  S(t-s) \varphi -\varphi \big)
\end{displaymath}
and observe that from the assumptions $\partial_{t} u = \Delta u \in
L^{1}_{\loc}((0,T), L^{1}_{\delta}(\R^{\N}))$ and in particular,
$u(s)$ is continuous as $s\to t$ in $L^{1}_{\delta}(\R^{\N})$. On the
other hand from Lemma \ref{heat_solution_Cc} we have $\e^{\delta |x|^{2}} \big (S(t-s)
\varphi(x)-\varphi(x) \big) \to 0$ uniformly in $\R^{\N}$ as $s \to
t$. Hence (\ref{eq:integral_semigroup_solution})  and part (i) are
proved.

Now we prove (ii).  For  fixed $0<t <T$ and $s$ small, from  Lemma \ref{heat_solution_Cc},
we get that $u(t) S(s) \varphi$ is integrable and then, from
(\ref{eq:uniqueness_semigroup_solution}),  using Lemma
\ref{lem:preparing_4_Fubini}
again (with $\mu = u(s) \in L^{1}_{\delta}(\R^{\N})$ and $\phi=S(s) \varphi$) and (\ref{eq:semigroup_solution}),  we get
\begin{displaymath}
\int _{\R^\N}  u(t) S(s) \varphi = \int _{\R^\N}  S(t-s)u(s) S(s)
\varphi = \int _{\R^\N}  u(s) S(t-s) S(s)
\varphi =\int _{\R^\N}  u(s) S(t) \varphi .
\end{displaymath}
Now, from  Lemma \ref{heat_solution_Cc} we have that, as
$s \to 0$,
\begin{displaymath}
  | \int _{\R^\N}  u(t) \big( S(s) \varphi -\varphi \big) | \leq
  \int_{\R^\N} \e^{\delta |x|^{2}}|u(t)| \big(S(s) \varphi   -\varphi)
  \, \e^{-\delta |x|^{2}} \to 0
\end{displaymath}
and then
the following limit exists
\begin{displaymath}
  \lim_{s\to 0} \int _{\R^\N}  u(s) S(t) \varphi  =  \lim_{s\to 0}
  \int _{\R^\N}  u(t) S(s) \varphi =\int _{\R^\N}
  u(t)  \varphi .
\end{displaymath}
This concludes the proof of part (ii).

To prove (iii) notice that from  Lemma \ref{heat_solution_Cc} with $t$
small and Corollary
\ref{cor:attain_initial_measure} we have that condition
(\ref{eq:uniqueness_initial_data_semigroup}) is necessary for $u(t)$
to be equal to  $S(t)u_{0}$. Conversely if
(\ref{eq:uniqueness_initial_data_semigroup})  is satisfied then for
$t$ small and   $\varphi \in C_{c}(\R^\N)$
\begin{displaymath}
 \int _{\R^\N}  u(t)  \varphi  =     \lim_{s\to 0} \int _{\R^\N}  u(s)
 S(t) \varphi  =  \int _{\R^\N}   S(t) \varphi \, \d u_{0}
\end{displaymath}
Then by  (\ref{Fubbini_4-S(t)}) in Lemma \ref{lem:preparing_4_Fubini}
with $\mu = u_{0} \in \mathcal{M}_{\eps}(\R^{\N})$, $\phi = \varphi$,
we get
\begin{displaymath}
 \int _{\R^\N}  u(t)  \varphi  =   \int _{\R^\N}   S(t) \varphi \,
\d u_{0} = \int _{\R^\N}  \varphi   S(t)  u_{0}
\end{displaymath}
 for every $\varphi \in C_{c}(\R^\N)$  and then $u(t) = S(t) u_{0}$
 for $t$ small. This and  (\ref{eq:uniqueness_semigroup_solution}) proves
 part (iii).

For part (iv-a) it is now  clear that if $u$ satisfies
(\ref{eq:uniqueness_initial_data_fastdecay}) then Lemma
\ref{heat_solution_Cc} and $t$ small allows to take $\phi=S(t)\varphi$
in (\ref{eq:uniqueness_initial_data_fastdecay}) to get that
(\ref{eq:uniqueness_initial_data_semigroup}) satisfied.

Finally,
assuming (\ref{eq:uniqueness_u_in_L1eps}) and
(\ref{eq:uniqueness_initial_data_Cc}) we prove  part (iv-b).
For this  consider a sequence of smooth functions $0\leq
\phi_{n}\leq 1$ with $supp(\phi_{n}) \subset B(0,2n)$ and $\phi_{n}=1$
in $B(0,n)$. Then we write
\begin{displaymath}
  \int_{\R^\N} u(s) S(t)    \varphi  =   \int_{\R^\N} u(s)
  \phi_{n}S(t)    \varphi  +  \int_{\R^\N} u(s) (1-\phi_{n}) S(t)    \varphi
\end{displaymath}
and then
\begin{displaymath}
  \int_{\R^\N} u(s) S(t)    \varphi -   \int_{\R^\N} S(t)
  \varphi   \, \d   u_{0}  =   I_{1} + I_{2} + I_{3} =
  \end{displaymath}
  \begin{displaymath}
\big( \int_{\R^\N} u(s)
  \phi_{n}S(t)    \varphi -  \int_{\R^\N} \phi_{n}S(t)
  \varphi  \, \d  u_{0} \big)  +\int_{\R^\N}  (\phi_{n}-1) S(t)    \varphi  \, \d  u_{0}  +
  \int_{\R^\N} u(s) (1-\phi_{n}) S(t)
  \varphi .
\end{displaymath}

Now  $I_{3}$ goes to zero with $n\to \infty$ uniformly in $0<s<\tau$. To see
this, observe that by  Lemma \ref{heat_solution_Cc},  for  some
$t_{0}>0$ small and $0< t < t_{0}$ we have $0\leq
\e^{\eps|x|^{2}} (1-\phi_{n})| S(t)\varphi | \leq
c \e^{-\gamma |x|^{2}} (1-\phi_{n})$, $\gamma >0$  and $ 0\leq 1-\phi_{n} \to
0$ uniformly in compact sets as $n\to \infty$. Hence $\e^{\eps|x|^{2}} (1-\phi_{n})
S(t)\varphi \to 0$ as $n\to \infty$, uniformly in $\R^{\N}$ and
uniformly in $0<t<t_{0}$. Thus by
(\ref{eq:uniqueness_u_in_L1eps}),
\begin{displaymath}
  I_{3} = \int_{\R^\N} \e^{-\eps|x|^{2}} u(s) \e^{\eps|x|^{2}}
  (1-\phi_{n}) S(t)  \varphi \to 0, \quad n\to \infty
\end{displaymath}
uniformly for  $0<s<\tau$ and $0<t<t_{0}$.

With the same argument,  $I_{2}$ goes to zero with $n\to
\infty$ uniformly in $0<t<t_{0}$  since $u_{0}\in \mathcal{M}_{\eps}(\R^{\N})$.
Finally, by (\ref{eq:uniqueness_initial_data_Cc}), for any fixed $n$
and $0<t<t_{0}$, $I_{1} \to 0$ as $s\to 0$.
Hence, for any  $0<t<t_{0}$ we get  $ \lim_{s \to 0} \int_{\R^\N} u(s) S(t) \varphi  = \int_{\R^\N}
  S(t)\varphi \, \d u_{0}$ and part (iv-b) is proved.
\end{proof}

Notice that   condition (\ref{eq:uniqueness_initial_data_fastdecay})
is  precisely the definition of   ``initial data''  for the weak
solutions considered in \cite{Aro68}, page 319.

Also observe that, for $t$ small, from Lemma
\ref{lem:preparing_4_Fubini}, condition
(\ref{eq:integral_semigroup_solution}) is indeed equivalent to
(\ref{eq:uniqueness_semigroup_solution}) provided $u(s) \in
L^{1}_{\delta} (\R^{\N})$ for some $\delta>0$.

Finally, observe that if  we assume  $u\in C(\R^{\N} \times (0,T])$ is such
that for any $0<s<T$ there exists $M,a>0$  such that
\begin{displaymath}
  |u(x,t)|\leq M \e^{a |x|^{2}} , \ x \in \R^{\N}, \ s\leq t<T.
\end{displaymath}
from the results in  in  \cite{Tychonov} and \cite{J} (Chapter
7, page 176), then  (\ref{eq:uniqueness_semigroup_solution}) is satisfied.
Also, observe that from Lemma \ref{lem:estimate_from_x=0} implies that
$S(t)u_{0}$ satisfies the quadratic exponential bound above.
Therefore, if additionally $u$ satisfies
(\ref{eq:uniqueness_initial_data_fastdecay}) or
(\ref{eq:uniqueness_u_in_L1eps}) and
(\ref{eq:uniqueness_initial_data_Cc}) then we have
$u(t)=S(t)u_{0}$. These contitions are slightly weaker than the
classical Tychonov condition
(\ref{eq:quadratic_exponential_bound_upt=0}).

\section{Global existence versus finite-time
  blowup}
\label{sec:blowup}
\setcounter{equation}{0}

From the results in Section \ref{sec:existence_regularity} it
is natural to  set
\begin{displaymath} 
  L^{1}_{0}(\R^\N) = \bigcap_{\eps >0}
L^{1}_{\eps} (\R^\N)
\end{displaymath}
 and
 \begin{displaymath} 
   \mathcal{M}_{0}(\R^\N) = \bigcap_{\eps >0} \mathcal{M}_{\eps} (\R^\N
   ) .
 \end{displaymath}
 Clearly $L_0^1(\R^\N)\subset\M_0(\R^\N)$.

It is a simple consequence of Lemma \ref{whyL1e} that these
are precisely the collections of initial data for which (non-negative) solutions
exist for all time.

\begin{proposition}\label{prop:global_existence_L10}

If  $u_0\in \mathcal{M}_0(\R^\N)$ then $u(x,t)$, given by
\eqref{eq:solution_heat_up_t=0}, is well defined for all $x\in
\R^{\N}$ and $t>0$; in particular $u(t)\in L^1_0(\R^\N)$ for every
$t>0$.  
Conversely, if $u_{0}\in \mathcal{M}_{\loc} (\R^\N)$ with $u_0\ge0$ and  $u(x,t)$ is
defined for all $t>0$ then $u_0\in \mathcal{M}_0(\R^\N)$.

\end{proposition}

Note that $L^1_0(\R^\N)$ is a natural space of functions in which to study the heat semigroup, since $S(t)\:L^1_0(\R^N)\to L^1_0(\R^N)$ for every $t\ge0$; this form the main topic of our paper \cite{RR2}. For the time being, one can note that if $u_0\in \M_0(\R^N)$ then the estimate from Proposition \ref{prop:semigroup_estimates} can be reinterpreted as
$$
\|u(t)\|_{L^1_\delta}\le\|u_0\|_{\M_{\delta(t)}},\qquad\mbox{where}\quad\delta(t):=\frac{\delta}{1+4\delta t}.
$$

The collection $ L^1_0(\R^\N)$ is a large set of
functions: it contains $L^{p}(\R^\N)$ and $L^p_U(\R^\N)$ for every
$1\le p\leq \infty$, and (for example) any function that satisfies
$$
|f(x)| \le M\e^{k|x|^\alpha}, \quad x\in \R^{\N}
$$
for some $M>0$ and $\alpha<2$. It also contains functions that are not
bounded by any quadratic exponential, such as 
\begin{displaymath}
f(x)= \sum_k \alpha_{k} \, \Chi_{B(x_{k},r_k)} (x) , \quad \alpha_{k} = \e^{|x_{k}|^{3}}
\end{displaymath}
   with $|x_{k}| \to \infty$ and $r_{k} \to 0$ such that $r_{k}^{\N}
   \leq \frac{1}{\alpha_{k} k^{2}}$.

We show below that   $\mathcal{M}_0(\R^\N)$  also contains the space
of uniform measures
$\mathcal{M}_U(\R^\N)$ defined as the set of measures  $\mu \in
\mathcal{M}_{\loc}(\R^\N)$  such that
\begin{equation} \label{M_U}
\sup_{x \in \R^\N} \ \int_{B(x,1)} \d |\mu(y)|  < \infty
\end{equation}
with norm
\begin{equation}  \label{eq:M_U_norm}
\|\mu\|_{\mathcal{M}_{U}(\R^\N)} = \sup_{x \in \R^\N} \   \int_{B(x,1)} \d |\mu(y)|.
\end{equation}
This is a Banach space, see Lemma
\ref{lem:MU_banach}.

In fact,  as a consequence  of Theorem
\ref{thm:properties_sltns_given_u0}, we can show that the uniform
space $\mathcal{M}_{U}(\R^\N)$ is precisely the set of initial data for which
non-negative  solutions
of the heat equation given by (\ref{eq:solution_heat_up_t=0}) remain  bounded in
$\R^\N$ for positive times. See \cite{ACDRB2004} for results of the
heat equation between uniform spaces  $L^p_U(\R^\N)$,  $1\leq p
<\infty$, the collection of all
functions $\phi\in L^p_{\loc}(\R^\N)$  such that
\begin{displaymath} 
\sup_{x \in \R^\N} \ \int_{B(x,1)} |\phi(y)|^{p} \, \d y < \infty
\end{displaymath}
with norm
$\|\phi\|_{L^{p}_{U}(\R^\N)} = \sup_{x \in \R^\N} \
\|\phi\|_{L^{p}(B(x,1))}$.
For $p=\infty$ we have $L^{\infty}_{U}(\R^\N) = L^{\infty}(\R^\N)$ with norm
$\|\phi\|_{L^{\infty}_{U}(\R^\N)} = \sup_{x \in \R^\N} \
\|\phi\|_{L^{\infty}(B(x,1))} =
\|\phi\|_{L^{\infty}(\R^\N)}$.

\begin{proposition}\label{boundediff}
\begin{itemize}
  \item [\rm(i)] If $u_0\in \mathcal{M}_{U}(\R^\N)$ then $u_0\in
    \mathcal{M}_{0}(\R^\N)$ and $u(t)\in
    L^{\infty}(\R^\N)$ for all $t>0$ and for every $1\leq q \leq
    \infty$
    \begin{displaymath}
      \|u(t)\|_{L^{q}_{U}(\R^\N)}  \leq
M_0 \Big( t^{-\frac{\N}{2}(1-\frac{1}{q})} +1 \Big)
\|u_{0}\|_{\mathcal{M}_U(\R^\N)} .
    \end{displaymath}

\item[\rm(ii)]
Conversely, assume that $0\leq u_0\in \mathcal{M}_0(\R^\N)$.  If $0\le
u(t_{0})\in L^{\infty}(\R^\N)$ for some $t_{0}>0$
then
\begin{displaymath}
  u_{0} \in \mathcal{M}_{U}(\R^\N);
\end{displaymath}
hence $u(t)\in L^{\infty}(\R^\N)$ for all $t>0$.
  \end{itemize}

\end{proposition}
\begin{proof}

\noindent (i) Note first that  from
(\ref{eq:solution_heat_up_t=0}) we have $|S(t) u_{0}| \leq S(t)
|u_{0}|$ and, by definition, that $u_{0} \in \mathcal{M}_{U}(\R^\N)$ iff
$|u_{0}| \in \mathcal{M}_{U}(\R^\N)$. Hence, using Proposition
\ref{prop:global_existence_L10},  it is enough to prove the result for non-negative $u_{0}$.

  Let us consider a cube decomposition of $\R^{\N}$ as follows.  For
any index $i\in \Z^{\N}$, denote by $Q_{i}$ the open  cube in $\R^{\N}$ of
center $i$ with all edges of length 1 and parallel to the axes. Then
 $Q_i\cap Q_j=\emptyset$ for $i\neq j$ and $\R^{\N}=\cup_{i\in \Z^{\N}}
 \overline{Q_i}$. For a given $i\in \Z^{\N}$ let us denote by $N(i)$ the
 set of indexes near $i$, that is, $j\in N(i)$ if and only if
 $\overline{Q_i} \cap \overline{Q_j} \ne \emptyset$. Obviously
\be{dij}
d_{ij}:= \inf\{{\rm dist}(x,y),\, x\in Q_i,y\in Q_j\}
\ee
satisfies $d_{ij} =0$, if $j\in N(i)$, $d_{ij} \geq 1$, if $j \not
\in N(i)$, and as a matter of fact it is not difficult to see that
$d_{ij} \geq \|i-j\|_{\infty} - 1$. Let us denote by
$Q_i^{\rm near}=\cup_{j\in N(i)}Q_i$ and $Q_i^{\rm far}=\R^{\N}\setminus
\overline{Q_i^{\rm near}}$.

Assume that $u_{0} \in \mathcal{M}_U(\R^{\N})$ and, for a fixed $i$,
decompose
\begin{displaymath}
  u_{0} = u_{0} \Chi_{Q_i^{\rm near}}  + u_{0} \Chi_{Q_i^{\rm far}};
\end{displaymath}
by applying the linear semigroup $S(t)$ to each term in this equality we obtain the decomposition
$$
u(t)   =  u_i^{\rm near}(t) + u_i^{\rm far}(t) .
$$

The result  will follow from the following estimates of the two terms of the
decomposition. First,
\begin{equation}\label{basicLp-Lq}
  \| u_i^{\rm near}(t)\|_{L^q(Q_i)} \leq (4\pi
t)^{-\frac{\N}{2}(1-\frac{1}{q})} \|u_{0}\|_{\mathcal{M}(Q_i^{\rm near})}
\quad t > 0 ,
\end{equation}
for $1\leq q \leq \infty$ 
and, second,
\begin{equation}\label{basicL1-Linfty}
\| u_i^{\rm far}(t) \|_{L^\infty(Q_i)} \leq
c(t)  \|u_{0}\|_{\mathcal{M}_U(Q_i^{\rm far})},\quad t \geq 0.
\end{equation}
for some bounded monotonic function $c(t)$ such that $c(0)=0$ and
$0 \leq c(t) \leq C t^{-\N/2}\e^{-\alpha/t}$ as $t \to 0$, where  $C$ and
$\alpha>0$ depend only on $N$.

Then  since the constant for the embedding $L^\infty(Q_i)
\hookrightarrow L^q(Q_i)$  is $1$, independent of $q$ and $i$,
(\ref{basicL1-Linfty}),(\ref{basicLp-Lq}) imply
$$
\|u(t) \|_{L^q(Q_i)}\leq
((4\pi t)^{-\frac{\N}{2}(1-\frac{1}{q})} + c(t))
\|u_{0}\|_{\mathcal{M}_U(\R^{\N})}   \qquad i\in \Z^{\N} .
$$ 
Since the   $L^q_U(\R^N)$ norm can be bounded by a constant, only depending
on $N$, times the supremum of the $L^q(Q_i)$ norms, (i) follows.

Now observe that  (\ref{basicLp-Lq}) follows from ``standard''  estimates for the heat equation, since in
fact, $u_{0}\Chi_{Q_i^{\rm near}} $ is a measure of bounded total
variation and then for $t>0$,
\be{xtra}
\| u_i^{\rm near}(t)\|_{L^{q} (Q_i)} \leq \|S(t)(u_{0}
\Chi_{Q_i^{\rm near}})\|_{L^{q} (\R^{\N})} \leq
(4\pi t)^{-\frac{\N}{2} (1-\frac{1}{q})}
\|u_{0}\|_{\mathcal{M} (Q_i^{\rm near})}
\ee
since $\|u_{0} \Chi_{Q_i^{\rm near}} \|_{\mathcal{M}_{BTV} (\R^{\N})} =
\|u_{0}\|_{\mathcal{M} (Q_i^{\rm near})} $ see Lemma \ref{lem:heat_total_variation} below.

We now prove (\ref{basicL1-Linfty}). Observe that
$u_{0} \Chi_{Q_i^{\rm far}}  =\sum_{j\in \Z^{\N}\setminus N(i)}u_0^{j}$ where
$u_0^{j} = u_{0} \Chi_{Q_j}$; for each $j$ we have
$$
S(t)u_{0}^{j}(x) = (4\pi t)^{-\N/2} \int_{\R^{\N}}
\e^{-\frac{|x-y|^2}{4t}} \, \d u_{0}^{j}(y) = (4\pi
t)^{-\N/2}\int_{Q_j}\e^{-\frac{|x-y|^2}{4t}} \, \d u_{0}^{j}(y)
$$
which implies that for $j \not \in N(i)$
$$
\|S(t)u_{0}^{j}\|_{L^\infty(Q_i)}\leq (4\pi
t)^{-\N/2}\e^{-d_{ij}^2\over 4t} \|u_{0}^{j}\|_{\mathcal{M} (Q_j)}\leq  (4\pi
t)^{-\N/2}\e^{-d_{ij}^2\over 4t}\|u_{0}\|_{\mathcal{M}_U(Q_i^{\rm far})},
$$
where $d_{ij}$ is defined above in \eqref{dij}.  Hence,
\begin{align*}
\| u_i^{\rm far}(t)\|_{L^\infty(Q_i)} &\leq \sum_{j\in
\Z^{\N}\setminus N(i)} \|S(t)u_{0}^{j}\|_{L^\infty(Q_i)}
&\leq  (4\pi
t)^{-\N/2} \|u_{0}\|_{\mathcal{M}_U(Q_i^{\rm far})} \sum_{j\in \Z^{\N}\setminus
  N(i)}\e^{-d_{ij}^2\over 4t} .
\end{align*}
But, using that $\#\{j \in \Z, \ d_{ij}=k\} \leq C k^{\N-1}$, we obtain
$$
\sum_{j\in \Z^{\N}\setminus N(i)}\e^{-d_{ij}^2\over 4t}\leq
C\sum_{k=1}^\infty k^{\N-1} \e^{-k^2\over 4t}
$$
which has the same character as the integral 
\begin{displaymath}
   \int_{1}^{\infty} r^{\N-1}  \e^{- \frac{r^{2}}{4t}} \,  \d r  =
  (4t)^{\N/2} \int_{\frac{1}{\sqrt{4 t}}}^{\infty} s^{\N-1}  \e^{-s^{2}} \,  \d s  = t^{\N/2}c(t)
\end{displaymath}
with $c(t)$ as claimed in (\ref{basicL1-Linfty}).

\noindent (ii)    If for some $t_{0}>0$ we have $u(t_{0})\in L^{\infty}(\R^\N)$
then from (\ref{eq:solution_heat_up_t=0}) we get
for all $x\in \R^\N$ and any $R>0$
\begin{align*}
  \infty> M \geq u(x,t_{0})& \geq \frac{1}{(4\pi t_{0})^{\N/2}} \int_{B(x,R)} \e^{-\frac{|x-y|^2}{4t_{0}}}
\, \d u_0(y) \\
& \geq \frac{1}{(4\pi t_{0})^{\N/2}}  \inf_{z\in B(0,R)} \e^{-\frac{|z|^2}{4t_{0}}} \int_{B(x,R)}
\, \d u_0(y)
\end{align*}
that is
\begin{displaymath}
  0\leq \int_{B(x,R)} \,\d  u_0(y)  \leq M \e^{\frac{R^2}{4t_{0}}}  (4\pi
  t_{0})^{\N/2} , \quad x\in \R^\N
\end{displaymath}
i.e. $0\leq u_{0} \in \mathcal{M}_{U}(\R^\N)$.
From part (i) we  obtain $u(t)\in L^{\infty}(\R^\N)$ for all
$t>0$.
\end{proof}

Now we prove the result used above in \eqref{xtra}.
Note that  from (\ref{eq:space_Meps}),   $\mathcal{M}_{\BTV} (\R^\N)
\subset \mathcal{M}_{U} (\R^{\N})  \subset
\mathcal{M}_{0} (\R^\N)$.
The following lemma shows that $\M_\BTV$ is invariant under the heat equation, and gives bound on the rate of decay in $L^q$ of solutions when $u_0\in\M_\BTV$.

\begin{lemma}\label{lem:heat_total_variation}

For  $\mu \in \mathcal{M}_{\BTV} (\R^\N) $ the solution of the heat equation
given by (\ref{eq:solution_heat_up_t=0}) satisfies
\begin{displaymath}
  \|S(t)\mu \|_{\BTV} \leq   \|\mu \|_{\BTV}, \quad t>0,
\end{displaymath}
and for every $1\leq q \leq \infty$
\begin{displaymath}
  \|S(t)\mu \|_{L^{q}(\R^\N)} \leq  (4\pi t)^{-\frac{\N}{2}(1-\frac{1}{q})}
    \|\mu \|_{\BTV}, \quad t>0 .
\end{displaymath}
\end{lemma}
\begin{proof}
  Observe that since for every
$\varphi \in C_{c}(\R^\N)$ and $0\leq t< \infty$, $u(t) =S(t)\mu$
satisfies (\ref{Fubbini_4-S(t)}) that is,
\begin{displaymath}
  \int_{\R^\N} u(t) \varphi =   \int_{\R^\N}  S(t)\varphi  \, \d \mu
\end{displaymath}
then
\begin{displaymath}
\left|  \int_{\R^\N} u(t) \varphi \right| \leq
\|S(t)\varphi\|_{L^{\infty}(\R^\N)}   \|\mu\|_{\BTV}  .
\end{displaymath}
Therefore the estimates (\ref{eq:heat_LpLq_estimates}) give, for every
$1\leq q \leq \infty$,
\begin{displaymath}
\left|  \int_{\R^\N} u(t) \varphi \right| \leq (4\pi
t)^{-\frac{\N}{2q'} } \|\varphi\|_{L^{q'}(\R^\N)}   \|\mu\|_{\BTV}
\end{displaymath}
and the claims follow. Note that in particular, for $q=1$ since $u(t)
\in L^{1}_{\loc}(\R^\N) \cap \mathcal{M}_{\BTV} (\R^\N)$ then $u(t)\in
L^{1}(\R^\N)$.
\end{proof}

\subsection{Finite-time blowup for non-negative initial data}
\label{sec:finite-time-blowup}

Now we turn to non-negative solutions that
may not exist for all time, that is, according to Proposition
\ref{prop:global_existence_L10},  $0\leq u_{0}\notin
\mathcal{M}_{0}(\R^\N)$. Lemma \ref{whyL1e} shows that the maximal
existence time for the solution arising from the non-negative initial condition
$0\leq u_0\in \mathcal{M}_\loc(\R^\N)$ will be determined by its `optimal index'
\begin{equation}\label{eps0}
\eps_{0}(\mu)  := \inf\{\eps: \  \mu \in
\mathcal{M}_{\eps}(\R^\N) \} = \sup\{\eps: \  \mu  \notin
\mathcal{M}_{\eps}(\R^\N) \} \leq \infty.
\end{equation}

The simplest example is to take $A>0$ and consider
$u_0(x)=\e^{A|x|^2}$; then $u_0\in L^1_\eps(\R^\N)$ if and only if $\eps>A$,
so in this case $\eps_0(u_0)=A$ but $u_0\notin L^1_A(\R^\N)$. If we set
$T=1/4A$ then the integral in (\ref{eq:solution_heat_up_t=0}) can be
computed explicitly and one gets
\begin{equation} \label{eq:blowing_up_quadratic_exponential}
  u(x,t) = \frac{T^{\N/2}}{ (T-t)^{\N/2}}   \e^{\frac{|x|^2}{4(T-t)}},
\end{equation}
which satisfies the heat equation  for $t\in (0,T)$, has $u(x,0)=
u_{0}(x)$, and blows up at every point $x\in\R^\N$ as $t\to T$. 
At the other extreme is an initial condition like
$$
u_0(x)=\e^{A|x|^2-\gamma|x|^\alpha}
$$
for some $\gamma>0$ and $1<\alpha<2$, which we treat as Example \ref{howodd}, below. In this case $\lim_{t\to 1/4A}u(x,t)$ exists for every $x\in\R^\N$, but the solution cannot be extended past $t=T$.

Below  we analyse the behaviour for a general non-negative initial
condition $u_0$ that is not an element of $\mathcal{M}_{0}(\R^\N)$. While the
time span of the solution does not depend specifically on any fine
properties of the initial data, but only its asymptotic growth as
$|x| \to \infty$ (in terms of its optimal index), the existence or
otherwise of a finite limit as $t\to T$ is more delicate. We will see
below that at the maximal existence time a number of different
behaviours are possible: from complete blowup, as in the example  (\ref{eq:blowing_up_quadratic_exponential})
above,  to the existence of a finite limit at all points in space. We will show that by `tuning' the
initial data it is possible to obtain solutions with  a finite limit
only at any chosen
convex subset of $\R^\N$.
These results, in turn, will depend on the integrability at the optimal index of the
translate of the initial data.

In the case of pointwise-defined functions, any translation $\tau_{y} f(x) :=
f(x-y)$ has the same optimal index as $f$, since whenever $f\in
L^1_\eps(\R^\N)$ we have $\tau_yf\in L^1_\delta (\R^\N)$ for any
$\delta>\eps$:
\begin{displaymath}
  \int_{\R^\N} \e^{-\delta |x|^{2}} |f(x-y)|\,  \d x =   \int_{\R^\N}
  \e^{-\delta |z+y|^{2}} |f(z)|\,  \d z\le  \e^{\delta  (\frac{1}{\alpha}-1) |y|^{2} }  \int_{\R^\N}
  \e^{-\delta (1-\alpha) |z|^{2}} |f(z)|\,  \d z<\infty
\end{displaymath}
for $\delta (1-\alpha) \geq  \eps$, using \eqref{sqlower}. However,
whether or not $\tau_{y}  f \in L^{1}_{\eps_{0}(f)}(\R^\N)$ depends
strongly on the decay at infinity of ``lower order terms'' of $f$ as
the examples below will show.

For the case of measures, observe that we can define translations of measures via the formula that would hold for a locally integrable $f$ and
$\varphi \in C_{c}(\R^\N)$, namely
\begin{displaymath}
  \int_{\R^\N} \tau_{y} f(x) \varphi(x) \,\d x =  \int_{\R^\N} f(z)
  \varphi (z+y) \, \d z =  \int_{\R^\N}  f (z)\tau_{-y} \varphi(z)\,
  \d z .
\end{displaymath}
That is, for $y\in \R^\N$ and $\mu \in \mathcal{M}_\loc (\R^\N)$ we
define $\tau_y\mu$ by setting
\begin{displaymath}
  \int_{\R^{\N}} \varphi (x)\, \d \tau_{y} \mu(x) := \int_{\R^\N} \tau_{-y} \varphi(z) \,
  \d \mu(z)\quad\mbox{for every}\quad \varphi \in C_{c}(\R^\N).
\end{displaymath}
Hence $\tau_{y} \mu \in \mathcal{M}_\loc (\R^\N)$ and
is a  positive measure whenever $\mu$ is. It also follows that if $\mu =
\mu^{+}-\mu^{-}$ then
\begin{displaymath}
  \tau_{y} \mu =   \tau_{y}  \mu^{+} - \tau_{y} \mu^{-} \qquad\mbox{and}\qquad |\tau_{y} \mu| = \tau_{y}
|\mu| .
\end{displaymath}

\begin{lemma}\label{lem:translation_of_measures}

For $\mu \in \mathcal{M}_\eps(\R^\N)$ and $y\in \R^{\N}$, we have $\tau_y\mu \in \mathcal{M}_\delta (\R^\N)$ for any
$\delta>\eps$ and

\begin{displaymath}
\int_{\R^{\N}}
  \e^{-\delta |x|^{2}} \, \d |\tau_{y} \mu(x) |=   \int_{\R^{\N}} \e^{-\delta |x+y|^{2}} \, \d |\mu(x)| .
\end{displaymath}
In particular,  $\tau_{y} \mu $ has the same optimal index as $\mu$.

\end{lemma}
\begin{proof}
  Take $\phi_{k} \in C_{c}(\R^{\N})$ such that $0\leq  \phi_{k} \leq 1$
  and $\phi_{k} \to 1$ as $k\to \infty$ monotonically in compact sets of
  $\R^{\N}$. Then
  \begin{displaymath}
    \int_{\R^{\N}} \phi_{k} (x)
  \e^{-\delta |x|^{2}} \, \d |\tau_{y} \mu(x) | =     \int_{\R^{\N}} \phi_{k} (x)
  \e^{-\delta |x|^{2}} \, \d \tau_{y} |\mu(x) | =     \int_{\R^{\N}} \phi_{k} (x+y)
  \e^{-\delta |x+y|^{2}} \, \d |\mu(x) | .
  \end{displaymath}
Using (\ref{sqlower}), the right-hand side above is bounded by
\begin{displaymath}
    \e^{\delta  (\frac{1}{\alpha}-1) |y|^{2} }  \int_{\R^\N}
  \e^{-\delta (1-\alpha) |x|^{2}} \,  \d |\mu(x)| < \infty
\end{displaymath}
  for $\delta (1-\alpha) \geq  \eps$. Then Fatou's Lemma gives
$\D \int_{\R^{\N}}
  \e^{-\delta |x|^{2}} \, \d |\tau_{y} \mu(x) |  < \infty$.
Now the Monotone Convergence Theorem gives the result.
\end{proof}

Now we can prove the following result on the pointwise behaviour, as
$t\to T$ of the solution of the heat equation
(\ref{eq:solution_heat_up_t=0}) with initial data $0\leq u_{0}\notin
\mathcal{M}_{0}(\R^\N)$.

\begin{theorem} \label{thr:blow-up}
Assume that $0\leq u_{0} \in \mathcal{M}_\loc(\R^\N)$ and that the optimal
index $\eps_0=\eps_0(u_{0})$, see \eqref{eps0},  satisfies  $0<
\eps_{0} <\infty$. Then the solution $u$
of the heat equation given by \eqref{eq:solution_heat_up_t=0} is
not defined (at any point $x\in \R^\N$) beyond  $T:=1/4\eps_0$. Furthermore
as $t\to
T$
\begin{displaymath}
u(x,t) \to
\begin{cases}
u(x,T) & \mbox{if }\tau_{-x} u_{0} \in \mathcal{M}_{\eps_{0}}(\R^\N)
\cr
\infty  &  \mbox{if }\tau_{-x} u_{0} \notin
\mathcal{M}_{\eps_{0}}(\R^\N).
\end{cases}
\end{displaymath}

\end{theorem}
\begin{proof}
It follows from Lemma \ref{whyL1e} that if $u(x,t)$ is finite for
some $x\in \R^\N$ and $t> T$ then $u_{0} \in \mathcal{M}_{\eps}(\R^\N)$ for some $\eps$ with $\eps_{0} >
\eps>\frac{1}{4t}$, which is impossible.

To analyse the limiting behaviour as $t\to T$, using Lemma
\ref{lem:translation_of_measures} we first write, for
$t<T$, 
\begin{displaymath}
  u(x,t) =  \frac{1}{(4\pi t)^{\N/2}} \int_{\R^\N} \e^{-\frac{|x-y|^2}{4t}}
\, \d u_0(y)   = \frac{1}{(4\pi t)^{\N/2}} \int_{\R^\N} \e^{-\frac{|z|^2}{4t}}
\, \d \tau_{-x} u_0(z) .
\end{displaymath}

Now, if $\tau_{-x}u_{0} \notin \mathcal{M}_{\eps_{0}}(\R^\N)$ then  by Fatou's Lemma
\begin{displaymath}
  u(x,t) = \frac{1}{(4\pi t)^{\N/2}} \int_{\R^\N} \e^{-\frac{|y|^2}{4t}}
\,\d \tau_{-x}u_0(y)  \to \infty , \quad t\to T.
\end{displaymath}
On the other hand, if $\tau_{-x}u_{0} \in \mathcal{M}_{\eps_{0}}(\R^\N)$ then $\e^{-\frac{|y|^2}{4t}}
\leq \e^{-\eps_{0}|y|^2} $ and using the Monotone Convergence Theorem
it follows that as $t\to T$,
\begin{displaymath}
  u(x,t) = \frac{1}{(4\pi t)^{\N/2}} \int_{\R^\N} \e^{-\frac{|y|^2}{4t}}
\,\d \tau_{-x}u_0(y)  \to  \frac{1}{(4\pi T)^{\N/2}} \int_{\R^\N}
  \e^{-\eps_{0} |y|^2}
\,\d \tau_{-x}u_0(y) <\infty .\qedhere
\end{displaymath}
\end{proof}

For a given initial data  $0\leq u_{0} \in \mathcal{M}_\loc(\R^\N)$
with optimal index $0<\eps_0=\eps_0(u_{0}) <\infty$ we now analyse the
`regular set'   of points $x\in \R^{\N}$ such that the solution of the heat
equation has a finite limit as  $t\to T:=1/4\eps_0$ as in Theorem
\ref{thr:blow-up}. For short we define $A:=\eps_0(u_{0})>0$. Then observe
that if no translation of $u_{0}$ satisfies $\tau_{-x} u_{0} \in
\mathcal{M}_{\eps_{0}}(\R^\N)$ then  the
solution $u(x,t)$ of the heat equation diverges to infinity  at every point in $\R^\N$
as $t \to T=1/4A$. Otherwise, assume that $u_{0} \in
\mathcal{M}_{\eps_{0}}(\R^\N)$; then
\begin{equation} \label{eq:choosing_u0}
u_0(x)=\e^{A|x|^2} v(x),  \quad x\in \R^{\N}
\end{equation}
where  $0\le v\in \mathcal{M}_{\BTV}(\R^\N)$. If, on the contrary  $u_{0} \notin
\mathcal{M}_{\eps_{0}}(\R^\N)$ then for some $x_{0} \in \R^{\N}$ we
have $v_{0}:= \tau_{-x_{0}} u_{0} \in \mathcal{M}_{\eps_{0}}(\R^\N)$ and
then $v_{0}$ is as in (\ref{eq:choosing_u0}), while from Lemma
\ref{lem:translation_of_measures} we obtain
\begin{displaymath}
  u(x,t,u_{0}) = u(x-x_{0},t, v_{0}), \quad x\in \R^{\N}
\end{displaymath}
so it suffices to study the `regular set' of an initial data as in
(\ref{eq:choosing_u0}) with $0\le v\in \mathcal{M}_{\BTV}(\R^\N)$.
For simplicity in the exposition we will restrict to the case $0\le
v\in L^{1}(\R^\N)$.
In such a case we have the following
result that shows that the `regular set' of $x\in\R^\N$ at
which $u(x,t)$ has a finite limit as $t\to T$ must be a convex
set. For a converse result see Proposition \ref{anyconvex} below.

\begin{lemma} \label{lem:blowup-set}
Assume $u_{0}$ is as in \eqref{eq:choosing_u0} with $0\leq v \in
L^{1}_{loc}(\R^{\N})$, so that $\eps_0(u_{0})=A$. Then
\begin{enumerate}
\item[\rm (i)]
  $\tau_{-x} u_{0} \in L^{1}_{\eps_{0}}(\R^\N)$ iff
$I_v(x):= \D  \int_{\R^{\N}} \e^{2A \<x,z\>} v(z) \, \d z < \infty$.

\item[\rm (ii)] If moreover $0\le
v\in L^{1}(\R^\N)$ then the   set of $x\in \R^{\N}$
such that $I_v(x) < \infty$ is a convex set that contains  $x=0$.

\end{enumerate}
\end{lemma}
\begin{proof}
Since $\e^{-A|y|^{2}} \tau_{-x}  u_0 (y) =
\e^{A|x|^{2}} \e^{2A \<x,y\>} v(x+y) = \e^{-A|x|^{2}} \e^{2A \<x,x+y\>} v(x+y) $ part (i) follows.

For part (ii) note  that since $v\in L^1(\R^\N)$, it is always the case that
$$
\int_{\<x,y\>\le0}\e^{\lambda 2A\<x,y\>}v(y)\,\d y<\infty
$$
whenever $\lambda\ge0$.

Now observe that if $I_v(x)=\infty$ then for any $\lambda\ge1$
$$
\int_{\<x,y\> > 0} \e^{\lambda 2A
  \<x,y\>} v(y) \, \d y = \infty,
  $$
  i.e.\ $I_v(\lambda x)= \infty$ for any $\lambda \geq 1$.

Consider now $x_{1}, x_{2}$ such that $I_v(x_{i}) <  \infty$ and take $\theta
\in (0,1)$. Then  $I_v(\theta x_{i}) <  \infty$ and  $I_v((1-\theta) x_{i}) <  \infty$ and
\begin{displaymath}
  I_v(\theta x_{1}+ (1-\theta) x_{2}) =  \int_{\R^{\N}} \e^{2A
    \theta\<x_{1},y\>} \e^{2A (1-\theta)\<x_{2},y\>}  v(y) \, \d y .
\end{displaymath}
Observe that the integral in the regions where either
$\<x_{1},y\> \leq 0$ or  $\<x_{2},y\> \leq 0$ is finite while
\begin{displaymath}
  \int_{\{ \<x_{1},y\> >0, \ \<x_{2},y\> >0\} } \e^{2A    \theta\<x_{1},y\>} \e^{2A (1-\theta)\<x_{2},y\>}  v(y) \, \d y
\end{displaymath}
can be split in the regions $\<x_{1},y\> \geq  \<x_{2},y\>$ and
$\<x_{1},y\> < \<x_{2},y\>$. In both of these the integral is finite,
which completes the proof.
\end{proof}

Now we give some examples of initial data as in (\ref{eq:choosing_u0})
and explicitly compute its regular set.
In our  first example  $I_v(x)$ is never finite so we obtain complete blowup, generalising the example in (\ref{eq:blowing_up_quadratic_exponential}).

\begin{example}
  If we take $v(x) \geq c >0$ for all $x\in \R^{\N}$ then
  $\eps_0(u_0)=A$ and in \eqref{eq:choosing_u0}  we have
  $\tau_{-x}u_0\notin L^1_A(\R^\N)$ for every $x\in\R^\N$.
  \end{example}

In this case the solution $u$ of the heat equation given by
(\ref{eq:solution_heat_up_t=0}) and initial data
(\ref{eq:choosing_u0})  blows up at every point in $\R^{\N}$ at
time $T=\frac{1}{4A}$.

In our next example $I_v(x)$ is only finite at $x=0$ so the regular
set consists of a single point at the origin. Thus  the solution $u$ of the heat equation given by
(\ref{eq:solution_heat_up_t=0}) has a finite limit at $x=0$  as $t\to T$,
but blows up at all other points of $\R^\N$.

\begin{example}
  When $v(x) = (1+|x|^{2})^{-\alpha/2}$ with $\alpha >
\N/2$ we have  $\eps_{0}(u_0) =A$ and  in (\ref{eq:choosing_u0})  we have  $\tau_{-x}  u_0 \in
L^{1}_{A}(\R^\N)$ only when $x=0$.
\end{example}

To see this we write $y = sx + y'$ with $y'\perp x$ to get
\begin{displaymath}
  I(x)= \int_{\R^{\N-1}}  \int_{-\infty}^{\infty}  \frac{\e^{2A|x|^{2}
    s}}{(1+s^{2}|x|^{2}
  + |y'|^{2})^{\alpha/2}}   \,  \d s\,  \d y'
\end{displaymath}
If $x\neq 0$ then the integral in $s$ is infinity for each $y'\in \R^{N-1}$.

In the next example $I_v(x)$ is finite only in the open ball
$|x|<\gamma/2A$; the solution $u$ of the heat equation given by
(\ref{eq:solution_heat_up_t=0}) has a finite limit here as $t\to T$,
but blows up at all other points of $\R^\N$.

\begin{example}
  Take  $v(x) = \e^{-\gamma |x| }$ with
$\gamma >0$. Then $\eps_{0}(u_0) =A$ and  in (\ref{eq:choosing_u0})  we have  $\tau_{-x}  u_0 \in
L^{1}_{A}(\R^\N)$ if and  only if  $|x| <\frac{\gamma}{2A}$.
\end{example}

To see this note first that
\begin{displaymath}
0\leq   \e^{2A \<x,y\>} v(y)  \leq   \e^{(2A  |x|-\gamma )|y|}
\end{displaymath}
which is integrable if $|x|< \frac{\gamma}{2A}$.
On the other hand writing $y = sx + y'$ with $y'\perp x$
\begin{displaymath}
I(x)=
\int_{\R^{\N-1}}  \int_{-\infty}^{\infty} \e^{2A |x|^{2}s}  \e^{-\gamma
  \sqrt{s^{2} |x|^{2} +|y'|^{2}}} \,  \d s\, \d y'.
\end{displaymath}
If $2A |x|^{2} -\gamma |x| \geq 0$, that is, $|x| \geq \frac{\gamma}{2A}$,
then the integral in $s$ is infinite for each $y'\in \R^{\N-1}$.

It is also possible to make  $I_{v}(x)$  finite  only on  a closed ball.

\begin{example}
  Take  $v(x) = \e^{-\gamma |x| } (1+|x|^{2})^{-\alpha/2}$ with $\alpha >
\N/2$ and  $\gamma >0$. Then $\eps_{0}(u_0) =A$ and  in (\ref{eq:choosing_u0})  we have $\tau_{-x}  u_0 \in
L^{1}_{A}(\R^\N)$ if and  only if  $|x| \leq \frac{\gamma}{2A}$.
\end{example}

To see this note first that
\begin{displaymath}
0\leq   \e^{2A \<x,y\>} v(y)  \leq   \e^{(2A  |x|-\gamma )|y|} (1+|y|^{2})^{-\alpha/2}
\end{displaymath}
which is integrable if $|x|\leq  \frac{\gamma}{2A}$. 
On the other hand, if $x\neq 0$ writing $y = sx + y'$ with $y'\perp x$
\begin{displaymath}
I_{v}(x)=
\int_{\R^{\N-1}}  \int_{-\infty}^{\infty} \e^{2A |x|^{2}s}  \e^{-\gamma
  \sqrt{s^{2} |x|^{2} +|y'|^{2}}}
\frac{1}{(1+s^{2}|x|^{2}
  + |y'|^{2})^{\alpha/2}} \,  \d s\, \d y'.
\end{displaymath}
If $2A |x|^{2} -\gamma |x| > 0$, that is, $|x| > \frac{\gamma}{2A}$,
then the integral in $s$ is infinite for each $y'\in \R^{\N-1}$.

Our last example is perhaps the most striking: here $I_v(x)$ is finite for all $x\in \R^{\N}$.

\begin{example}\label{howodd}
  Take  $v(x) = \e^{-\gamma |x|^{\alpha}}$ with
$\gamma >0$, $1<\alpha <2$. Then $\eps_{0}(u_0) =A$ and  in
(\ref{eq:choosing_u0})  we have   $\tau_{-x}  u_0 \in
L^{1}_{A}(\R^\N)$   for any $x\in \R^\N$.
\end{example}

To see this note that
\begin{displaymath}
 0\leq   \e^{2A \<x,y\>} v(y)  =  \e^{2A <x,y> -\gamma
   |y|^{\alpha}} \leq \e^{2A |x| |y| -\gamma   |y|^{\alpha}} \in L^{1}(\R^\N).
\end{displaymath}
Thus for the initial data $u_0(x)=\e^{A|x|^2-\gamma|x|^\alpha}$ the
solution $u(x,t)$ of the heat equation takes a finite value at every point in $\R^\N$
at $T=1/4A$, but cannot be continued beyond this time.

We now show that in fact we can arrange for the regular set,  $\{x:\
I_v(x)<\infty\}$,  to be any chosen closed convex subset of $\R^\N$. First we recall the following characterisation of such sets.

\begin{lemma} \label{lem:closed_convex_set}
  Any closed convex set  with $0 \in K$  is of the form
  $$
  K= \bigcap_{j\in J} \{x:\  \<x, n_j\> \le c_j\}
  $$
  for some unit vectors $n_{j}$ and
  $c_{j} \geq 0$,   where $J$ is at most countable.

\end{lemma}
\begin{proof}
Note first  that $K$ is the intersection of all closed half spaces containing
$K$. Then observe that   $K= \bigcap_{y   \in \partial K} \{x:\ \<x,
  n(y)\> \le c(y)\}$  for some unit vectors
$n(y)$ and constants $c(y) \geq 0$, see e.g.   \cite{S}.

This implies that, $\bigcup_{y   \in \partial K} \{x:\ \<x,
  n(y)\> > c(y)\}$  is an open  covering of the open set
  $\R^{\N} \setminus K$. Thus we can extract an, at most,  countable
covering.
\end{proof}

Note that the form of $K_0$ in the following  results is more general
than that given by the previous lemma; in particular it allows for any
closed convex set.

\begin{proposition}\label{anyconvex}
Assume that  $K \subset \R^{\N}$ is a  convex set  given by the intersection of at
most a countable number of half spaces, that is, $K=
x_{0} + K_{0}$ with
\be{convexK}
0\in K_{0}= \bigcap_{j\in J_{1}} \{x:\
  \<x, n_j\> \leq c_j\} \cap \bigcap_{j\in J_{2}} \{x:\
  \<x, n_j\> < c_j\}
\ee
 for some unit vectors $n_{j}$ and
  $c_{j} \geq 0$ for $j\in J_{1}$ and $c_{j} >0$ for $j\in J_{2}$,
  where $J_{1}$ and $J_{2}$ are  at most countable and $\inf_{j\in
    J_{2}} c_{j} >0$.

Then    there exist $0\leq v \in
L^{1}_{\loc}(\R^{\N})$ such that
\begin{displaymath}
  I_{v}(x) < \infty \quad \mbox{if and only if }x\in K:
\end{displaymath}
the solution $u$ of the heat equation with initial data $u_{0}(x)= \e^{A|x|^2}v(x)$
 has a finite limit at every point $x \in
K$ but blows up at every other point in $\R^{\N}$  as $t\to \frac{1}{4A}$.

\end{proposition}
\begin{proof}
Assume first that $x_{0}=0$.
Take any orthonormal  basis $\mathcal{B} = \{e_{j}\}_{j}$ of
$\R^{\N}$. Using coordinates with respect to this basis, write
$x\in\R^{\N}$ as $(x_{1},x')$ and $y=(y_{1},y')$. We choose
$\eta_0\geq 0$
and set
\begin{displaymath} 
0\leq v(x)=\begin{cases}\chi(x')\phi(x_{1})\e^{-\eta_0 x_{1}},&x_{1}>0\\
0&x_{1}<0,
\end{cases}
\quad
0\leq w(x)=\begin{cases}\chi(x') \e^{-\eta_0 x_{1}},&x_{1}>0\\
0&x_{1}<0,
\end{cases}
\end{displaymath}
where $\chi$ is the characteristic function of the unit ball in $\R^{\N-1}$ and
$\phi(s)=\frac{1}{1+s^2}$.
Note that $v \in L^{1}(\R^{\N})$ with $\|v\|_{ L^{1}(\R^{\N})} \leq M$ and
if $\eta_{0} >0$ then $w \in L^{1}(\R^{\N})$ with $\|w\|_{
  L^{1}(\R^{\N})} \leq \frac{M}{\eta_{0}}$ with $M$ independent
of $\eta_{0}$.  Also,  for $x\in \R^{\N}$
\begin{align*}
I_{v}(x) = \int_{\R^{\N}} \e^{2A\<x, y\>} v(y) \,\d y
& = \int_0^\infty \int_{|y'|\le  1} \e^{(2A x_{1}- \eta_0)  y_{1}} \e^{2A \<x', y' \>}
   \phi(y_{1}) \,\d y' \,\d   y_{1} \\
& = \left( \int_{|y'|\le 1}\e^{2A \< x',  y'\>} \,\d  y'\right) \left(
  \int_0^\infty \e^{ (2A x_{1}-\eta_0) y_{1}}\phi(y_{1})\,\d y_{1}\right).
\end{align*}
The first factor is always finite, and the second is finite if
$x_{1} \leq  \frac{\eta_{0}}{2A}$ and infinite if $x_{1} >
\frac{\eta_{0}}{2A}$. So choosing $\eta_{0} \geq 0$ appropriately, for
 any given  $c\geq 0$ we
can ensure that $I_{v}(x) <  \infty$ iff $x\in \{z:\ z_{1}
\leq c\}$.

An analogous computation with $w$ gives that  for
 any given  $c > 0$ we obtain $I_{w}(x) <  \infty$ iff $x\in \{z:\ z_{1}
< c\}$.

Hence for any given unit vector $n$ and $c\geq  0$ (respectively
$c>0$) we can find an
integrable function $v=v(n,c)$ ($w=w(n,c)$ respectively)  such that
\begin{displaymath}
I_{v}(x) <\infty \quad \mbox{ only in the half space  $\<x, n\> \leq c$ \quad  ($\<x, n\> <  c$ respectively)}
\end{displaymath}
 with $\|v\|_{ L^{1}(\R^{\N})}$ bounded
independent of $n$ and $c\geq 0$ ($\|w\|_{ L^{1}(\R^{\N})} \leq
\frac{M}{c}$, $M$ independent of $n$).

Based on the assumed form of $K_{0}$ in \eqref{convexK} 
 we  set
 \begin{displaymath}
v_{0}(x)=\sum_{j \in J_{1}} j^{-2}v_j(x) + \sum_{j \in
  J_{2}}  j^{-2}w_j(x)
 \end{displaymath}
where $v_j(x) =
v(n_{j}, c_{j})(x)$ and $w_j(x) =
w(n_{j}, c_{j})(x)$  are
constructed as above. Since $\inf_{j\in
    J_{2}} c_{j} >0$ then  $v_{0} \in L^{1}(\R^{\N}) $ and clearly   $I_{v_{0}}(x)
  <\infty$ iff $x\in K_{0}$.

Now, for $x_{0} \neq 0$, define
\begin{displaymath}
  v(x) = e^{-2A \<x_{0},x\>} v_{0}(x), \quad x\in \R^{\N} .
\end{displaymath}
Then $I_{v}(x) = \D \int_{\R^{\N}} \e^{2A \<x-x_{0}, y\>} v_{0}(y) \, \d y <\infty$ iff $x-x_{0}\in
K_{0}$.
\end{proof}

\subsection{Continuation of signed solutions}
\label{sec:cont-sign-solut}

For signed solutions the maximal existence time of the solution may not be
given by $T= \frac{1}{4\eps_{0}(u_{0})}$ as in Theorem \ref{thr:blow-up}; see
Section \ref{sec:prescribed_g_at_x=0}. However, we can establish the
following continuation result.

\begin{proposition} \label{prop:prolongation}

Assume that $u_{0} \in \mathcal{M}_{\eps}(\R^\N)$ and that $u(t;u_{0}) = S(t) u_{0}$ given by
(\ref{eq:solution_heat_up_t=0}) is defined on $[0,T)$ but cannot be
defined any time after. Then for any $\delta >0$
\begin{displaymath}
  \limsup_{t\to T}  \|u(t,u_{0})\|_{ L^1_{\delta}(\R^\N)}  = \infty .
\end{displaymath}

\end{proposition}
\begin{proof}
  Assume otherwise that for some $\delta >0$ (which, without loss of
  generality, we  can take such that $\delta > \eps$)
\begin{displaymath}
  \|u(t,u_{0})\|_{ L^1_{\delta}(\R^\N)} \leq M, \quad 0\leq t < T.
\end{displaymath}
Take $t_{0}
<T$ such that $T < t_{0} + T(\delta)/2$ and define $v_{0} = S(t_{0})
u_{0} \in  L^1_{\delta}(\R^\N)$.
Then define
\begin{displaymath}
  U(t) =
  \begin{cases}
    S(t)u_{0} & 0\leq t < t_{0} \\
S(t-t_{0}) v_{0} & t_{0}\leq t < t_{0} + T(\delta) ,
  \end{cases}
\quad 0\leq t <  t_{0} + T(\delta)  .
\end{displaymath}
Then we claim  that $U$ satisfies   assumptions
(\ref{eq:uniqueness_Du_in_L1delta}),
(\ref{eq:uniqueness_u_in_L1eps}),
(\ref{eq:uniqueness_initial_data_Cc})
in $[0, t_{0} +  T(\delta)) $.
Hence, Theorem \ref{thr:uniqueness} implies $U(t) = S(t) u_{0}$ for
$0\leq t <  t_{0} +  T(\delta)$ which
contradicts the maximality of $T$.

To prove the claim, notice that (\ref{eq:uniqueness_initial_data_Cc})  is
satisfied (using Theorem \ref{thm:properties_sltns_given_u0}).
Also, (\ref{eq:uniqueness_u_in_L1eps}) holds because of the assumption
on $u$ and by part (i) in Proposition \ref{prop:semigroup_estimates}
applied to $S(t-t_{0}) v_{0}$ for   $t_{0}\leq t \leq t_{0} + \tau$ for
any $\tau <T(\delta)$. Finally  (\ref{eq:uniqueness_Du_in_L1delta})
follows from (\ref{bound_above_x=0})
and (\ref{eq:exponential_bound_further_derivatives}) applied to
$S(t)u_{0}$ with  $0\leq t < t_{0}$ and $S(t-t_{0}) v_{0}$ with
$t_{0}\leq t \leq t_{0} +\tau$ for any $\tau <  T(\delta)$.
\end{proof}

\section{Long-time behaviour of heat solutions}
\label{sec:asympt-behav-heat}
\setcounter{equation}{0}

We now discuss the asymptotic behaviour as $t\to\infty$ of solutions when
$u_{0} \in \mathcal{M}_{0}(\R^\N)$. For this  Lemma
\ref{lem:estimate_from_x=0} will be a central tool.
We start with some simple consequences of this result, which show that
the asymptotic behaviour is largely determined by the behaviour at
$x=0$. Observe that the converse of parts  (i), (ii), (iii), and (v) are
obviously true.

\begin{proposition} \label{prop:estimates_from_u(0;t)}
Assume that $u_{0} \in \mathcal{M}_{0}(\R^\N)$.
\begin{itemize}
\item[\rm(i)] If $u(0,t,|u_{0}|)$ is bounded for $t>0$ then $u(t,u_{0})$
  remains uniformly bounded
in sets $\frac{|x|}{\sqrt{t}} \leq R$. In particular, $u \in
L^{\infty}_\loc(\R^\N \times (0,\infty))$.

\item[\rm(ii)] If $u(0,t,|u_{0}|) \to 0$ as $t\to \infty$ then $u(t,u_{0})\to 0$ uniformly in
sets $\frac{|x|}{\sqrt{t}} \leq R$.
\end{itemize}
Assume in addition that $u_0\ge0$. Then
\begin{itemize}
\item[\rm(iii)] If $\lim_{t\to \infty} u(0,t,u_{0})=L \in (0,\infty)$ exists,
then $u(x,t, u_{0}) \to L$ uniformly in compact sets of $\R^\N$.

\item[\rm(iv)] If $u(0,t,u_{0})$ is unbounded for $t>0$ then $u(t,u_{0})$ is
unbounded in sets $\frac{|x|}{\sqrt{t}} \leq R$, and so in particular unbounded in  any compact subset of $\R^\N$.

\item[\rm(v)] If $u(0,t, u_{0}) \to \infty$ as $t\to \infty$ then $u(t,u_{0})\to
\infty$ uniformly in sets $\frac{|x|}{\sqrt{t}} \leq R$.
\end{itemize}
\end{proposition}

\begin{proof}
(i) and (ii).  From the upper bound in Lemma \ref{lem:estimate_from_x=0}  in sets
  with $\frac{|x|^{2}}{t} \leq R$ we get
  \begin{displaymath}
  |  u(x,t) |\leq c_{\N,a} u(0,at,|u_{0}|) \e^{\frac{|x|^{2}}{4(a-1)t}}  \leq
    c_{\N,a} u(0,at,|u_{0}|) \e^{\frac{R^{2}}{4(a-1)}}  .
  \end{displaymath}

Assume furthermore that $0\leq u_{0} \in \mathcal{M}_{0}(\R^\N)$. Then

\noindent  (iii) Using the lower and upper bounds   in Lemma
\ref{lem:estimate_from_x=0}  if  $|x|^{2}\leq R$ we get  for every $b<1<a$
  \begin{displaymath}
   b^{\N/2}   u(0,bt) \e^{-\frac{R^{2}}{4(1-b)t}}  \leq u(x,t)
   \leq  a^{\N/2} u(0,at) \e^{\frac{R^{2}}{4(a-1)t}}  .
  \end{displaymath}

\noindent (iv) and (v) From the lower bound  in Lemma
\ref{lem:estimate_from_x=0} 
in sets  with $\frac{|x|^{2}}{t} \leq R$ we get
$$
c_{\N,b} u(0,bt) \e^{-\frac{R^{2}}{4(1-b)}} \leq   \inf_{\frac{|x|^2}{t}\leq R} u(x,t).\eqno\qedhere
$$
\end{proof}

Recalling the definition of the $\mathcal{M}_\eps(\R^{\N})$ norm (\ref{eq:norm_Meps}) observe that
\begin{equation} \label{eq:value_at_zero_and_normMeps}
  u(0,t,|u_{0}|)=\frac{1}{(4\pi t)^{\N/2}} \int_{\R^\N} \e^{-\frac{|y|^2}{4t}}
\,\d |u_0(y)|  = \|u_{0}\|_{\mathcal{M}_{1/4t}(\R^\N)}
\end{equation}
hence Proposition \ref{prop:estimates_from_u(0;t)} could easily be restated in
terms of the behavior of the norms
$\|u_0\|_{\mathcal{M}_\eps(\R^{\N})}$ as $\eps\to0$. In particular  $u(t,u_{0})$
  remains uniformly bounded in sets $\frac{|x|}{\sqrt{t}} \leq R$ if
  and only if
\begin{displaymath}
   \sup_{\eps >0}  \| u_{0} \|_{ \mathcal{M}_{\eps} (\R^{\N})} < \infty
   .
\end{displaymath}

Also notice that part (iii) implies that there are no other stationary
solutions of (\ref{eq:intro_heat_equation}) other than constants. In other
words, a harmonic function in $\mathcal{M}_{0}(\R^\N)$ must be
constant.

\subsection{Sufficient conditions for decay}
\label{sec:suff-cond-decay}

As observed above,  solutions converge to zero as $t\to\infty$ if and only if
$$
\lim_{\eps\to0}\|u_0\|_{\mathcal{M}_\eps (\R^\N)  }=0.
$$
We now give some (non-sharp) conditions to ensure this, in terms of
the distribution of mass of the initial condition  measured in terms of the
averages over balls. Note that from
Proposition \ref{prop:estimates_from_u(0;t)} this behaviour is
determined by the value of the solution at $x=0$.

\begin{theorem}\label{astti}
Suppose that $u_0\in \mathcal{M}_0(\R^\N)$.
\begin{itemize}
\item[\rm(i)] If
$$
\frac{1}{R^d}\int_{R/2\le|x|\le R} \,\d  |u_0(x)|\le M
$$
then $u$ remains uniformly bounded
in sets of the form $\frac{|x|}{\sqrt{t}} \leq R$ for any $R>0$; in particular, $u \in
L^{\infty}_\loc(\R^\N \times (0,\infty))$.

\item[\rm(ii)] If
$$
\lim_{R\to\infty}\frac{1}{R^d} \int_{R/2\le|x|\le R} \,\d  |u_0(x)|=0;
$$
then $u(0,t) \to 0$ as $t\to \infty$ and hence
 $ u(t) \to 0$ in $L^{\infty}_\loc(\R^\N)$
and uniformly in sets $\frac{|x|}{\sqrt{t}} \leq R$.

\end{itemize}

Assume in addition that $u_0\ge0$. Then
\begin{itemize}
\item[\rm(iii)] If $$
\liminf_{R\to\infty} \frac{1}{R^d}\int_{|x|\le R}\,\d u_0(x) >0,
$$
then $\liminf_{t
  \to \infty} u(0,t) >0$.

\item[\rm(iv)] If
$$
\lim_{R\to\infty}\frac{1}{R^d}\int_{|x|\le R} \,\d u_0(x) =\infty
$$
then $u(0,t)\to\infty$ as $t\to\infty$.
\end{itemize}
\end{theorem}

\begin{proof} %

(i) Consider
\begin{displaymath}   
  |u(0,t)|=\frac{1}{(4\pi t)^{d/2}} \int_{\R^d}\e^{-|x|^2/t} \,\d
          |u_0(x)| =\frac{1}{(4\pi t)^{d/2}} \sum_{k=-\infty}^\infty \int_{2^k\le|x|\le
  2^{k+1}} \e^{-|x|^2/t} \,\d |u_0(x)| .
\end{displaymath} 
Note that
\begin{align*}
\int_{R\le |x|\le 2R} \e^{-|x|^2/4t} \,\d |u_0(x)| & \le \e^{-R^2/4t}
                                     \int_{R\le|x|\le 2R} \,\d  |u_0(x)|\\
&\le(2R)^dM \e^{-R^2/4t} \le 2^{d+1} M \int_{R/2\le |x|\le R}
  \e^{-|x|^2/4t} \, \d x ,
\end{align*}
and so
\begin{align*}
|u(0,t)|&\le 2^{d+1}M \, \frac{1}{(4\pi t)^{d/2}}
          \sum_{k=-\infty}^\infty \int_{2^{k-1}\le|x|\le
          2^k}\e^{-|x|^2/4t} \, \d x\\
&= 2^{d+1}M\,\frac{1}{(4\pi t)^{d/2}} \int_{\R^d}\e^{-|x|^2/4t } \, \d
  x = 2^{d+1}M.
\end{align*}

(ii) Given $\eps>0$ take $k_0>0$ such that
$$
\frac{1}{R^d} \int_{R/2\le|x|\le R} \, \d |u_0(x)| <\eps
$$
for all $R\ge R_0:=2^{k_0}$.

Then we can split the domain of integration in the integral expression
for $|u(0,t)|$ into $|x|\le R_0$ and $|x|>R_0$. The above argument
shows that the integral over the unbounded region contributes at most
$2^{d+1}\eps$ for all $t>0$, while the integral over the bounded
region contributes no more than
$$
\frac{1}{(4\pi t)^{d/2}} \int_{|x|\le R_0} \, \d |u_0(x)| \le \frac{c}{t^{d/2}}.
$$
It follows that $|u(0,t)|\to0$ as $t\to\infty$.

(iii) There exists an $R_0$ and $m>0$ such that
$$
\frac{1}{R^d} \int_{|x|\le R} \, \d u_0(x) \ge m>0
$$
for all $R\ge R_0$. Then for $t$ sufficiently large such that $\sqrt t>R_0$ we have
\begin{displaymath}
  u(0,t)\ge\frac{1}{(4\pi t)^{d/2}} \int_{|x|\le\sqrt t} \e^{-|x|^2/4t}
        \, \d u_0(x) \ge\frac{1}{(4\pi t)^{d/2}}\e^{-1/4}mt^{d/2}=\frac{1}{(4\pi)^{d/2}}\e^{-1/4} m.
\end{displaymath}

(iv) We repeat the above argument, taking $m$ arbitrary.\end{proof}

\subsection{Wild behaviour of solutions as $t\to\infty$.}
\label{sec:wild-behav-solut}

Theorem \ref{astti} gives conditions on the averages over balls to distinguish between various time-asymptotic regimes; but in the case that
$$
\liminf_{R\to\infty} \frac{1}{R^d}\int_{|x|\le R} \, \d u_0(x) >0
$$
and
$$
\frac{1}{R^d} \int_{|x|\le R}  \, \d u_0(x) \not\to\infty \qquad\mbox{as}\quad R\to\infty
$$
some very rich behaviour is possible.

In the following theorem we show that there is initial data in
$L^1_0(\R^\N)$ that gives rise to unbounded oscillating solutions. For
bounded oscillations in the case of bounded initial data and solutions, see
also Section \ref{sec:rescalling-appr-vazq} below and \cite{VZ2002}.

\begin{theorem} \label{thr:oscillation_at_0}

For any sequence of non-negative numbers $\{\alpha_{k}\}_{k}$ there exists a non-negative
 $u_{0} \in L^{1}_{0}(\R^\N)$ and a sequence
$t_{n}\to \infty$ such that for every $k$ there exists a subsequence
$t_{k,j}$ with  $u(0,t_{k,j}) \to \alpha_{k}$ as $j\to \infty$.

\end{theorem}

In fact  it is enough to prove the following  (apparently weaker)
result.

\begin{proposition} \label{prop:oscillations}
  Given any  sequence of non-negative numbers $\{b_{k}\}_{k}$ there exists a non-negative
 $u_{0} \in L^{1}_{0}(\R^\N)$ and a sequence $t_{k}\to \infty$ such
 that, as $k\to \infty$
 \begin{displaymath}
   |u(0,t_{k}) -b_{k}| \to 0 .
 \end{displaymath}
\end{proposition}

Indeed, given a sequence $\{\alpha_{k}\}_{k}$ as in the statement of Theorem \ref{thr:oscillation_at_0} construct the sequence $\{b_{k}\}_{k}$ as
\begin{displaymath}
  \alpha_{1} | \alpha_{1} ,\alpha_{2} |  \alpha_{1},  \alpha_{2} ,
  \alpha_{3} |  \alpha_{1},\ldots,  \alpha_{4} |   \alpha_{1}, \ldots,
  \alpha_{5} |  \ldots .
\end{displaymath}
Now apply Proposition \ref{prop:oscillations} and note that for any $k\in \NN$ there is
a subsequence ${k_j}$ such that
$b_{k_{j}} = \alpha_{k}$.

 We now prove Proposition \ref{prop:oscillations}, inspired by the
 proof of Lemma 6 in \cite{VZ2002}.

\begin{proof}
Observe, with
Proposition \ref{prop:estimates_from_u(0;t)} in mind, that we can
write
\begin{displaymath}
  u(0,t) = \frac{1}{(4\pi t)^{\N/2}} \int_{\R^\N} \e^{-\frac{|y|^2}{4t}}
u_0(y) \,\d y = \frac{1}{\pi^{\N/2}} \int_{\R^\N} \e^{-|z|^2}
u_0(2\sqrt{t} z) \, \d z .
\end{displaymath}
So if $\lambda_{n} \to \infty$,
\begin{equation} \label{eq:u(0,tn)}
  u(0,t_{n}) = \frac{1}{\pi^{\N/2}} \int_{\R^\N} \e^{-|z|^2}
u_0(\lambda_{n} z) \, \d z .
\end{equation}
with $\lambda_{n} = 2\sqrt{t_{n}}$. We set $c_{\N} =:
\frac{1}{\pi^{\N/2}}$.

 Consider for $r>1$  the annulus
 $A(r) =\{y, \ r^{-1} < |y|< r\}$
and given the sequence  $\{b_{k}\}_{k}$ consider a
  function
  \begin{displaymath}
    u_{0}(x) = \sum_{j} b_{j}  \Chi_{\lambda_{j} A(r_{j}) }(x)
  \end{displaymath}
for some increasing and divergent sequences $\{\lambda_{k}\}_{k}$, $\{r_{k}\}_{k}$
chosen recursively as follows: 
first we  choose $r_{k}$ large
with respect to the sequence $\{b_k\}_{k}$, according to
\begin{equation} \label{eq:condition_4_near_origin}
2^{k}  \beta_{k-1} < r_{k-1}^{\N}
\end{equation}
where $\beta_{k} =\max_{1\leq j\leq k} \{ b_{k}\}$  and
\begin{equation} \label{eq:condition_4_mid_term}
  b_{k} c_{\N}\int_{\R^\N \setminus A(r_{k})} \e^{-|x|^2} \,\d x <
  2^{-k} .
\end{equation}
Then choose $\lambda_{k}$ sufficiently large that
\begin{equation} \label{eq:condition_4_u0_in_L10}
2^{k}   b_{k}   r_{k}^{3\N} < \lambda_{k}^{\N}
\end{equation}
and
\begin{equation}  \label{eq:condition_4_far_away}
  b_{k}   \exp\left(-\frac{\lambda_{k}^{2}}{\lambda_{k-1}^{2}r_{k}^{2}}\right)
  \lambda_{k}^{\N}  r_{k}^{\N}  < 2^{-k} .
\end{equation}
Finally choose the
next value $\lambda_{k+1}$ large enough that
\begin{equation} \label{eq:disjoint_annulus}
  \lambda_{k} r_{k} < \frac{\lambda_{k+1}}{r_{k+1}} .
\end{equation}

\noindent {\it Step 1.} Observe that from (\ref{eq:disjoint_annulus})
the scaled annulae $\lambda_{j} A(r_{j})$ are disjoint and increasing.

\noindent {\it Step 2.} Now we prove that from
(\ref{eq:condition_4_u0_in_L10}), 
we get  $u_{0} \in L^{1}_{0}(\R^\N)$.
For this take  any $\eps >0$ and then
\begin{displaymath}
   \int_{\R^\N} \e^{-\eps|x|^{2}} u_{0}(x) \,\d x = \sum_{j} b_{j}
   \int_{\lambda_{j} A(r_{j}) } \e^{-\eps |x|^{2}} \,\d x \leq  \sum_{j}
   b_{j} \e^{-\eps \lambda_{j}^{2}/r_{j}^{2}} \lambda_{j}^{\N}
   r_{j}^{\N} .
\end{displaymath}
Now for any $m\in \NN$ there exist $R_{\eps}, c_{\eps}$ (depending on
$m$ as well)  such that if $z\geq R_{\eps}$ then
\begin{displaymath}
  \e^{-\eps z} \leq \frac{c_{\eps}}{z^{m}} .
\end{displaymath}
Since from (\ref{eq:disjoint_annulus}) we get
$\frac{\lambda_{k}}{r_{k}}  \to \infty$, then for some $j_{0} \in
\NN$,  we get
\begin{displaymath}
   \sum_{j\geq j_{0}}
   b_{j} \e^{-\eps \lambda_{j}^{2}/r_{j}^{2}} \lambda_{j}^{\N}
   r_{j}^{\N} \leq c_{\eps} \sum_{j\geq j_{0}}   b_{j}
   \frac{r_{j}^{2m+\N}}{\lambda_{j}^{2m-\N}} .
\end{displaymath}
For example with $m=\N$ we get, by (\ref{eq:condition_4_u0_in_L10}),
\begin{displaymath}
c_{\eps} \sum_{j\geq j_{0}}   b_{j}
   \frac{r_{j}^{3\N}}{\lambda_{j}^{\N}} \leq c_{\eps} \sum_{j\geq j_{0}}
   2^{-j} < \infty .
\end{displaymath}

\noindent {\it Step 3.} Now we prove that from
(\ref{eq:condition_4_near_origin}), (\ref{eq:condition_4_mid_term})
and (\ref{eq:condition_4_far_away}) 
then $|u(0,t_{k}) -b_{k}| \to 0$.

For this  observe that for any $\lambda>0$
 \begin{displaymath}
    u_{0}(\lambda x) = \sum_{j} b_{j}  \Chi_{\lambda_{j}\lambda^{-1} A(r_{j}) }(x)
  \end{displaymath}
and then for each $k$ we have in (\ref{eq:u(0,tn)})
\begin{align}
  \int_{\R^\N} \e^{-|z|^2} u_0(\lambda_{k} z) \, \d z =
  \sum_{1\leq j\leq k-1} b_{j}
   \int_{{\lambda_{j}}{\lambda_{k}^{-1}} A(r_{j}) }& \e^{- |z|^{2}}
   \, \d z +  b_{k}
   \int_{A(r_{k}) } \e^{- |z|^{2}}
   \, \d z\nonumber\\
   & + \sum_{j\geq k+1} b_{j}
   \int_{{\lambda_{j}}{\lambda_{k}^{-1}} A(r_{j}) } \e^{- |z|^{2}}
   \, \d z. \label{eq:2_estimate_u(0)}
\end{align}

Then the first term in (\ref{eq:2_estimate_u(0)})  is bounded by
\begin{displaymath}
\beta_{k-1}   \frac{\lambda_{k-1}^{\N}}{\lambda_{k}^{\N}} r_{k-1}^{\N}
<\frac{\beta_{k-1}} {r_{k-1}^{\N}} < 2^{-k}
\end{displaymath}
by (\ref{eq:condition_4_near_origin}), where we used (\ref{eq:disjoint_annulus}).

For the second term in (\ref{eq:2_estimate_u(0)}) observe that by (\ref{eq:condition_4_mid_term})
\begin{displaymath}
\left | b_{k}   \int_{A(r_{k}) } \e^{- |z|^{2}}   \, \d z -b_{k} \right |
< \frac{1}{2^{k} c_{\N}}   .
\end{displaymath}

Finally,  observe that the third term in (\ref{eq:2_estimate_u(0)}) is bounded by
\begin{displaymath}
\sum_{j\geq k+1} b_{j}   \exp\left(-
  \frac{\lambda_{j}^{2}}{\lambda_{k}^{2}r_{j}^{2}}\right)
  \frac{\lambda_{j}^{\N}}{\lambda_{k}^{\N}} r_{j}^{\N} \leq \sum_{j\geq k+1} b_{j}   \exp\left(-
  \frac{\lambda_{j}^{2}}{\lambda_{j-1}^{2}r_{j}^{2}}\right)
  \lambda_{j}^{\N}  r_{j}^{\N}  \leq \sum_{j\geq k+1} 2^{-j}
\end{displaymath}
by (\ref{eq:condition_4_far_away}).

Then, from (\ref{eq:u(0,tn)}) and the bounds above on the three terms
in (\ref{eq:2_estimate_u(0)}) we get, with $\lambda_{k} = 2\sqrt{t_{k}}$
\begin{displaymath}
  \big| u(0,t_{k}) -b_{k} \big | \leq \frac{c_{\N}}{2^{k}} +
  \frac{1}{2^{k}} + c_{\N} \sum_{j\geq k+1} 2^{-j}  \to 0, \qquad k \to
  \infty.\qedhere
\end{displaymath}
\end{proof}

The next result shows that the oscillatory behavior in Theorem
\ref{thr:oscillation_at_0} is somehow generic for heat solutions. For
this, given a sequence of positive numbers $\alpha=
\{\alpha_{k}\}_{k}$ denote $\mathscr{O}_{\alpha}$ the nonempty family
of $0\le u_{0} \in L^{1}_{0}(\R^\N)$ that satisfy the statement in Theorem
\ref{thr:oscillation_at_0}.

We use the topology on $L^1_0(\R^\N)$ generated by the family of
$L^1_\eps(\R^{\N})$ norms defined in \eqref{eq:norm_L1eps}, which
makes $L^1_0(\R^\N)$ into a Fr\'echet space (see \cite{RR2} for more details); the following more
explicit definition is sufficient for our statement of the following
theorem: we say that $u_n\to u_0$ in $L^1_0(\R^\N)$ if and only if
$u_n\to u_0$ in $L^1_\eps(\R^\N)$ for every $\eps>0$.  Note that, in
particular, such convergence implies that $u_n\to u_0$ in
$L^1_\loc(\R^\N)$.

\begin{theorem} \label{thr:density_of_oscillations}
 For any sequence of positive numbers $\alpha=
\{\alpha_{k}\}_{k}$, 
$\mathscr{O}_{\alpha}$ is dense in  $L^{1}_0(\R^\N)$.

\end{theorem}

\begin{proof}
  Denote by $U_{0}$ the initial data constructed in Theorem
  \ref{thr:oscillation_at_0}. Then $U_{0} \in \mathscr{O}_{\alpha}$
  and for any $n\in {\NN}$,
  $U_{0} \Chi_{\R^\N \setminus B(0,n)} \in \mathscr{O}_{\alpha}$ since we are
  only suppressing a finite number of annulae in $U_{0}$.

Then for given $0\leq v_{0}   \in L^{1}_{0}(\R^\N)$ define
\begin{displaymath}
  v_{0}^{n} = v_{0} \Chi_{B(0,n)} + U_{0} \Chi_{\R^\N \setminus
    B(0,n)} \in \mathscr{O}_{\alpha}.
\end{displaymath}
Take $\eps>0$; since
$$
v_0^{n}-v_0=(v_0-U_0)\Chi_{\R^\N\setminus B(0,n)}
$$
we have
$$
\|v_0^{n}-v_0\|_{L^1_\eps(\R^\N)}=\left(\frac{\eps}{\pi}\right)^{\N/2}\int_{|x|\ge
  n}\e^{-\eps|x|^2}|v_0-U_0|;
$$
since $v_0,U_0\in L^1_0(\R^\N) \subset L^1_\eps(\R^\N)$, it follows
that $v_0^{n}\to v_0$ in $L^1_0(\R^\N)$.
\end{proof}

The following result shows that  any heat solution can be ``shadowed''
as close as we want, in any large time interval
and any large compact set by an oscillatory solution of the heat
equation.

\begin{theorem} \label{thr:shadowing}
For any sequence of positive numbers $\alpha=
\{\alpha_{k}\}_{k}$  and any $0\leq v_{0}   \in L^{1}_{0}(\R^\N)$, any
$\delta>0$ and $T >0$ and any compact set $K\subset \R^\N$,
there exists   $u_{0}\in \mathscr{O}_{\alpha}$  such that
\begin{displaymath}
 \sup_{K\times [0,T]} |u(x,t,v_{0}) - u(x,t ,u_{0})| \leq \delta .
\end{displaymath}

\end{theorem}
\begin{proof}
Observe that it is enough to find  $u_{0}\in \mathscr{O}_{\alpha}$  such that
\begin{equation} \label{eq:shadow_at_0}
 \sup_{[0,T+1]}
 u(t,0,|v_{0} -u_{0}|) \leq \delta .
\end{equation}
In such a case, from (\ref{bound_above_x=0}) we would get, for any $a>1$
\begin{displaymath}
  \sup_{K\times [0,T]}  |u(x,t,v_{0}) - u(x,t,u_{0})| \le c_{\N,a}\,
   \sup_{K\times [0,T]}   u(0,at, |v_{0} -u_{0}|)  \e^{\frac{|x|^{2}}{4(a-1)t}} \leq C(K,T)
  \delta .
\end{displaymath}

  Denote by $U_{0} \in \mathscr{O}_{\alpha}$  the initial data constructed in Theorem
  \ref{thr:oscillation_at_0}. Then for given $0\leq v_{0}   \in L^{1}_{0}(\R^\N)$ define
\begin{displaymath}
  u_{0} = v_{0} \Chi_{B(0,R)} + U_{0} \Chi_{\R^\N \setminus
    B(0,R)} \in \mathscr{O}_{\alpha}.
\end{displaymath}
Then we show that  (\ref{eq:shadow_at_0}) holds provided we take $R$
large enough. For this, observe that
\begin{displaymath}
  |v_{0} -u_{0}| \leq |v_{0} -U_{0}|\Chi_{\R^\N \setminus B(0,R)}
\end{displaymath}
hence
\begin{equation} \label{eq:bound_at_0}
  0\leq  u(0,t,|v_{0} -u_{0}|) \leq \frac{1}{(4\pi t)^{\N/2}}
  \int_{|y|\geq R} \e^{-\frac{|y|^2}{4t}}
|v_0(y) -U_{0}(y)|  \,\d y .
\end{equation}

Taking $R>1$ we have, for any given $0<t_{0} < T+1$ and  $0<\alpha<1$,
$|y|^{2} \geq  \alpha
|y|^{2} + (1-\alpha)$ and then for $0<t <t_{0}$  we obtain in
(\ref{eq:bound_at_0})
\begin{displaymath}
  0\leq  u(0,t,|v_{0} -u_{0}|) \leq   \frac{1}{(4\pi t)^{\N/2}} \e^{-(1-\alpha) \frac{1}{4t}}
  \int_{|y|\geq 1} \e^{-\alpha \frac{|y|^2}{4t_{0}}} |v_0(y)
  -U_{0}(y)|  \,\d y  \leq \frac{\delta}{2}
\end{displaymath}
provided $t_{0}$ is small enough since $ \frac{1}{(4\pi t)^{\N/2}}
\e^{-(1-\alpha) \frac{1}{4t}} \to 0$ as $t\to 0$.

Now for  $t_{0} < t< T+1$  we obtain in
(\ref{eq:bound_at_0})
\begin{displaymath}
  0\leq  u(0,t,|v_{0} -u_{0}|) \leq \frac{1}{(4\pi t_{0})^{\N/2}}
  \int_{|y|\geq R} \e^{-\frac{|y|^2}{4(T+1)}}
|v_0(y) -U_{0}(y)|  \,\d y \leq  \frac{\delta}{2}
\end{displaymath}
provided $R$ is sufficiently large.
\end{proof}

\subsection{The rescaling approach of V\'azquez \& Zuazua}
\label{sec:rescalling-appr-vazq}
For the case of solutions of the heat equation that remain locally
bounded, the results in Propositions \ref{prop:estimates_from_u(0;t)}
and \ref{prop:oscillations} and Theorem \ref{thr:oscillation_at_0} can
be revisited in terms of the rescaling argument of \cite{VZ2002} as follows. As we now show, it is relatively straightforward to extend their approach from
$L^{\infty}(\R^{\N})$ initial data to more general measure-valued data that leads to globally bounded solutions.

\noindent (i) We can define  dilatations of measures through
the analogous result holding for a locally integrable $f$ and
$\varphi \in C_{c}(\R^\N)$, namely $f_{\lambda}(x)= f(\lambda x)$
satisfies
\begin{displaymath}
  \int_{\R^\N} f_{\lambda} (x) \varphi(x) \,\d x =
  \frac{1}{\lambda^{\N}}  \int_{\R^\N} f(z)
  \varphi (\frac{z}{\lambda}) \, \d z = \frac{1}{\lambda^{\N}}    \int_{\R^\N}  f (z) \varphi
  _{\frac{1}{\lambda}}  (z)\, \d z .
\end{displaymath}
That is, for $\lambda >0$ and $\mu \in \mathcal{M}_\loc (\R^\N)$
\begin{displaymath}
 \int_{\R^\N}   \varphi (z)\, \d \mu _{\lambda}  (z) :=
 \frac{1}{\lambda^{\N}}    \int_{\R^\N}
 \varphi _{\frac{1}{\lambda}}  (z)\, \d \mu(z) .
\end{displaymath}
Hence $\mu_{\lambda} \in \mathcal{M}_\loc (\R^\N)$ and
is  a positive measure whenever $\mu$ is. Then it follows that
\begin{displaymath}
 \int_{\R^\N}   \varphi (z)\, \d |\mu _{\lambda}  (z)| :=
 \frac{1}{\lambda^{\N}}    \int_{\R^\N}
 \varphi _{\frac{1}{\lambda}}  (z)\, \d |\mu(z)|  .
\end{displaymath}

These  extend, by density,  to $\varphi \in L^{1}(\d \mu) = L^{1}(\d
\mu_{\lambda})$.

\noindent (ii) For $\mu \in \mathcal{M}_{\eps} (\R^\N)$ and $\eps >0$,
$\lambda >0$
\begin{displaymath}
 \|\mu_{\lambda}\|_{\mathcal{M}_{\eps} (\R^\N)} = \left(\frac{\eps}{\pi}\right)^{\N/2} \int_{\R^\N}\e^{-\eps|x|^2}
  \,\d |\mu_{\lambda} (x)| =  \left(\frac{\eps}{\pi
      \lambda^{2}}\right)^{\N/2}
  \int_{\R^\N}\e^{-\frac{\eps}{\lambda^{2}} |y|^2}  \,\d |\mu (y)| =
  \|\mu\|_{\mathcal{M}_{\frac{\eps}{\lambda^{2}} } (\R^\N)} .
\end{displaymath}

Therefore, $\{\mu_{\lambda}\}_{\lambda >0}$ is bounded in
$\mathcal{M}_{\eps} (\R^\N)$ if and only if $\mu \in \mathcal{M}_{0,B}
(\R^{\N})$, that is,
\begin{displaymath}
 \Norm{\mu}_{ \mathcal{M}_{0,B} (\R^{\N})}:=   \sup_{\eps >0}  \| \mu \|_{ \mathcal{M}_{\eps} (\R^{\N})} < \infty .
\end{displaymath}
According to \eqref{eq:value_at_zero_and_normMeps} and part (i) in
Proposition \ref{prop:estimates_from_u(0;t)}, this is equivalent to
the solution of the heat equation  $u(x, t, \mu)$
  being  uniformly bounded in sets $\frac{|x|}{\sqrt{t}} \leq R$.

\noindent (iii) We also get for $u_{0} \in \mathcal{M}_{0,B} (\R^\N)$ and
$\lambda >0$
\begin{align*}
S(t) u_{0, \lambda} (x)&=   u(x,t, u_{0, \lambda}) = \frac{1}{(4\pi t)^{\N/2}} \int_{\R^\N} \e^{-\frac{|x-y|^2}{4t}}
\, \d u_{0,\lambda}(y) = \frac{1}{(4\pi t \lambda^{2})^{\N/2}}
\int_{\R^\N} \e^{-\frac{|x-\frac{y}{\lambda}|^2}{4t}} \, \d u_{0}(y)\\
&  = \frac{1}{(4\pi t \lambda^{2})^{\N/2}}
\int_{\R^\N} \e^{-\frac{|\lambda x-y|^2}{4t \lambda^{2}}} \, \d
u_{0}(y)  = u(\lambda x, \lambda^{2} t, u_{0})  = S(\lambda^{2} t)
u_{0} (\lambda x) .
\end{align*}

In particular, with $t=1$
\begin{displaymath}
  S(1)  u_{0, \lambda} (x) = S(\lambda^{2}) u_{0} (\lambda x)
\end{displaymath}

\noindent (iv) As a consequence of  Lemma \ref{lem:Meps_banach} it
follows that $\mathcal{M}_{\eps}(\R^{\N}) = (C_{-\eps,0}(\R^{\N}))'$,
where
$$
f\in C_{-\eps,0} (\R^{\N})\qquad\mbox{if and only if}\qquad \e^{\eps|x|^2}
f(x) \in  C_{0}(\R^{\N})
$$ 
and the norm  is $\|f\|_{C_{-\eps,0} (\R^\N)} :=   \sup_{x\in \R^\N}
\e^{\eps|x|^2} |f(x)|$.

Now, if $u_{0} \in \mathcal{M}_{0,B} (\R^\N)$ then
$\{u_{0,\lambda}\}_{\lambda>0}$ is sequentially weak-$*$ compact in
$\mathcal{M}_{\eps}(\R^{\N})$ for any $\eps >0$.
Taking subsequences $\lambda_{n} \to \infty$,   we can assume that
$u_{0,\lambda_{n}}$ converges weakly-$*$ to $\mu \in \mathcal{M}_{0,B}
(\R^\N)$ with $\Norm{\mu}_{ \mathcal{M}_{0,B} (\R^{\N})} \leq
\Norm{u_{0}}_{ \mathcal{M}_{0,B} (\R^{\N})}$ and, by smoothing,
$S(1)u_{0,\lambda_{n}} $ converges in  $L^{\infty}_\loc(\R^\N)$ to $v=
  S(1) \mu$.
Hence, setting $t_{n}= \lambda_{n}^{2}$ (so $t_n \to \infty$) we obtain
\begin{displaymath}
  S(t_{n}) u_{0} ( \sqrt{t_{n}} x) =  u(\sqrt{t_{n}} x, t_{n}, u_{0}) \to S(1)
    \mu (x) , \quad  L^{\infty}_\loc(\R^\N) .
\end{displaymath}
In particular, with $x=0$ we have
\begin{displaymath}
   u(0, t_{n},  u_{0}) \to S(1) \mu  (0) , \quad  L^{\infty}_\loc(\R^\N) .
\end{displaymath}
This relates the results in  Propositions \ref{prop:estimates_from_u(0;t)},
\ref{prop:oscillations} and Theorem \ref{thr:oscillation_at_0} to the
set of weak-$*$ sequential limits of
$\{u_{0,\lambda}\}_{\lambda>0}$ in $\mathcal{M}_{\eps}(\R^{\N})$.
Notice that for all such $\mu$
\begin{displaymath}
| S(1) \mu  (0)| = \left| \frac{1}{(4\pi)^{\N/2}}
\int_{\R^\N} \e^{-\frac{|y|^2} {4}} \, \d \mu (y) \right| \leq
\|\mu\|_{\mathcal{M}_{\frac{1}{4}}(\R^{\N})} \leq  \Norm{u_{0}}_{ \mathcal{M}_{0,B} (\R^{\N})};
\end{displaymath}
thus the above argument only applies when
$u(0, t,  u_{0})$  is bounded for large times.

\subsection{Prescribed behaviour at $x=0$}
\label{sec:prescribed_g_at_x=0}

We now show that if we drop the restriction that the solutions are
non-negative then any (sufficiently smooth) behaviour of the solution
at $x=0$ can be obtained with an appropriate choice of initial
condition. We use a construction inspired by the Tychonov example of
an initial condition that leads to non-uniqueness with zero initial
data (see \cite{Tychonov} and  Chapter 7, pages 171-172 in \cite{J}, for example).

\begin{proposition} \label{prop:prescribed_value_at_x=0}

Let $\gamma$ be any real analytic function on $[0,T)$ with $T\leq \infty$. Then there
exists $u_{0} \in L^{1}_\eps(\R^\N)$ for some $\eps>0$  such that $u(t)
= S(t) u_{0}$ given by \eqref{eq:solution_heat_up_t=0} is defined for
all  $t \in [0,T)$ and such that
\begin{displaymath}
  u(0,t)= \gamma(t), \quad 0\leq t < T .
\end{displaymath}

\end{proposition}
\begin{proof}

We seek a
solution of the one-dimensional heat equation in
the form
$$
u(x,t) = \sum_{k=0}^{\infty} g_{k}(t) x^{k}, \quad x\in \R
$$
converging for every $x\in\R$  (for each $t$ in some range), such that
$u(0,t)=\gamma(t)$ for all  $t \in [0,T)$. Substituting this expression into
the PDE gives
\begin{displaymath}
  g_{k}'(t) = (k+2)(k+1) g_{k+2}(t) , \quad t \in [0,T).
\end{displaymath}
Assume that $u_{x}(0,t) =0$; then $g_{1}(t)=0$ and so $g_{2m+1}(t)=0$ for
all $t$ and $m\in \NN$. Solving the recurrence for even powers gives
$g_{0}(t) = \gamma(t)$ and
\begin{displaymath}
  g_{2m}(t)= \frac{\gamma^{(m)}(t)}{(2m)!}
\end{displaymath}
and therefore
\begin{equation}\label{g2u}
  u(x,t) = \sum_{k=0}^{\infty} \frac{\gamma^{(k)}(t)}{(2k)!}   x^{2k} .
\end{equation}

As $\gamma$ is real analytic it follows that for each $t \in [0,T)$
there exist constants $C,\tau>0$ (potentially depending on $t$) such
that
\begin{equation}\label{RAgk}
|\gamma^{(k)}(t)|\le Ck!\tau^{-k}
\end{equation}
(see Exercise 15.3 in \cite{R}, for example; in fact the
constants $C$ and $\tau$ can be chosen uniformly on any compact
subinterval of  $[0,T)$). It follows that the series in \eqref{g2u}
converges for all $x\in\R^\N$ and for every  $t \in [0,T)$, and given
(\ref{RAgk}) we have
then
$$
|u(x,t)|\le \sum_{k=0}^\infty C \frac{k!}{(2k)!} x^{2k}\tau^{-k}\le
C\sum_{k=0}\frac{(x/\sqrt\tau)^2}{k!} =  C\e^{|x|^2/\tau} .
$$
In particular $u$ satisfies
(\ref{eq:quadratic_exponential_bound_upt=0}) on any compact time
interval of  $[0,T)$. Also it is easy to see that $u(x,t)$ satisfies
(\ref{eq:uniqueness_Du_in_L1delta}) in compact intervals of $(0,T)$ and
(\ref{eq:uniqueness_initial_data_Cc}).
By Theorem \ref{thr:uniqueness}, we get
\begin{displaymath}
  u(t) = S(t) u_{0} , \quad  t \in [0,T).
\end{displaymath}

We can embed this solution   in $\R^\N$  by
setting $u(x_{1}, \ldots, x_{\N}, t) = u(x_{1},t)$.
\end{proof}

When $T=\infty$ this provides an example showing how the condition
$u_0\in L^1_0(\R^\N)$ is not required to ensure global existence for
initial data that is not required to be non-negative: 
the solution
$u(x,t)$ satisfies the heat equation for all time, remains in one of
the $L^1_{\eps(t)}(\R^{\N})$ spaces for each  $t\ge0$, but does not necessarily
satisfy $u_0\in L^1_0(\R)$.

The non-uniqueness example of Tychonov uses precisely the above
construction, but based on a function such as $\gamma(t)=\e^{-1/t^2}$
whose radius of analyticity shrinks as $t\to 0^+$, see \cite[pg 172]{J}. For such a case, we
have that the heat solution in (\ref{g2u}) satisfies $u(x,t)\to
u_{0}(x)= 0$
uniformly in compact sets as
$t\to 0^{+}$.
In the language of
this paper, $u(t)\in L^1_{\eps(t)}(\R^{\N})$ for every $t>0$, but as $t\to 0^{+}$
we have $\eps(t)\to\infty$ and
(\ref{eq:uniqueness_u_in_L1eps}) is not satisfied.
In this way this classic non-uniqueness example does not
contradict the uniqueness result of Theorem \ref{thr:uniqueness}.

It would be interesting to find conditions on $\gamma(t)$ that ensure the
positivity of $u_0$ (and hence of $u(x,t)$). Certainly positivity of
$\gamma$ itself is not sufficient; indeed, note that 
$$
\gamma(t)=\sum_{k=0}^\infty
\frac{\alpha_k}{k!}t^k\qquad\Rightarrow\qquad u_0(x)=\sum_{k=0}^\infty
\frac{\alpha_k}{(2k)!}x^{2k}.
$$
The simple choice $\alpha_0=1$, $\alpha_1=-2$, $\alpha_2=2$ yields
$$
\gamma(t)=1-2t+t^2\ge0\qquad\mbox{but}\qquad u_0(x)=1-x+\frac{x^2}{24}
$$
and $u_0(2)<0$.

\section{Extension to other problems}
\label{sec:extensions-other-probl}
\setcounter{equation}{0}

First note that by simple reflection arguments, we can also consider
the heat equation in  the half space $\R^\N_{+}$, that is
\begin{equation} \label{eq:heat_in_halfspace}
 \begin{cases}
u_{t}- \Delta u  =0,  & x \in \R^\N_{+} , \ t>0,\cr
u(x,0) = u_{0}(x),  &  x \in \R^\N_{+},\cr
 B(u)(x)=0,  & x \in \partial \R^\N_{+}
\end{cases}
\end{equation}
where  $B(u)$ denotes boundary conditions of Dirichlet type, i.e.
$B(u)=u$ or  Neumann, i.e.  $B(u)=  \frac{\partial u}{\partial
 \vec{n}}  = -\partial_{x_{\N}} u (x',0)$. Indeed, performing odd or
even reflection respectively we extend (\ref{eq:heat_in_halfspace}) to
the heat equation in $\R^{\N}$ for solutions with odd or even
symmetry. Hence, the arguments in previous sections  apply.

Also note  that a basic ingredient in the proofs above  is the
gaussian structure  of the heat kernel. Hence, the same results apply
to any   parabolic operator with a similar gaussian bound for the
kernel, see  \cite{Daners00}. In particular,  our results apply for
differential operators of the form
\begin{displaymath} 
L(u) = -\sum_{i=1}^{N} \partial_{i} \Big( a_{i,j}(x) \partial_{j} u +
a_{i}(x) u\Big) + b_{i}(x) \partial_{i} u + c_{0}(x) u
\end{displaymath}
with real coefficients $a_{i,j}, a_{i}, b_{i}, c_{0} \in
L^{\infty}(\R^{\N})$ and satisfies the ellipticity condition
\begin{displaymath} 
\sum_{i,j=1}^{N} a_{i,j}(x) \xi_{i} \xi_{j} \geq \alpha_{0}
|\xi|^{2}
\end{displaymath}
for some $\alpha_{0} >0$ and for every $\xi \in \R^{\N}$.
In such a case   the fundamental solution of the parabolic problem $u_{t} + Lu=0$
in $\R^{\N}$ satisfies a gaussian bound
$$
0 \leq k(x,y,t,s) \leq C(t-s)^{-\N/2} \e^{\omega (t-s)}
\e^{-c\frac{|x-y|^2}{(t-s)}}
$$
for $t>s$ and $x,y \in \R^{\N}$ where $C,c, \omega$ depend on the
$L^{\infty}$ norm of the coefficients.
The gaussian bounds are obtained from \cite{Daners00} while the positivity
of the kernel comes from the maximum principle, see \cite{GT}, chapter
8. Therefore the analysis of previous sections, applies to solutions
of the form
\begin{displaymath}
u(x,t)=S_{L}(t)u_{0} = \int_{\R^\N} k(x,y,t,0) \, \d u_0(y)  .
\end{displaymath}
Other results on Gaussian upper bounds can be found in   \cite{Aro68, Davies,Ouhabaz,Stroock}.

\appendix

\section{Some auxiliary results}
\label{sec:some-auxil-results}
\setcounter{equation}{0}

Here we prove several technical results used above. First, we prove
that certain spaces of functions or measures used above are Banach
spaces.

\begin{lemma}\label{lem:Meps_banach}

 The sets  $\mathcal{M}_{\eps}(\R^{\N})$ and $L^{1}_{\eps} (\R^\N)$ in
 \eqref{eq:space_Meps} and \eqref{eq:space_L1eps} with
    the norms \eqref{eq:norm_Meps} and
    \eqref{eq:norm_L1eps} respectively, are Banach spaces.

\end{lemma}
\begin{proof}
For  $\mathcal{M}_{\eps}(\R^{\N})$  we proceed as follows.  Given $\mu
\in \mathcal{M}_{\loc}(\R^{\N})$ we define the Borel
measure such that for all Borel sets $A \subset \R^{\N}$
\begin{displaymath}
  \Phi_{\eps}(\mu) (A) = \int_{A} \rho_{\eps}(x)\, \d \mu (x)
\end{displaymath}
where $\rho_{\eps}(x) = \left(\frac{\eps}{\pi}\right)^{\N/2}
\e^{-\eps|x|^2}$.
Then $\Phi_{\eps}(\mu) \in \mathcal{M}_{loc}(\R^{\N})$, is clearly
absolutely continuous with respect to $\mu$ and for all $\varphi \in
C_{c}(\R^{\N})$ we have
\begin{displaymath}
   \int_{\R^{\N}} \varphi (x)\, \d \Phi_{\eps}(\mu) (x)  =
   \int_{\R^{\N}} \varphi (x) \rho_{\eps}(x)\, \d \mu (x) .
\end{displaymath}

Now we claim that the total variation of $\Phi_{\eps}(\mu)$ satisfies
\begin{displaymath}
   |\Phi_{\eps}(\mu)| (A) = \int_{A} \rho_{\eps}(x)\, \d |\mu (x) |
\end{displaymath}
 for all Borel sets $A \subset \R^{\N}$, which would imply that  for all $\varphi \in
C_{c}(\R^{\N})$ we have
\begin{displaymath}
   \int_{\R^{\N}} \varphi (x)\, \d |\Phi_{\eps}(\mu) (x)|  =
   \int_{\R^{\N}} \varphi (x) \rho_{\eps}(x)\, \d |\mu (x) |.
\end{displaymath}
To prove the claim, observe that the positive part in the Jordan
decomposition satisfies
\begin{displaymath}
   \Phi_{\eps}(\mu)^{+} (A) = \sup_{B\subset A} \int_{B}
   \rho_{\eps}(x)\, \d \mu (x)  = \sup_{B\subset A} \int_{B^{+}}
   \rho_{\eps}(x)\, \d \mu^{+} (x) - \int_{B^{-}}
   \rho_{\eps}(x)\, \d \mu^{-} (x)
\end{displaymath}
\begin{displaymath}
=  \sup_{B\subset A} \int_{B^{+}}   \rho_{\eps}(x)\, \d \mu^{+} (x)  =  \sup_{B\subset A} \int_{B}
   \rho_{\eps}(x)\, \d \mu^{+} (x) =   \int_{A}
   \rho_{\eps}(x)\, \d \mu^{+} (x)
\end{displaymath}
where we have used that $B= B^{+} \cup B^{-}$ are the positive and
negative parts of a set $B$, according to the Hahn decomposition of
the measure $\mu$; see Theorem 3.3 in \cite{Folland}.
Analogously, $ \Phi_{\eps}(\mu)^{-} (A) = \int_{A}
   \rho_{\eps}(x) \d \mu^{-} (x)$ and with this the claim follows.

Now if, $\mu \in \mathcal{M}_{\eps}(\R^{\N})$ we have
\begin{displaymath}
   |\Phi_{\eps}(\mu)| (\R^{\N}) = \int_{\R^{\N}} \rho_{\eps}(x)\, \d
   |\mu (x) | < \infty
\end{displaymath}
that is
\begin{displaymath}
  \Phi_{\eps}: \mathcal{M}_{\eps} (\R^{\N}) \to \mathcal{M}_{\BTV}(\R^{\N})
\end{displaymath}
is an isometry. To prove the result it remains to show that
$\Phi_{\eps}$ is onto. In fact if $\sigma \in
\mathcal{M}_{\BTV}(\R^{\N})$ we define $\mu$ such that for all borel sets $A \subset \R^{\N}$
\begin{displaymath}
  \mu (A) = \int_{A} \frac{\d \sigma (x)}{\rho_{\eps}(x)}
\end{displaymath}
and thus  for all $\varphi \in
C_{c}(\R^{\N})$ we have
\begin{displaymath}
   \int_{\R^{\N}} \varphi (x)\, \d \mu (x)  =
   \int_{\R^{\N}} \frac{\varphi (x)}{ \rho_{\eps}(x)}\, \d \sigma (x) .
\end{displaymath}
Clearly $\mu \in \mathcal{M}_{loc} (\R^{\N})$ and arguing as above we
get for all borel sets $A \subset \R^{\N}$
\begin{displaymath}
 |\mu| (A) = \int_{A} \frac{\d |\sigma (x)|}{\rho_{\eps}(x)}
\end{displaymath}
and   for all $\varphi \in
C_{c}(\R^{\N})$ we have
\begin{displaymath}
   \int_{\R^{\N}} \varphi (x)\, \d |\mu (x)|  =
   \int_{\R^{\N}} \frac{\varphi (x)}{ \rho_{\eps}(x)}\, \d |\sigma (x)| .
\end{displaymath}

Now take an increasing sequence $0\leq \varphi_{n} \in C_{c}(\R^{\N})$
such that $\varphi_{n} \to \rho_{\eps}$ pointwise in $\R^{\N}$ and
then
\begin{displaymath}
   \int_{\R^{\N}} \varphi_{n} (x)\, \d |\mu (x)|  =
   \int_{\R^{\N}} \frac{\varphi_{n} (x)}{ \rho_{\eps}(x)}\, \d |\sigma
   (x)|  \leq    \int_{\R^{\N}}\, \d |\sigma
   (x)| = |\sigma (\R^{\N})| <\infty
\end{displaymath}
and by Fatou's lemma we get $ \int_{\R^{\N}} \rho_{\eps} (x)\, \d |\mu
(x)| <\infty$,  that is, $\mu \in \mathcal{M}_{\eps}(\R^{\N})$. Clearly
$\Phi_{\eps}(\mu) = \sigma$ and we conclude the proof.

On the other hand, note  that $L^{1}_{\eps}
(\R^\N) = L^{1} (\R^\N, \rho_{\eps} \, \d x)$ and so is a Banach
space. Also, note that along the lines of the proof above it is easy
to see that  the  operator $\Phi_{\eps}(f) = \rho_{\eps} f$, $\Phi _{\eps}:
L^{1}_{\eps}(\R^{\N}) \to L^{1}(\R^{\N})$ is an isometric
isomorphism.
\end{proof}

\begin{lemma} \label{lem:MU_banach}

The space of uniform measures
$\mathcal{M}_U(\R^\N)$ defined in \eqref{M_U} with the norm
\eqref{eq:M_U_norm} is a Banach space.

\end{lemma}
\begin{proof}
  Clearly a Cauchy sequence in the norm (\ref{eq:M_U_norm}) is a
  Cauchy sequence in $\mathcal{M}_{\BTV}(\overline{B(x,1)})$ uniformly
  for  $x\in  \R^{\N}$, that is, the dual of $C(\overline{B(x,1)})$ with the
  uniform convergence. Therefore, it converges in $\mathcal{M}_{\BTV}(\overline{B(x,1)})$ uniformly
  for  $x\in  \R^{\N}$. Hence it converges in $\mathcal{M}_U(\R^\N)$.
\end{proof}

\begin{lemma} {\bf (Green's formulae)}
\label{lem:green_formula}

\noindent (i) Assume that $u\in W^{1,1}_{\loc}(\R^{\N})$ satisfies $u,
\nabla u \in L^{1}_{\eps}(\R^{\N})$ and that $\xi$ is a smooth function
such that $|\nabla \xi (x) |, |\Delta \xi (x) |\leq c \e^{-\alpha |x|^{2}}$ for $x\in
\R^{\N}$ and $\alpha \geq  \eps$. Then
\begin{displaymath}
  \int_{\R^{\N}} u (-\Delta \xi) = \int_{\R^{\N}} \nabla u \nabla \xi.
\end{displaymath}

\noindent (ii) Assume that $u\in W^{1,1}_{\loc}(\R^{\N})$ satisfies $\Delta
u\in L^{1}_{\loc}(\R^{\N})$ and  $\nabla u, \Delta u  \in L^{1}_{\eps}(\R^{\N})$ and that $\xi$ is a smooth function 
such that $| \xi (x) |, |\nabla \xi (x) |\leq c \e^{-\alpha |x|^{2}}$ for $x\in
\R^{\N}$ and $\alpha \geq \eps$. Then
\begin{displaymath}
\int_{\R^{\N}} \nabla u \nabla \xi =   \int_{\R^{\N}}  (-\Delta u) \xi.
\end{displaymath}

\end{lemma}
\begin{proof}
\noindent (i) Observe that for any $R>0$
\begin{displaymath}
  \int_{B(0,R)} u (-\Delta \xi) = \int_{B(0,R)} \nabla u \nabla \xi -
  \int_{\partial B(0,R)} u \frac{\partial \xi}{\partial \vec{n}} \, \d
  S.
\end{displaymath}
Thanks to the Dominated Convergence Theorem it is enough to prove that
the last term above converges to zero, as $R\to \infty$, since we can
write
\begin{displaymath}
  \int_{B(0,R)} u (-\Delta \xi) = \int_{B(0,R)} \e^{-\eps|x|^{2}} u \,
  \e^{\eps|x|^{2}} (-\Delta \xi)
\end{displaymath}
and
\begin{displaymath}
   \int_{B(0,R)} \nabla u \nabla \xi =  \int_{B(0,R)} \e^{-\eps|x|^{2}}
   \nabla u\,  \e^{\eps|x|^{2}}\nabla \xi
\end{displaymath}
and pass to the limit in $R\to \infty$ in both terms.

For this
observe that
\begin{displaymath}
  \int_{0}^{\infty}  \int_{\partial B(0,R)} |u \frac{\partial
    \xi}{\partial \vec{n}} |  \, \d S \, \d R  \leq \int_{\R^{\N}} |u| | |\nabla \xi| <
  \infty.
\end{displaymath}
Hence for some sequence $R_{n} \to \infty$ we have $\displaystyle
\int_{\partial B(0,R_{n})} |u \frac{\partial    \xi}{\partial \vec{n}} | \to
0$.

Part  (ii) is obtained  in a similar fashion.
\end{proof}



\addcontentsline{toc}{section}{References}


\begin{thebibliography}{}

\bibitem{Aro68} D.G. Aronson, Non-negative solutions of linear
  parabolic equations. Ann. Scuola Norm. Sup. Pisa 22, 607–694
  (1968).

\bibitem{Aro71} D.G. Aronson, Non-negative solutions of linear
  parabolic equations: an addendum. Ann. Scuola Norm. Sup. Pisa 25, 221--228
  (1971).


\bibitem{ACDRB2004} J.M. Arrieta, J.Cholewa, T. Dlotko,
  A. Rodriguez--Bernal, Linear parabolic equations in locally uniform
  spaces. Math. Models Methods Appl. Sci. 14, 253 (2004).





\bibitem{BeneCzaja} J. Benedetto, W.Czaja, {\it  Integration and Modern
  Analysis}, Birkh\"auser Basel, 2009.



\bibitem{CazenaveDW} T. Cazenave, F. Dickstein, F. Weissler, Universal solutions of a
nonlinear heat equation on $\R^{N}$, Ann. Scuola Norm Sup. Pisa
Cl. Sci 5, vol II, 77-117 (2003).


\bibitem{ColletEckman} P. Collet,  J. P. Eckmann, Space-time behaviour in problems of
hydrodynamic type: a case study, Nonlinearity 5, 1265--1302  (1992).




\bibitem{Daners00} D. Daners, Heat kernel estimates for operators with
  boundary conditions, Math.  Nachr. 217 13–41 (2000).


\bibitem{Davies} E. B. Davies, {\it Heat kernels and spectral theory},
   Cambridge University Press, 1989.


\bibitem{DuoanZuazua} J. Duoandikoetxea,  E. Zuazua, Moments, masses de Dirac et
d'ecomposition de fonctions, C. R. Acad. Sci. Paris Serie. I
Math. 315, 693--698 (1992).


\bibitem{Folland} G.B. Folland, {\it Real analysis. Modern techniques and their
  applications}, Wiley 1990.



\bibitem{GigaGigaSaal} M.H. Giga, Y. Giga, J. Saal, {\it Nonlinear Partial
Differential Equations. Asymptotic Behavior
of Solutions and self--similar solutions}, Progress in Nonlinear Differential Equations and Their
Applications 79, Birkhauser, 2010.


\bibitem{FritzRoy} P. Fitzpatrick, H. Royden, {\it Real Analysis} (4th
  Edition), Prentice Hall, 2010

\bibitem{GT} D. Gilbarg, N. S. Trudinger, {\it Elliptic partial
    differential equations of second order}, Springer, 1998.

\bibitem{GiaqModicaSoucek} M. Giaquinta, G.  Modica, J.  Souček, {\it
    Cartesian Currents in the Calculus of Variations I}, Springer,
  1998.


\bibitem{HKMM96} M. Hieber, P. Koch-Medina, and S. Merino, Linear and
  semilinear parabolic equations on $BUC(\R^{N})$, Math. Nachr.,
  vol. 179, pp. 107--118 (1996).

\bibitem{J} F. John, {\it Partial differential equations}, 3rd Edition,
  Springer--Verlag 1991.

\bibitem{L} A. Lunardi,  {\it Analytic semigroups and optimal regularity
  in parabolic problems}. Basel: Birkhäuser Verlag, 1995

\bibitem{Mora} X. Mora, Semilinear parabolic problems define semiflows
  on $C^{k}$ spaces, Trans. Amer. Math. Soc., vol. 278, no. 1, pp. 21--55 (1983).


\bibitem{Ouhabaz} E. M. Ouhabaz, {\it Analysis of heat equations on
  domains}, London Math. Soc. Monographs, vol. 31, Princeton
  University Press 2004.

\bibitem{QuittnerSouplet}  P. Quittner and P. Souplet. {\it Superlinear parabolic
problems.  Blow-up, global existence and steady states}. Birkha\"user
Advanced Texts: Basler Lehrb\"cher. Birkha\"user Verlag, Basel,
2007.

\bibitem{PolacikYanagida} P.Pol\'acik, E. Yanagida, On bounded and unbounded global solutions
of a  supercritical semilinear heat equation, Math. Ann 327, 745--771 (2003).


 \bibitem{R} J.C. Robinson, \textit{Dimensions, Embeddings, and Attractors}, Cambridge Tracts in Mathematics 186. Cambridge University Press, Cambridge, 2011.

 \bibitem{RR2} J.C. Robinson \& A. Rodriguez Bernal The heat
flow in a Frechet space of unbounded initial
data and applications to elliptic equations in $\R^d$, (2018).

\bibitem{S} B. Simon, {\it Convexity. An Analytic Viewpoint}, Cambridge
  Tracts in Mathematics 187.  Cambridge University Press, 2011
.

\bibitem{Stroock} D. W. Stroock, {\it Partial differential equations
    for probabilists}, Cambridge Studies in Advanced Mathematics,
  112. Cambridge University Press, Cambridge, 2008.


\bibitem{Tychonov} A. Tychonoff, Th\'eor\`emes d'unicit\'e  pour
  l'\'equation de la chaleur, Mat. Sb., Volume 42, Number 2, 199--216
  (1935).

\bibitem{VZ2002} J.L.V\'azquez, E. Zuazua, Complexity of large time
  behaviour of evolution equations with bounded data Chinese Annals of
  Mathematics, 23, ser. B, 2, 293-310 (2002). Special issue in honor
  of J.L. Lions.





\end{thebibliography}
\end{document}